\documentclass[12pt,reqno]{amsart}

\usepackage{amsfonts}
\usepackage{amscd}
\usepackage{amssymb}
\usepackage{amsfonts}
\usepackage{amscd}
\usepackage{amsmath} 
\usepackage{mathrsfs}
\usepackage{dsfont}
\usepackage{enumerate}
\usepackage{float} % To use [H]
\usepackage[pdftex]{graphicx}
\allowdisplaybreaks
\usepackage{url}
\usepackage{soul}
\usepackage[hidelinks]{hyperref} 
\usepackage{subcaption}
\usepackage{tikz}
\usepackage{caption}
\usepackage{subcaption}
\usetikzlibrary{arrows}
\hypersetup{
     colorlinks,
     linkcolor={red!70!black},
     citecolor={blue!80!black},
     urlcolor={blue!60!black}
 }

\date{\today}

\newtheorem*{theorem*}{Theorem}
\newtheorem{theorem}{Theorem}[section]
\newtheorem{corollary}[theorem]{\bf{Corollary}}
\newtheorem{lemma}[theorem]{Lemma}

\theoremstyle{definition}

\theoremstyle{remark}
\newtheorem{remark}[theorem]{\bf{Remark}}

\numberwithin{equation}{section}

\newcommand{\beas}{\begin{eqnarray*}}
\newcommand{\eeas}{\end{eqnarray*}}
\newcommand{\bes} {\begin{equation*}}
\newcommand{\ees} {\end{equation*}}
\newcommand{\be} {\begin{equation}}
\newcommand{\ee} {\end{equation}}
\newcommand{\bea} {\begin{eqnarray}}
\newcommand{\eea} {\end{eqnarray}}

\newcommand{\R}{\mathbb R}
\newcommand{\C}{\mathbb C}
\newcommand{\Z}{\mathbb Z}%
\newcommand{\N}{\mathbb N}

\newcommand{\X}{\mathbb{X}}

\newcommand{\fa}{\mathfrak{a}}
\newcommand{\bfc}{\mathbf{c}}

\newcommand{\Chi}{\mbox{\large$\chi$} }

\renewcommand{\Im}{\text{Im}}
\renewcommand{\Re}{\text{Re}}

\renewcommand{\Re}{\operatorname{Re}}
\renewcommand{\Im}{\operatorname{Im}}
 
\usepackage{color}

\title[Shifted Pitt and uncertainty inequalities  on $\mathbb{X}$]{Shifted Pitt and uncertainty inequalities on Riemannian symmetric spaces of noncompact type}

\author[Rana, Ruzhansky]{Tapendu Rana, Michael Ruzhansky}	
\address{Tapendu Rana  \endgraf Department of Mathematics: Analysis, Logic and Discrete Mathematics,	\endgraf Ghent University, 	\endgraf Krijgslaan 281, Building S8, B 9000 Ghent, Belgium.} \email{tapendurana@gmail.com, tapendu.rana@ugent.be}
\address{Michael Ruzhansky \endgraf Department of Mathematics: Analysis, Logic and Discrete Mathematics,	\endgraf Ghent University, 	\endgraf Krijgslaan 281, Building S8, B 9000 Ghent, Belgium. 
\endgraf and
\endgraf School of Mathematical Sciences
\endgraf Queen Mary University of London
\endgraf United Kingdom}
\email{michael.ruzhansky@ugent.be}

\date{}

\begin{document}
\begin{abstract}
Our primary objective is to study Pitt-type inequalities on Riemannian symmetric spaces $\mathbb{X}$ of noncompact type, as well as within the framework of Jacobi analysis.  Inspired by the spectral gap of the Laplacian on $\mathbb{X}$, we introduce the notion of a \textit{shifted} Pitt's inequality as a natural and intrinsic analogue tailored to symmetric spaces, capturing key aspects of the underlying non-Euclidean geometry. In the rank one case (in particular, for hyperbolic spaces), we show that the sufficient condition for the \textit{shifted} Pitt's inequality matches the necessary condition in the range $p \leq q \leq p'$, yielding a sharp characterization of admissible polynomial weights with non-negative exponents.

In the Jacobi setting, we modify the transform so that the associated measure exhibits polynomial volume growth. This modification enables us to fully characterize the class of polynomial weights with non-negative exponents for which Pitt-type inequalities hold for the modified Jacobi transforms. As applications of the \textit{shifted} Pitt's inequalities, we derive $L^2$-type Heisenberg-Pauli-Weyl uncertainty inequalities and further establish generalized $L^p$ versions. Moreover, the geometric structure of symmetric spaces allows us to formulate a broader version of the uncertainty inequalities previously obtained by Ciatti-Cowling-Ricci in the setting of stratified Lie groups.
\end{abstract}
\maketitle
\tableofcontents
\addtocontents{toc}{\setcounter{tocdepth}{2}}
\section{Introduction}
Pitt's inequality is a classical result in harmonic analysis that quantifies how the weighted integrability of a function controls the decay of its Fourier transform. In 1937, H.R. Pitt  \cite{Pit37} established a weighted $L^p$–$L^q$ inequality for Fourier coefficients of integrable functions on the unit circle. Over the following decades, this result was extended and refined in the Euclidean setting, most notably by Zygmund-Stein \cite{Ste56} and Benedetto–Heinig \cite{BH03}. In the Euclidean case, Pitt's inequality typically takes the form 
\begin{equation}\label{Clss_Pitt_Rn_pw}
 \left( \int_{\R^N} |\mathcal{F}(f)(\xi)|^q |\xi|^{-q\sigma}  \, d\xi \right)^{\frac{1}{q}}\leq C \left( \int_{\R^N} |f(x)|^{p}  |x|^{p \kappa} dx \right)^{\frac{1}{p}},
\end{equation}
where $\mathcal{F}$ denotes the Fourier transform operator  on $\R^N$. It is known that \eqref{Clss_Pitt_Rn_pw} holds for all $f \in C_c^{\infty}(\R^N)$ and for exponents $1<p\leq q<\infty$, if and only if 
    \begin{align}
        &\max \left\{0, N\left(\frac{1}{p}+\frac{1}{q} -1\right)  \right\}\leq \sigma <\frac{N}{q}, \notag  \\
         \label{k-s=N_equiv}
 \hspace{-4.45cm}\text{ and the \textit{balance condition:}}\qquad & \hspace{2cm}\sigma -\kappa  = N\left(\frac{1}{p}+\frac{1}{q}-1 \right).
    \end{align}
Pitt's inequality unifies and generalizes several fundamental inequalities in harmonic analysis, including the Hausdorff–Young and Hardy–Littlewood–Paley inequalities. Moreover, it provides a natural method to measure uncertainty \cite{Bec95, Bec08, GIT16}: it captures how the spatial decay of a function constrains the integrability of its spectral content.  This perspective highlights the role of Pitt-type inequalities in probing the limits of weighted Fourier inequalities and their dependence on geometric and analytic structures.

In parallel, the uncertainty principle has emerged as a central theme in both harmonic analysis and mathematical physics.  Broadly speaking, it expresses the idea that a nontrivial function and its Fourier transform cannot both be sharply localized. This principle was famously articulated by  Heisenberg \cite{Hei27} in 1927. More precise and rigorous mathematical formulations were later developed by Kennard \cite{Ken27} and Weyl \cite{Wey27}, while Pauli’s early insights have since been recognized as foundational. The mathematical inequality that formalizes this principle is commonly referred to as the Heisenberg–Pauli–Weyl uncertainty inequality. In its most general form, it states that for every $\gamma, \delta > 0$ and $f \in L^2(\R^N)$,
\begin{align}\label{eqn_uncn_L2_Rn}
    \|f\|_{L^2(\R^N)}^{\gamma+\delta}\leq C_{\gamma, \delta} \left( \int_{\R^N} |x|^{2 \gamma} |f(x)|^2 \, dx\right)^{\frac{\delta}{2}} \left( \int_{\R^N} |\xi|^{2 \delta} |\mathcal{F} f(\xi)|^2 \, d\xi\right)^{\frac{\gamma}{2}}.
\end{align}
We refer to the beautiful survey  \cite{FS97} for an overview of the history, as well as
to its generalizations allowing other $L^p$-norms and different powers of $|x|$ and $|\xi|$ and extensions of these results in different settings. 

These inequalities have been extensively studied in the Euclidean setting and subsequently extended to a variety of contexts. Extensions of the classical Pitt's inequality \eqref{Clss_Pitt_Rn_pw} to different types of Fourier transforms on $\R^N$ have been investigated in \cite{DC08, Joh16, GLT18}. Notably, Beckner \cite{Bec95, Bec12} established a sharp form of Pitt's inequality for the case $p=q=2$, which was later extended to the setting of generalized Fourier transforms in \cite{GIT16}. The problem of generalizing Pitt-type inequalities to broader classes of weights on $\R^N$ has been studied by several authors, including \cite{Muc83, JS84, He84}. More recent advances on general weighted Fourier inequalities in the Euclidean setting appear in \cite{LT12, DCGT17, Deb20, ST25}, while extensions to certain non-Euclidean contexts can be found in \cite{GLT18, KPRS24}.

 Uncertainty inequalities have likewise been generalized in numerous directions and extended to a wide range of settings, including Lie groups and manifolds. To place these results in context, we recall that, by the Plancherel formula, the uncertainty inequality \eqref{eqn_uncn_L2_Rn} on $\R^N$ is equivalent to
 \begin{align}\label{eqn_uncn_L_RN}
    \|f\|_{L^2(\R^N)}^{\gamma+\delta}\leq C_{\gamma, \delta} \| |x|^{\gamma} f\|_{L^2(\R^N)}^{\delta} \| (-\mathcal{L}_{\R^N})^{\frac{\delta}{2}}f \|_{L^2(\R^N)}^{\gamma},
\end{align}
where $\mathcal{L}_{\R^N}$ denotes the Laplacian on $\R^N$ and the operator  $ (-\mathcal{L}_{\R^N})^{\frac{\delta}{2}}$ is defined by 
\begin{align}\label{eqn_symb_Frac_L}
     (-\mathcal{L}_{\R^N})^{\frac{\delta}{2}}f(\xi) =|\xi|^{\delta}\mathcal{F}f(\xi).
\end{align}
This formulation of the Heisenberg–Pauli–Weyl (HPW) inequality is particularly amenable to generalization, as it allows for the Laplacian to be replaced by a positive self-adjoint operator and the Euclidean norm $|x|$ to be replaced by an appropriate distance function in other geometric settings.

Over the years, uncertainty principles have developed into a rich and influential area of research in mathematics, generating an extensive and diverse body of work. In what follows, we highlight a selection of key contributions that are particularly relevant and motivating for the problems addressed in this article. In 1995, Beckner \cite{Bec95} obtained a logarithmic version of the uncertainty principle \eqref{eqn_uncn_L2_Rn} by employing a sharp form of Pitt's inequality. Early extensions of the HPW inequality of the form \eqref{eqn_uncn_L_RN} were developed by Thangavelu \cite{Tha90}, who replaced the Laplacian with the Hermite operator and the sub-Laplacian on Heisenberg groups, obtaining results for the case $\gamma = \delta = 1$. These results were later extended to arbitrary $\gamma, \delta > 0$ by Sitaram–Sundari–Thangavelu \cite{SST95} and Xiao–He \cite{XH12}, thereby broadening the applicability of uncertainty inequalities to non-Euclidean settings. Following the direction in \cite{Tha90}, Ciatti–Ricci–Sundari \cite{CRS05, CRS07} extended the framework further to Lie groups with polynomial volume growth. They replaced the Laplacian with a Hörmander-type sub-Laplacian and established an exact analogue of the $L^2$ HPW inequality given in \eqref{eqn_uncn_L_RN}. In a subsequent development, Ciatti-Cowling-Ricci \cite{CCR15} generalized this result by replacing the $L^2$ norms with $L^p$ norms in the setting of stratified Lie groups. They also extended Beckner’s logarithmic uncertainty inequality \cite{Bec95} to this broader context. Specifically, they proved the following.
\begin{theorem}[{{\cite[Theorem C]{CCR15}}}]\label{thm_uncn_Lp_G}
    Let  $G$ be a stratified Lie group equipped with a sub-Laplacian $L$, and  $|\cdot|$ is a homogeneous norm.   Suppose that $\gamma, \delta \in (0,\infty)$, $p, r \in (1,\infty)$, and $q \geq 1$, and assume that
\begin{align}\label{eqn_p,q,r,g,d_Str}
\frac{\gamma+\delta}{p} = \frac{\delta}{q} + \frac{\gamma}{r}.
\end{align}
Then, for all $f \in C_c^{\infty}(G)$, the following inequality holds
        \begin{align}\label{eqn_Lp_uncn_L,G}
             \|f\|_{L^p(G)} \leq C \| |x|^{\gamma}f\|_{L^q(G)}^{\frac{\delta}{\gamma+\delta}} \| L^{\frac{\delta}{2}}f\|_{L^r(G)}^{\frac{\gamma}{\gamma+\delta}}.
        \end{align}
\end{theorem}
We would like to mention that a straightforward argument based on dilation and homogeneity shows that inequality \eqref{eqn_Lp_uncn_L,G} can only hold for indices $(p, q, r)$ satisfying the condition \eqref{eqn_p,q,r,g,d_Str}. More recent developments on the uncertainty principle in the setting of groups can be found in \cite{DT15,RS17,RS17_A,RS19}.
For results on the $L^p$ version of the HPW inequality \eqref{eqn_uncn_L2_Rn} expressed via the Fourier transform, we refer the reader to \cite{Ste21, FX23, GS25} and the references therein. 

While significant progress has been made in the study of HPW-type uncertainty inequalities, most of the existing literature focuses on settings with polynomial volume growth. Notably, Martini \cite{Mar10} extended the results of \cite{CRS07} by formulating the $L^2$ HPW uncertainty inequality within a more abstract framework that accommodates spaces with non-polynomial, including exponential, volume growth. As applications, he derived $L^2$ HPW-type inequalities of the form \eqref{eqn_uncn_L_RN} on Riemannian manifolds and symmetric spaces of noncompact type; see \cite[Cor. 2, 3]{Mar10} for details. See also \cite{Bra24, DGLL25} for further developments in similar settings. In addition, Hardy–Littlewood–Sobolev and Stein–Weiss inequalities in this context were studied in \cite{KRZ_23, KKR24}.

In this work, we introduce and systematically study a \textit{shifted} version of Pitt's inequality and HPW–type uncertainty inequalities, formulated in accordance with the geometry and spectral properties of the Laplacian on Riemannian symmetric spaces $\X$ of noncompact type and the Jacobi transform. Unlike the Euclidean case, where no spectral shift appears, the Laplacian on $\X$ naturally involves a spectral shift, making the classical formulation of Pitt's inequality inadequate in this context. We propose the \textit{shifted} Pitt's inequality as the appropriate and intrinsic analogue for symmetric spaces, capturing the essential features of the underlying non-Euclidean structure.

As applications of the \textit{shifted} Pitt's inequality, we derive the $L^2$ HPW uncertainty inequality for a class of Laplacians on $\mathbb{X}$. We also establish an analogue of the Landau–Kolmogorov inequality for the same class of operators. Furthermore, by leveraging the spectral gap of the Laplacian on $\mathbb{X}$, we obtain a broader version of the $L^p$ HPW uncertainty principle (compare to Theorem~\ref{thm_uncn_Lp_G}), illustrating how the underlying geometry yields stronger and more general forms of such inequalities. We now proceed by providing a specific description of the problems considered in this article to facilitate a mathematically rigorous discussion. We refer the reader to Section \ref{sec_preliminaries} for any other unexplained notations. 

\subsection{Shifted Pitt's inequality on noncompact type symmetric spaces}
Let $G$ be a noncompact connected semisimple Lie group with finite center, and let $K$ be a maximal compact subgroup of $G$. The associated Riemannian symmetric space of noncompact type (referred to simply as a symmetric space throughout) is given by $\mathbb{X} = G/K$. We denote by $-\mathcal{L}$ the positive Laplace–Beltrami operator on $\X$.

Before delving into the main results of this article, we briefly motivate our study and highlight key differences between Pitt-type inequalities on symmetric spaces $\X$ and their classical Euclidean counterpart. In the Euclidean setting $\R^N$, the $L^2$-spectrum of the Laplacian $-\mathcal{L}_{\R^N}$ is the nonnegative half-line $[0, \infty)$, and the symbol of the fractional Laplacian $(-\mathcal{L}_{\R^N})^{\sigma/2}$ is simply $|\xi|^\sigma$ (see \eqref{eqn_symb_Frac_L}). In this sense, the weight for Pitt’s inequality on the Fourier side is actually hidden as the symbol of the fractional Laplacian on the corresponding spaces. In contrast, on a symmetric space $\X$ of noncompact type, the Laplacian has a spectral gap, with the bottom of the $L^2$-spectrum equal to $|\rho|^2$. Moreover, the symbol of the fractional Laplacian $(-\mathcal{L})^{\sigma/2}$ takes the \textit{shifted} form $(|\lambda|^2 + |\rho|^2)^{\sigma/2}$. This structural difference motivates the consideration of Pitt-type inequalities on $\X$ involving weights of the form $(|\lambda|^2 + \zeta^2)^{\sigma/2}$ on the Fourier transform side, for $0 \leq \zeta \leq |\rho|$, in contrast to the Euclidean case, where the weight is $|\xi|^\sigma$.

Furthermore, for a suitable function $f$ on $\X$, its Fourier transform $\widetilde{f}: \fa \times B \rightarrow \C$ is defined as a function of two variables, where $B$ is given in \eqref{defn:hft}. Unlike the Euclidean case, the (Poisson) kernel involved in defining $\widetilde{f}$ is not bounded on $\X$. This nontrivial behavior necessitated the introduction of mixed norms on the Fourier side  \cite{RS09,RR24} to establish corresponding Fourier inequalities. In light of these considerations, we propose to investigate the following version of \textit{shifted} Pitt's inequality on symmetric spaces of noncompact type, valid for $0 \leq \zeta \leq |\rho|^2$:
\begin{align}\label{inq_pitt_pol}
\left( \int_{\fa} \left( \int_B |\widetilde{f}(\lambda, b)|^2 \, db \right)^{\frac{q}{2}} (|\lambda|^2 + \zeta^2)^{- \frac{\sigma q}{2}} |\mathbf{c}(\lambda)|^{-2} \, d\lambda \right)^{\frac{1}{q}}
\leq C \left( \int_{\X} |f(x)|^p |x|^{\kappa p} \, dx \right)^{\frac{1}{p}}.
\end{align}
We will see that this formulation is not only the appropriate analogue of Pitt's inequality in the context of symmetric spaces, but it also reveals new admissible regimes for the weight exponents $(\sigma, \kappa)$ that differ substantially from the classical Euclidean case. In particular, we show that introducing a spectral shift $\zeta > 0$ enlarges the admissible region for Pitt-type inequalities. We further provide a sharp characterization of the admissible parameters with non-negative exponents in both the rank one symmetric space and Jacobi transform settings for the range $p \leq q \leq p'$.

Let us begin by analyzing the necessary conditions for the validity of \eqref{inq_pitt_pol}, starting with the unshifted case $\zeta = 0$. Unlike the constant Plancherel density in Euclidean spaces, the density $|\mathbf{c}(\lambda)|^{-2}$ on $\X$ exhibits distinct behavior both near the origin and at infinity (see \eqref{est_c-2,hr}). Particularly, in the rank one case, one has the asymptotic estimates $|\mathbf{c}(\lambda)|^{-2} \asymp |\lambda|^{\nu - 1}$ as $|\lambda| \to 0$ and $|\mathbf{c}(\lambda)|^{-2} \asymp |\lambda|^{n - 1}$ as $|\lambda| \to \infty$, where $\nu$ is the pseudo-dimension of $\X$ and $n = \dim \X$. This asymptotic behavior yields the necessary conditions $\sigma - \kappa \geq n(1/p + 1/q - 1)$ and $\sigma < \nu/q$ for the validity of \eqref{inq_pitt_pol} (see Theorem \ref{thm_nec_pitt_n}), in contrast to the Euclidean requirement $\sigma < N/q$. As noted in Remark~\ref{rem_prob_lam_0}, these two constraints may be incompatible, implying that the classical Hardy–Littlewood–Paley inequality (corresponding to $1 < p = q \leq 2$ and $\kappa = 0$ in \eqref{inq_pitt_pol}) may fail to hold in this setting unless $n \leq \nu$.

These obstructions, however, vanish when $\zeta \neq 0$, making it the more appropriate and geometrically natural case to consider. In particular, by applying the shifted Pitt’s inequality \eqref{inq_pitt_pol} with $\zeta = |\rho|^2$, we obtain a precise analogue of the HPW inequality \eqref{eqn_uncn_L_RN} adapted to the geometry of $\X$. This perspective strongly motivates the formulation and study of the shifted Pitt’s inequality in the context of symmetric spaces.

To the best of our knowledge, this shifted formulation of Pitt’s inequality has not previously appeared in the literature. It provides an intrinsic extension of the classical theory to non-Euclidean settings, where the spectral properties of the Laplacian inherently involve such a shift. We now present our first main result, establishing the validity of the shifted Pitt’s inequality \eqref{inq_pitt_pol} on symmetric spaces of noncompact type.  For $1 \leq p \leq \infty$, $p' = p/(p-1)$ denotes the conjugate exponent of $p$. 
\begin{theorem}\label{thm_pitt_X_n}
     Let $\X$ be a Riemannian symmetric space of noncompact type with dimension $n \geq 2$ and pseudo-dimension $\nu \geq 3$. Suppose $1 \leq p \leq q \leq p'$ and  $\sigma, \kappa \geq 0 $, which satisfy
\begin{align}\label{eqn_suff_s-k>X}
\kappa < \frac{n}{p'} \, (\text{when } p=1, \kappa=0), \quad \text{and} \quad \sigma -\kappa \geq  n\left(\frac{1}{p} + \frac{1}{q} - 1\right).
\end{align}
Then for $\zeta \neq 0$ and $1 \leq p \leq 2$, the shifted Pitt’s inequality \eqref{inq_pitt_pol} holds for all $f \in C_c^{\infty}(\X)$.

Moreover, the Pitt's inequality \eqref{inq_pitt_pol} also holds for $\zeta = 0$ provided that $\min\{n, \nu\} = n$ and the following additional conditions are satisfied:
\begin{align}\label{eqn_suff_s-k<X}
\sigma < \frac{\nu}{q}, \quad \text{and} \quad \sigma - \kappa \leq \nu\left( \frac{1}{p} + \frac{1}{q} - 1 \right).
\end{align}
Conversely, in the rank one case, the condition \eqref{eqn_suff_s-k>X} (for $p > 1$) together with $p \leq 2$ are necessary for \eqref{inq_pitt_pol} to hold for all $\zeta \geq 0$. Furthermore, when $\zeta = 0$, the condition \eqref{eqn_suff_s-k<X} is also necessary in the case $p = 2$.
\end{theorem}
\begin{remark}
\begin{enumerate}
        \item The admissible region for the weight parameters in the shifted Pitt’s inequality, as illustrated in Figure \ref{pitt_X_Fig}, highlights key distinctions between the Euclidean and symmetric space settings. In the Euclidean case, both $\kappa$ and $\sigma$ are constrained, and the \textit{balance condition} (magenta segments in Figure \ref{pitt_X_Fig}) is necessary for the inequality to hold (see \eqref{k-s=N_equiv}). In contrast, for symmetric spaces, while a similar restriction on $\kappa$ remains, the additional constraint on $\sigma$ is absent when $\zeta \neq 0$, and the \textit{balance condition} is relaxed due to the lack of a dilation structure. As a result, the inequality \eqref{inq_pitt_pol} holds over a broader region in the $(\sigma, \kappa)$-plane, rather than being confined to a line. Moreover, when $\zeta = 0$ and $n \leq \nu$, the admissible region remains at least as large, and in some cases larger, than in the Euclidean setting.

    \item We note that in \cite[Corollary 1.8]{KRZ_23}, the second author and his collaborators obtained Pitt's inequality with $\zeta = 0$ for the case $q = 2$ and $1 \leq p \leq 2$. 
    \item As a direct consequence of the theorem above, we obtain the following characterization result in the rank one case.
\end{enumerate}
\end{remark}
\begin{figure}[ht]
    \centering
        \begin{subfigure}{0.4\textwidth}
            \begin{tikzpicture}[line cap=round,line join=round,>=triangle 45,x=1.0cm,y=1.0cm, scale=0.33]
\clip(-3,-2) rectangle (15,11);
\draw [line width=0.5pt] (0,0)-- (0,10);
\draw[-latex, color=black](10,0)--(10.5,0);
%\draw [line width=1pt,color=cyan] (10,1)-- (10,10);
\draw[-latex, color=cyan](0,10.5);
%\draw [line width=0.5pt] (10,0)-- (10,1);
%\draw [line width=0.5pt] (0.,10)-- (10,10);
%\draw [line width=0.5pt] (10,10)--(8,8);
%\draw [line width=0.5pt] (10,0)--(8,0);

%\draw [line width=1pt,color=cyan] (1,1)--(10,10);
%\draw [line width=0.5pt] (1,1)--(0,0);
%\draw [line width=1pt,color=cyan,dashed] (1,1)--(10,1);
%\draw [line width=1pt,color=cyan,dashed] (8,8)--(8,0);
\draw [line width=0.5pt] (0,0)--(10,0);

%\fill[line width=1pt,color=cyan,fill=cyan, opacity=0.1] (1,1) -- (10,1) -- (10,10) --cycle;

% \draw (10,1) node[color=cyan]{$\circ$};
% \draw (1,1) node[color=cyan]{$\circ$};
%\draw (10,10) node[color=cyan]{$\circ$};
%%%
\draw (-0.75,-0.75) node{\small $0$};
\draw (10,-0.75) node{\small $\kappa$};
\draw (-0.75,10) node{\small $\sigma$};

\draw (4,3.5) node[rotate=45]{\tiny $\!\frac{\sigma\! -\!\kappa}{n} = \frac{1}{q}-\!\!\frac{1}{p'}\!$};
%%%%
%%%
\draw [line width=1pt,color=cyan] (0,1)--(6,7);
\draw [line width=1pt,color=cyan] (0,1)--(0,10);

\draw [line width=1pt,color=cyan,dashed] (6,7)--(6,10);
\draw[line width=1pt,color=magenta] (0,1)--(6,7);

\fill[line width=1pt,color=cyan,fill=cyan, opacity=0.2] (0,1) -- (6,7) -- (6,10) --(0,10)--cycle;

\draw (0,1) node[color=cyan]{$\circ$};
\draw (6,7) node[color=cyan]{$\circ$};
%\draw (6,0) node[color=cyan]{$\circ$};
%\draw (0,7) node[color=cyan]{$\circ$};

 \draw (8,7) node{\tiny $(\frac{n}{p'},\frac{n}{q})$};
%\draw (-1,7) node{\small $\frac{n}{q}$};
%\draw (6,-1) node{\small $\frac{n}{p'}$};
%%%
%%%%
 
% \draw (10.7,1) node{\small $\frac{\beta}{n}$};
% \draw (1,-1) node{\small $\frac{\beta}{n}$};
\end{tikzpicture}
            \noindent\subcaption{When $\zeta \neq 0$}
        \end{subfigure}
        ~
        \begin{subfigure}{0.33\textwidth}
            \begin{tikzpicture}[line cap=round,line join=round,>=triangle 45,x=1.0cm,y=1.0cm, scale=0.33]
\clip(-3,-2) rectangle (15,11);
\draw [-latex,line width=0.5pt] (0,0)-- (0,10);
\draw[-latex, color=black](10,0)--(10.5,0);
%\draw [line width=1pt,color=cyan] (10,1)-- (10,10);
%\draw[line width=1pt, color=cyan](0,3);
%\draw [line width=0.5pt] (10,0)-- (10,1);
%\draw [line width=0.5pt] (0.,10)-- (10,10);
%\draw [line width=0.5pt] (10,10)--(8,8);
%\draw [line width=0.5pt] (10,0)--(8,0);

%\draw [line width=1pt,color=cyan] (1,1)--(10,10);
%\draw [line width=0.5pt] (1,1)--(0,0);
%\draw [line width=1pt,color=cyan,dashed] (1,1)--(10,1);
%\draw [line width=1pt,color=cyan,dashed] (8,8)--(8,0);
\draw [line width=0.5pt] (0,0)--(10,0);

%\fill[line width=1pt,color=cyan,fill=cyan, opacity=0.1] (1,1) -- (10,1) -- (10,10) --cycle;

% \draw (10,1) node[color=cyan]{$\circ$};
% \draw (1,1) node[color=cyan]{$\circ$};
%\draw (10,10) node[color=cyan]{$\circ$};
%%%
\draw (-0.75,-0.75) node{\small $0$};
\draw (10,-0.75) node{\small $\kappa$};
\draw (-0.75,10) node{\small $\sigma$};

\draw (4,3.5) node[rotate=45]{\tiny $\!\frac{\sigma\! -\!\kappa}{n} = \frac{1}{q}-\!\!\frac{1}{p'}\!$};
\draw (2.5,6.6) node[rotate=45]{\tiny $\!\frac{\sigma\! -\!\kappa}{\nu} = \frac{1}{q}-\!\!\frac{1}{p'}\!$};
%%%%
%%%
\draw [line width=1pt,color=cyan] (0,1)--(6,7);
%\draw [line width=1pt,color=cyan] (0,1)--(0,10);
\draw [line width=1pt,color=cyan] (0,3)--(6,9);
\draw [line width=1pt,color=cyan] (0,1)--(0,3);

\draw [line width=1pt,color=cyan,dashed] (6,7)--(6,10);
\draw[line width=1pt,color=magenta] (0,1)--(6,7);

\fill[line width=1pt,color=cyan,fill=cyan, opacity=0.2] (0,1) -- (6,7) -- (6,9) --(0,3)--cycle;

\draw (0,1) node[color=cyan]{$\circ$};
\draw (6,7) node[color=cyan]{$\circ$};
\draw (6,9) node[color=cyan]{$\circ$};
\draw (0,3) node[color=cyan]{$\circ$};
%\draw (6,0) node[color=cyan]{$\circ$};
%\draw (0,7) node[color=cyan]{$\circ$};

 \draw (8,7) node{\tiny $(\frac{n}{p'},\frac{n}{q})$};
 \draw (9.3,9) node{\tiny $(\frac{n}{p'},\frac{\nu}{q} -\frac{\nu-n}{p'})$};
%\draw (-1,7) node{\tiny $\frac{n}{q}$};
%\draw (6,-1) node{\tiny $\frac{n}{p'}$};
%%%
%%%%
 
% \draw (10.7,1) node{\small $\frac{\beta}{n}$};
% \draw (1,-1) node{\small $\frac{\beta}{n}$};
\end{tikzpicture}

 
            \subcaption{When $\zeta =0$ and $n< \nu$ }
        \end{subfigure}
        \caption{Admissible regions of $\sigma$ and $\kappa$ for Pitt-type inequality on $\X$, given $1 < p \leq 2$ and $ p < q < p'$.}
         \label{pitt_X_Fig}
    \end{figure}
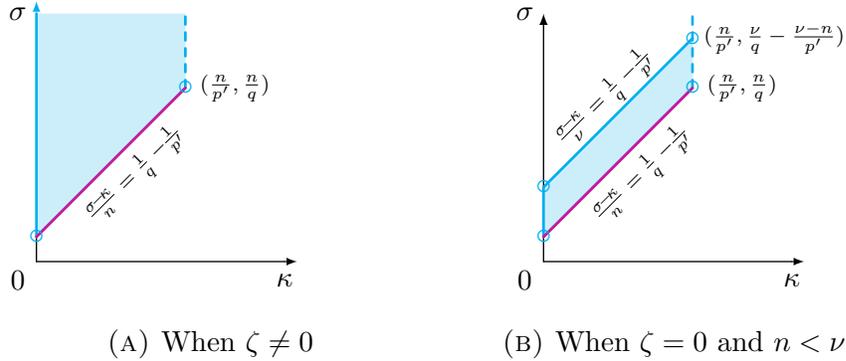 
\begin{corollary}\label{cor_char_pitt_X<p'}
    Let $\X$ be a rank one symmetric space of noncompact type (in particular, hyperbolic space) with dimension $n \geq 2$. Suppose that  $1 < p \leq q \leq p'$ and $\kappa \geq 0, \sigma \in \R$. Then the shifted Pitt's inequality \eqref{inq_pitt_pol} holds for $\zeta>0$ if and only if $p\leq 2$ and  \eqref{eqn_suff_s-k>X} holds.
\end{corollary}

In \cite{KPRS24}, motivated by the work of Heinig \cite{He84}, the first author and his collaborators established $(L^p, L^q)$ Fourier inequalities with general weights in the setting of rank one symmetric spaces of noncompact type, using Calderón’s rearrangement method. 
 
However, a key obstacle in applying \cite[Theorem 1.4]{KPRS24} to polynomial weights is that, in spaces with exponential volume growth, the decreasing rearrangement of a polynomial weight behaves like a logarithmic function rather than retaining its polynomial nature. Consequently, an exponential weight had to be introduced on the function side to derive a version of the unshifted Pitt’s inequality. This formulation, however, appears somewhat unnatural and has limited applicability (see \cite[Remark 5.2]{KPRS24}).

To overcome this limitation, in the present article we first establish an analogue of the classical Hardy–Littlewood–Paley inequality, and then apply a mixed norm interpolation method,
to derive a more effective version of the Pitt-type inequalities.

Furthermore, in the context of the Jacobi transform, we demonstrate how Calder{\'o}n’s rearrangement method can be adapted in a more effective manner. Specifically, in Section~\ref{sec_jac_transform}, we modify the standard Jacobi transform so that the associated measure on the function side exhibits polynomial growth. This modification makes Calder{\'o}n’s method suitably applicable and leads to sharp sufficient conditions for Pitt-type inequalities in the modified Jacobi transform setting. We also prove that these conditions are necessary, thereby completely characterizing the admissible class of polynomial weights with non-negative exponents. In addition, by employing sharp estimates of the Jacobi (spherical) functions, we derive \textit{shifted} Pitt's inequalities for the standard Jacobi transform and extend the results to include the case $q > p'$, which was not accessible for symmetric spaces in Theorem~\ref{thm_pitt_X_n}.

Due to the additional notation and technical framework required, we refer the reader to Section~\ref{sec_jac_transform} for a detailed discussion of the results on shifted Pitt's inequality and its necessary conditions in the Jacobi setting, and to Section~\ref{sec_app_jac} for applications to uncertainty principles.
\subsection{Uncertainty principles on noncompact type symmetric spaces}
Uncertainty principles of various forms, such as Hardy’s theorem \cite{Sen02}, Beurling’s theorem \cite{SS08}, and uncertainty principles related to the Schrödinger equation \cite{PS12}, have been formulated within the framework of symmetric spaces of noncompact type. Ingham–Chernoff-type theorems were introduced in this context by Bhowmik-Pusti-Ray \cite{BPR20, BPR23}, with further developments appearing in more recent works, including \cite{GMT22, GT22, Sar25}. 
Within this broader context, it is natural to explore analogues of the Heisenberg–Pauli–Weyl uncertainty inequality in non-Euclidean settings, such as symmetric spaces of noncompact type (in particular, hyperbolic spaces) and in the framework of the Jacobi transform. Our first result in this direction establishes an $L^2$ version of the HPW inequality \eqref{eqn_uncn_L2_Rn} for a class of Laplacians on $\X$, obtained through a direct application of shifted Pitt's inequality. 
\begin{corollary}\label{cor_unc_2,2_F}
     Let $\X$ be a general symmetric space of noncompact type and let $\gamma, \delta \in (0,\infty) $, $p_0\geq 1$, and $|\rho_{p_0}| := |2/{p_0}-1| |\rho|$. Then we have the following 
    \begin{align}\label{eqn_unc_2,2_F}
        \|f\|_{L^2(\X)} \leq C \| |x|^{\gamma} f\|_{L^2(\X)}^{\frac{\delta} {\gamma+\delta}}  \| (|\lambda|^2+ |\rho_{p_0}|^2)^{\frac{\delta}{2}}\widetilde{f}\|_{L^2(\fa \times B, |\bfc(\lambda)|^{-2}d\lambda db)}^{\frac{\gamma} {\gamma+\delta}}
    \end{align}
    for all $f \in C_c^{\infty}(\X)$. Equivalently, we have 
     \begin{align}\label{eqn_unc_2,2_X}
        \|f\|_{L^2(\X)} \leq C \| |x|^{\gamma} f\|_{L^2(\X)}^{\frac{\delta} {\gamma+\delta}}  \|  (-\mathcal{L}_{p_0})^{\frac{\delta}{2}} f \|_{L^2(\X)}^{\frac{\gamma} {\gamma+\delta}},
    \end{align}
    where $\mathcal{L}_{p_0}:= \mathcal{L}+(|\rho|^2-|\rho_{p_0}|^2) I$ is known as modified or $L^{p_0}$-Laplacian on $\X$. 
\end{corollary}
We would like to emphasize the importance of the modified Laplacians $(-\mathcal{L}_{p_0})^{\delta/2}$ and, more generally, of operators of the form $(zI - \mathcal{L})^{{\delta}/{2}}$, which have been extensively studied due to their important role in harmonic analysis and multiplier theory in various settings. For a comprehensive discussion, we refer the reader to \cite[Sec. 4]{Ank92}. In the specific context of symmetric spaces of noncompact type, we refer \cite[Sec. 4]{CGM93} and \cite[Sec. 5]{Ion03}.

Our next result is an $L^p$-version of the HPW inequality on $\X$, inspired by the work of Ciatti, Cowling, and Ricci \cite{CCR15} on stratified Lie groups. While their results serve as a motivating framework, the setting of symmetric spaces allows us to formulate a broader version than Theorem~\ref{thm_uncn_Lp_G}, reflecting structural differences between the two contexts.
\begin{theorem}\label{thm_Lp_unc_x}
           Let $\X$ be a general symmetric space of noncompact type, and let $1 \leq p_0 < 2$. Suppose $\gamma, \delta  \in (0,\infty)$, and that the exponents $p,q$, and $r$ with $p, r \in (p_0, p_0')$ and $q\geq 1$, which satisfy \eqref{eqn_p,q,r,g,d_Str}.
Then, for any $\sigma \geq \delta$, we have the following inequality
        \begin{align}\label{eqn_uncn_p0_thm_int}
            \|f\|_{L^p(\X)} \leq C \| |x|^{\gamma}f\|_{L^q(\X)}^{\frac{\delta}{\gamma+\delta}} \| (-\mathcal{L}_{p_0})^{\frac{\sigma}{2}}f\|_{L^r(\X)}^{\frac{\gamma}{\gamma+\delta}}
        \end{align}
        for all $f \in C_c^{\infty}(\X)$.
       \end{theorem}
    \begin{remark}
      Theorem~\ref{thm_Lp_unc_x} is a generalization of the corresponding result for stratified Lie groups (Theorem~\ref{thm_uncn_Lp_G}) in the following sense: when $p_0=1$ and $\sigma = \delta$, inequality~\eqref{eqn_uncn_p0_thm_int} recovers an exact analogue of Theorem~\ref{thm_uncn_Lp_G}. However, on symmetric spaces, our result remains valid for any $\sigma \geq \delta$, offering a broader range of applicability. Moreover, Corollary~\ref{cor_unc_2,2_F} recovers \cite[Cor. 2]{Mar10} as a special case and provides an improvement of \cite[Cor. 3]{Mar10}. We also note that Theorem~\ref{thm_Lp_unc_x} remains valid if $\mathcal{L}_{p_0}$ is replaced by $\mathcal{L} + (|\rho|^2 - \zeta^2) I$ for any $\zeta > |\rho|$, for all $p, q, r \in (1, \infty)$ satisfying \eqref{eqn_p,q,r,g,d_Str}.
        \end{remark}

We conclude this section with an outline of the article. The next section provides essential background on harmonic analysis related to semisimple Lie groups and noncompact type symmetric spaces, along with a review of relevant previous results. In Section \ref{sec_pitt_X}, we establish the shifted Pitt's inequality in general symmetric spaces of noncompact type, specifically Theorem \ref{thm_pitt_X_n}. Section~\ref{sec_jac_transform} investigates Pitt's inequality in the context of various Jacobi transforms, addressing both the sufficient and necessary conditions. In Section~\ref{sec_uncn_principle}, we study uncertainty principles on $\X$ and their connection to shifted Pitt's inequality, culminating in the proof of Theorem~\ref{thm_Lp_unc_x}. 

\vspace{3mm}
\textbf{Notation.} We use the standard notation $\N$, $\Z$, $\R$, and $\C$ to denote the set of natural numbers, the ring of integers, and the fields of real and complex numbers, respectively. For any $z \in \C$, we write $\Re z$ and $\Im z$ for the real and imaginary parts of $z$. Throughout this article, the symbols $C$, $C_1$, $C_2$, etc., will denote positive constants, which may vary from line to line. We write $X \lesssim Y$ (respectively, $X \gtrsim Y$) to mean that there exists a constant $C > 0$ such that $X \leq CY$ (respectively, $X \geq CY$). The notation $X \asymp Y$ indicates that both $X \lesssim Y$ and $X \gtrsim Y$ hold.

\section{Preliminaries}\label{sec_preliminaries}
Here, we review some general facts and necessary preliminaries regarding semisimple Lie groups and harmonic analysis on Riemannian symmetric spaces of noncompact type. Most of this information is already known and can be found in \cite{GV88, Hel08}. We will use the standard notation from \cite{Ion03, KRZ_23}. To keep the article self-contained, we will include only the results necessary for this article.
\subsection{Riemannian symmetric spaces} Let $G$ be a connected, noncompact, real semisimple Lie group with a finite center, and let $\mathfrak{g}$ be its Lie algebra. Let  $\theta$ be a Cartan involution of $\mathfrak{g}$, inducing the decomposition   $\mathfrak{g} = \mathfrak{k} \oplus \mathfrak{p}$. Let $K= \exp \mathfrak{k} $ be a maximal compact subgroup of $G$, and let $\mathbb{X} = G / K$ be an associated symmetric space with origin $\mathbf{0} = \{eK\}$. The Killing form of $\mathfrak{g}$ induces a $K$-invariant inner product  $\left\langle \cdot, \cdot \right\rangle$ on $ \mathfrak{p}$ and therefore a $G$-invariant Riemannian metric on $\X$. 
  
    Let $\mathfrak{a}$ be a maximal abelian subalgebra of $\mathfrak{p}$. The dimension of $\mathfrak{a}$ determines the rank of the symmetric space $\mathbb{X}$, denoted by $l$. Using the  inner product $\left\langle \cdot, \cdot \right\rangle$ inherited from  $\mathfrak{p}$, we identify  $\mathfrak{a}$ with $\mathbb{R}^l$. Let $\mathfrak{a}^*$ and $\mathfrak{a}_{\mathbb{C}}^*$ denote the real and complex duals of $\mathfrak{a}$, respectively. For any $\lambda \in \mathfrak{a}^*$, let $H_{\lambda}$ be the unique element of $\mathfrak{a}$ such that
    \begin{align*}
    \lambda(H) = \left\langle H_{\lambda}, H \right\rangle, \quad \text{for all } H \in \mathfrak{a}.
    \end{align*}
   With this, we transfer the inner product $\left\langle \cdot, \cdot \right\rangle$ from $\mathfrak{a}$ to $\mathfrak{a}^*$, and, by abuse of notation, denote it again by $\left\langle \cdot, \cdot \right\rangle$, defined by the rule
    \begin{align*}
    \left\langle \lambda, \lambda' \right\rangle= \left\langle H_{\lambda}, H_{\lambda'} \right\rangle, \quad \text{for all }  \lambda, \lambda' \in \mathfrak{a}^*.
\end{align*}
We identify $\fa$ with its dual $\fa^*$ by
means of the inner product defined above. We further extend the inner product on $\mathfrak{a}^*$ to a $\mathbb{C}$-bilinear form on $\mathfrak{a}_{\mathbb{C}}^*$.

Let $\Sigma \subset \mathfrak{a}^*$ be the set of restricted roots for the pair $(\mathfrak{g}, \mathfrak{a})$. Let $M$ denote the centralizer of $\mathfrak{a}$ in $K$, and let $W$ represent the Weyl group associated with $\Sigma$. For each $\alpha \in \Sigma$, let $\mathfrak{g}_\alpha$ be the corresponding root space, with $m_\alpha = \dim \mathfrak{g}_\alpha$. We choose a system of positive roots $\Sigma^+$ and define the associated positive Weyl chamber as $\mathfrak{a}^+ = \{H \in \mathfrak{a} : \alpha(H) > 0 \text{ for all } \alpha \in \Sigma^+\}$. Let $\Sigma_0^+$ denote the set of positive indivisible roots. The half-sum of all positive roots, counted with their multiplicities, is given by
$\rho = \frac{1}{2} \sum_{\alpha \in \Sigma^+} m_\alpha \alpha. $ 

Let $n$ represent the dimension of $\mathbb{X}$ and $\nu$ the pseudo-dimension (or dimension at infinity) of $\mathbb{X}$. These quantities are given by the following expressions
\begin{align*} n = l + \sum_{\alpha \in \Sigma^+} m_\alpha, \qquad \nu = l + 2 | \Sigma_0^+ |. \end{align*}
For example, $\nu=3$ while $n\geq 2 $ is arbitrary in the rank one case and $\nu=n$ if $G$ is complex.
In general, we have $l \geq 1$, $n \geq 2$, and $\nu \geq 3$. We denote by $-\mathcal{L}$ the positive Laplace-Beltrami operator on $\X$, and it is known that $|\rho|^2$ is the bottom of its $L^2$-spectrum.

By $\mathfrak{n} = \bigoplus_{\alpha \in \Sigma^+} \mathfrak{g}_\alpha$, we denote the nilpotent Lie subalgebra of $\mathfrak{g}$ associated with $\Sigma^+$, and by $N$ and $A$ the corresponding Lie subgroups of $\mathfrak{n}$ and $\mathfrak{a}$, respectively, in $G$. The Iwasawa decomposition of $G$ is given by $G = KAN$, meaning each element $g \in G$ can be uniquely expressed as \begin{align*} g = K(g) \exp(H(g))N(g), \quad K(g) \in K,\, H(g) \in \mathfrak{a}, \, N(g) \in N. \end{align*} Let $A^{+} = \exp(\mathfrak{a}^+)$, and let $\overline{A^{+}}$ denote the closure of $A^{+}$ in $G$. We then obtain the polar decomposition $G = K \overline{A^{+}} K$. As usual, on the compact group $K$ (respectively on $B=K/M$), we fix the normalized Haar
measure $dk$ (respectively the normalized measure  $db$). Under the Cartan decomposition, the Haar measure on $G$ can be expressed as follows.
\begin{equation*}
\begin{aligned}
\int_G f(g) d g & =\int_K \int_{\mathfrak{a}^{+}} \int_K f\left(k_1 (\exp H^+) k_2\right) \delta(H^+) \, d k_1\, d H^+\, d k_2,
\end{aligned}
\end{equation*}
where  $K$ is equipped with its normalized Haar measure, the density function $\delta(H^+)=C_0 \prod_{\alpha \in \Sigma^{+}}(\sinh \alpha(H^+))^{m_\alpha} $.
 If $f$ is a function on $\mathbb{X} = G / K$, it can be viewed as a right $K$-invariant function on $G$. Moreover, a function $f $ is called a $K$-biinvaraint if $f(x)= f(\exp x^+)$ for all $x \in G$, where $x^+$ denotes the component of $x$ in $\overline{A^{+}}$ under the Cartan decomposition.

\subsection{Fourier transform on Riemannian symmetric spaces}
The Fourier transform $\widetilde{f}$ of a smooth, compactly supported function $f$ on $\mathbb{X}$ is defined on $\mathfrak{a}_{\mathbb{C}} \times B$ and is given by (see \cite[page 199]{Hel08}):
\begin{equation}\label{defn:hft}
\widetilde{f}(\lambda,b) = \int_{G} f(g) e^{(i\lambda- \rho)H(g^{-1}b)} dg,\:\:\:\:\:\:  \lambda \in \mathfrak{a}_{\mathbb{C}}, b \in B,
\end{equation}
where $B= K/ M$. It is known that if $f \in L^1(\mathbb{X})$, then $\widetilde{f}(\lambda, b)$ is continuous with respect to $\lambda \in \mathfrak{a}$ for almost every $b \in B$ (and is, in fact, holomorphic in $\lambda$ on a domain containing $\mathfrak{a}$). Furthermore, if $\widetilde{f} \in L^1(\mathfrak{a} \times B, |\bfc(\lambda)|^{-2}  \,d\lambda \,db)$, the following Fourier inversion formula holds for almost every $gK \in \mathbb{X}$ (\cite[Ch. III]{Hel08})
\begin{align*}
f(gK)= \frac{1}{|W|} \int_{\fa} \int_B \widetilde{f}(\lambda, b)~e^{-(i\lambda+\rho)H(g^{-1}b)} ~ |\bfc(\lambda)|^{-2}\,d\lambda~db,
\end{align*}
 where $\bfc(\lambda)$ is the Harish-Chandra's $\bfc$-function. Moreover, the Plancherel theorem states \cite[Ch. III, Theorem 1.5]{Hel08}
\begin{align}\label{Planc}
    \int_{\X} f_1(x) \overline{f_2(x)} \,dx= \frac{1}{|W|} \int_{\fa} \int_B \widetilde{f_1}(\lambda,b) \overline{\widetilde{f_2}(\lambda,b)}  |\bfc(\lambda)|^{-2} \, db \, d\lambda,
\end{align}
for all $f_1, f_2 \in C_c^{\infty}(\X)$. We also require the following estimates \cite[(2.2)]{CGM93} (see also \cite[Ch. IV, prop 7.2]{Hel00})
\begin{equation}\label{est_c-2,hr}
\begin{aligned}
    |\bfc(\lambda)|^{-2} \leq C |\lambda|^{\nu -l} (1+ |\lambda|)^{n-\nu}, \quad \text{for all } \lambda \in \mathfrak{a}.
\end{aligned}
\end{equation}
If $f$ is $K$-biinvariant, 
%meaning  $f(k_1gk_2)=f(g)$ for all  $k_1, k_2\in K$ and $g \in G$,
then  $\widetilde{f}$ does not depend on $B$ and  it reduces to spherical transform. In this case, the formula \eqref{defn:hft} simplifies to
\begin{equation}\label{eqn_sph_trans}
\widetilde{f}(\lambda, b)=\widehat{f}(\lambda)=\int_G f(g) \varphi_{-\lambda}(g) \,d g,
\end{equation}
for all $\lambda \in \mathfrak{a}, b \in B$, where
\begin{align*}
\varphi_\lambda(g)=\int_K e^{-(i \lambda+\rho)H(g^{-1} k)} d k, \quad \lambda \in \mathfrak{a}_{\mathbb{C}},
\end{align*} is the Harish-Chandra's elementary spherical function \cite{HC58}. In the rank one case, it is known that spherical analysis can be naturally embedded into the broader framework of Jacobi analysis developed by Koornwinder \cite{Koo84} (see also \cite{Koo75, ADY96}).
This connection between spherical analysis and Jacobi Fourier analysis enables the translation of many problems on Riemannian symmetric spaces into a setting where classical tools, such as special function estimates, weighted norm inequalities, and operator theory, can be effectively employed.
We will come back to this in Section \ref{sec_jac_base}, where we briefly recall some facts related to the Jacobi transform.

Now we like to recall some results from \cite{KRZ_23,RR24} which will be useful for us.  We can write from  \cite[Theorem 1.8]{RR24} the following mixed norm version of the Hausdorff-Young inequality on $\X$.
\begin{theorem}[{{\cite[Theorem 1.8]{RR24}}}]\label{thm_HY_UD}
     Let $\X$ be a general symmetric space of noncompact type, and let $1\leq p\leq 2 $. Then, for $1 < p \leq 2$ and all $f \in C_c^{\infty}(\X)$, we have the inequality
\begin{align}\label{eqn_HY_X,n}
    \left( \frac{1}{|W|}\int_{\mathbb{\fa}} \left(\int_B|\widetilde{f}(\lambda, b)|^{2} \, db\right)^{\frac{p'}{2}} |\bfc(\lambda)|^{-2} \,d \lambda \right)^{\frac{1}{p'}} \leq  \|f\|_{L^p(\X)}.
\end{align}  
For $p = 1$, we have the following analogue of a restriction estimate
%from \cite[(3.1)]{RR24}
     \begin{align}\label{f_L2K<L1_pre}
         \left( \int_B  \left| \widetilde{f}(\lambda,b)\right|^2 \,db \right)^{\frac{1}{2}}  & \leq  \|f\|_{L^1(\X)}.
         \end{align}
Moreover, if $p \geq 2$, then for all $f \in C_c^{\infty}(\X)$ we have the reverse inequality
%Moreover, if $p \geq 2$, then we have for all $f\in C_c^{\infty}(\X)$
\begin{align}\label{eqn_HYP_dual-int}
     \|f\|_{L^p(\X)} \leq   \left( \frac{1}{|W|}\int_{\mathbb{\fa}} \left(\int_B|\widetilde{f}(\lambda, b)|^{2} \, db\right)^{\frac{p'}{2}} |\bfc(\lambda)|^{-2} \,d \lambda \right)^{\frac{1}{p'}} .
\end{align}
\end{theorem}
We also recall from \cite[Theorem 1.9]{RR24} the following analogue of Paley's inequality \cite[Theorem 1.10]{Hor60} in the context of noncompact type symmetric spaces, which can also be seen as a weighted version of Plancherel’s formula for $L^p(\X)$ with $1 < p \leq 2$.
\begin{theorem}[{{\cite[Theorem 1.9]{RR24}}}]\label{thm_paley_uni}  Let $\X$ be a general symmetric space of noncompact type and let $1<p\leq 2$.
 Assume that $u$ is a positive function on $\fa$ satisfying the following condition 
	\begin{align*}
	 	 \|u\|_{c,\infty} := \sup_{\alpha>0} \alpha  \int\limits_{\substack{\{\lambda \in \fa : u(\lambda)>\alpha\} }} |\bfc(\lambda)|^{-2} \,d\lambda <\infty.
	 \end{align*}
  Then we have for all $f\in C_c^{\infty}(\X)$
  \begin{align*}
       \left( \int_{\mathbb{\fa}} \left(\int_B|\widetilde{f}(\lambda, b)|^{2} \, db\right)^{\frac{p}{2}} u(\lambda)^{2-p} |\bfc(\lambda)|^{-2} \,d \lambda \right)^{\frac{1}{p}} \leq C_p  \|u\|_{c,\infty}^{\frac{2}{p}-1}\|f\|_{L^p(\X)}.
  \end{align*}
\end{theorem}
We have the following from \cite[Corollary 1.8]{KRZ_23}.
\begin{corollary}\label{cor_pitt_KRZ,0}
    Let $\X$ be a symmetric space of dimension $n\geq 2$ and pseudo-dimension $\nu \geq 3$. Suppose that $0< \sigma< \nu$ and $f \in C_c^{\infty}(\X)$. Then, the inequality 
    \begin{align}\label{inq_pit_p,2}
     \left(   \int_{\fa} \int_{B}|\widetilde{f}(\lambda, b)|^2 |\lambda|^{-2\sigma} |\bfc(\lambda)|^{-2} \,db\, d\lambda \right)^{\frac{1}{2}} \leq C \left( \int_{\X}|f(x)|^p |x|^{\kappa p}\,dx \right)^{\frac{1}{p}}
    \end{align}
    holds for $1\leq   p<2$ if and only if   $0\leq \kappa\leq \sigma<{\nu}/{2}$, $\kappa<{n}/{p'}$ (when $p=1$, $\kappa=0$), and $\sigma-\kappa \geq n\left( {1}/{p}-{1}/{2}\right)$. Moreover, the inequality \eqref{inq_pit_p,2} holds for $p=2$ if and only if $\kappa=\sigma<\min\{\frac{\nu}{2},\frac{n}{2}\}$.
\end{corollary}

\subsection{Mixed norm interpolation theory}\label{subsec: mix_norm}
 Let $\left(X_j, dx_j\right), j=0,1$, be two $\sigma$-finite measure spaces and $(X, dx)$ their product measure space. For an ordered pair $P=\left(p_0, p_1\right) \in[1, \infty] \times[1, \infty]$ and a measurable function $f\left(x_0, x_1\right)$ on $(X, \mu)$,  we define the mixed norm $\left(p_0, p_1 \right)$ of $f$ as
\begin{align}\label{defn_mix_norm}
\|f\|_P=\|f\|_{{\left(p_0, p_1\right)}}=\left(\int_{X_0}\left(\int_{X_1}\left|f\left(x_0, x_1\right)\right|^{p_0} d x_1\right)^{p_1 / p_0} d x_0\right)^{1 / p_1} .
\end{align}
 For two such ordered pairs $P=\left(p_0, p_1\right)$ and $Q=\left(q_0, q_1\right)$, we write $1 / R=(1-\theta) / P+\theta / Q$, $0<\theta<1$ to mean $1 / r_0=(1-\theta) / p_0+\theta / q_0$ and $1 / r_1=(1-\theta) / p_1+\theta / q_1$, where $R=\left(r_0, r_1\right)$. We have the following analytic interpolation for the mixed norm spaces.

Let us consider two product measure spaces $X=X_0 \times X_1$ and $Y=Y_0 \times Y_1$. Let $d x$ and $d y$ denote respectively the (product) measures on $X$ and $Y$. Let $T_z $ be an analytic family of linear operators between $X$ and $Y$ of admissible growth, defined in the strip $\{ z \in \C : 0 \leq  \Re z \leq 1 \}$. We suppose that for all finite linear combinations of characteristic functions of rectangles of finite measures $f$ on $X$:
\begin{align*}
\left\|T_{j+ i \xi}(f)\right\|_{Q_j} \leq A_j(\xi)\|f\|_{P_j} 
\end{align*}
for $P_j, Q_j \in[1, \infty] \times[1, \infty]$ such that $\log \left|A_j(\xi)\right| \leq A e^{a|\xi|}, a<\pi, j=0, 1$. Let $1 / R=(1-\theta) / P_0+\theta / P_1$ and $1 / S=(1-\theta) / Q_0+\theta / Q_1, 0<\theta<1$. Then it follows similarly as in \cite[Page 313, Theorem 1]{BP61} (see also \cite[Theorem 1]{Ste56}) that
\begin{align}\label{eqn_ana_int}
\left\|T_\theta(f)\right\|_S \leq A_{\theta}\|f\|_R.
\end{align}
We refer the reader to \cite{BP61} for more details about the mixed norm spaces. However, we would like to mention that the order in which the mixed norms are taken in this article differs from that in \cite{BP61}. Following the calculation as a mixed norm space version of the  Stein-Weiss interpolation theory \cite[Theorem 2]{Ste56}, which can be derived using \eqref{eqn_ana_int}; see \cite[Corollary 7.3]{RR24} for more details.  
\begin{lemma}\label{cor_sw_ana}  
Let $1 \leq p, q_j, \tilde{q}_j \leq \infty$, and let $T$ be a linear operator defined on simple functions on the product space $X = X_0 \times X_1$, mapping into measurable functions on $Y = Y_0 \times Y_1$. Suppose ${w}_j: Y_0 \rightarrow \mathbb{R}^+$ and ${v}_j: X_0 \rightarrow \mathbb{R}^+$ are non-negative weight functions that are integrable over every subset of finite measure, for $j = 0, 1$. We further assume that for all functions $f$ that are finite linear combinations of characteristic functions of rectangles with finite measure in $X$, the following holds
\begin{align}\label{hyp_T_z_SW}
\|T (f)\cdot { w}_j \|_{L^{Q_j} (Y, dy)} \leq A_j \|f\cdot v_j\|_{L^{P_j}(X,dx)}
\end{align}
where $j = 0, 1$, and $ Q_j = (2,q_j),$ and $ P_j=(p_j,p_j).$ Then we have 
     \begin{align*}
        \|T (f)\cdot  {w}_{\theta}\|_{L^{Q_{\theta}} (Y, dy)} \leq A_{\theta} \|f\cdot v_{\theta} \|_{L^{P_\theta}(X, dx)}
    \end{align*}
    for all simple functions $f$ of finite measure support, $0<\theta<1$, where $Q_{\theta}= (2,q_{\theta})$, $P_{\theta}=(p_{\theta},p_{\theta})$ with 
    \begin{equation}\label{theta_exp_again}
    \begin{aligned}
       &\frac{1}{p_\theta}=  \frac{1-\theta}{p_0}+\frac{\theta}{p_1}, \qquad \quad \quad &&\frac{1}{q_{\theta}}= \frac{1-\theta}{q_0}+\frac{\theta}{q_1},\\
       & v_{\theta} = v_0^{ (1-\theta)} v_1^{ \theta} , \qquad \quad    && {w}_{\theta} = w_0^{ (1-\theta)} w_1^{ \theta} .
       \end{aligned}
       \end{equation}
\end{lemma} 
\section{Shifted Pitt's inequality on noncompact type symmetric spaces}\label{sec_pitt_X}
In this section, we prove Theorem \ref{thm_pitt_X_n}. We begin by establishing the following analogue of the classical Hardy–Littlewood–Paley inequality, as stated in \cite[Theorem 4.3]{Sad79}.
\begin{lemma}[Hardy-Littlewood-Paley's inequality]\label{lem_plaey_inq}
        Let $\X$ be a symmetric space of dimension $n\geq 2$, pseudo-dimension $\nu \geq 3$, and $1<p\leq 2$. Suppose that $ \zeta \geq 0, \sigma>0$, then for all $f \in C_c^{\infty}(\X)$ we have the  following inequality 
  \begin{align}\label{eqn_inq_paley_lem} 
\left( \int_{\mathbb{\fa}} \left(\int_B|\widetilde{f}(\lambda, b)|^{2} \, db\right)^{\frac{p}{2}} (|\lambda|^2+\zeta^2)^{-\frac{\sigma p}{2}} |\bfc(\lambda)|^{-2} \, d\lambda \right)^{\frac{1}{p}}
\leq C_p \left( \int_{\X}|f(x)|^p \,dx \right)^{\frac{1}{p}},
\end{align}
which holds for  $\zeta > 0$ if $\sigma \geq n(2/p-1)$, and  for $\zeta = 0$ if $n (2/p-1) \leq \sigma \leq \nu(2/p-1)$.

\end{lemma}
\begin{proof}
     For a fixed $\zeta\geq 0$, let us consider $u_{\sigma, \zeta}(\lambda):=(|\lambda|^2+\zeta^2)^{-\frac{\sigma}{2}}$, for $\sigma > 0$. 
  We will determine for which values of $\sigma$, $u_{\sigma,\zeta}$ will satisfy the hypothesis of Theorem \ref{thm_paley_uni}. First, we consider the case $\zeta=0$. The distribution function of $u_{\sigma,0}= u_\sigma$ with respect to the Plancherel measure $d\mu (\lambda):=|\bfc(\lambda)|^{-2} d\lambda$ is given by
 \begin{align*}
    \hspace{1.2cm} d_{u_{\sigma}}(\gamma)=  \mu\{\lambda \in \fa:   |\lambda|^{-\sigma} >\gamma\} = \mu\{\lambda \in \fa:   |\lambda|^{\sigma} < 1/\gamma\} = \mu\{\lambda \in \fa:   |\lambda| < \left( 1/\gamma\right)^{\frac{1}{\sigma}}\}.
 \end{align*}
 When $1<\gamma<\infty$, we can write using  \eqref{est_c-2,hr}
\begin{align*}
      d_{u_{\sigma}}(\gamma) \leq C \int_{0}^{\left( 1/\gamma\right)^{\frac{1}{\sigma}}}  s^{\nu-l} s^{l-1} \,ds \leq C \gamma^{-\frac{\nu}{\sigma}}.
\end{align*}
 For $0< \gamma\leq 1 $,  we have
 \begin{align*}
     d_{u_{\sigma}}(\gamma)\asymp \int_{0}^{1} s^{\nu -1}    ds +  \int_{1}^{\left( 1/\gamma\right)^{\frac{1}{\sigma}}} s^{n-l} s^{l-1} \,ds \leq C {\gamma^{-\frac{n}{\sigma}}}.
 \end{align*}
Thus, utilizing the estimates above, we obtain that for $\zeta=0$, if $n \leq \sigma\leq \nu$, then
\begin{align}\label{inq_u,L1}
     \sup_{\gamma>0} \gamma  \int\limits_{\substack{\{\lambda \in \fa : u_{\sigma, \zeta}(\lambda)>\gamma\} }} |\bfc(\lambda)|^{-2} \,d\lambda <\infty.
\end{align}  
For $\zeta > 0$, we observe that when $\gamma$ is sufficiently large, $d_{u_{\sigma, \zeta}}(\gamma)$ vanishes. More precisely, if $\gamma \geq (1/\zeta^2)^{\sigma/2}$, then 
\begin{align*}
    d_{u_{\sigma, \zeta}}(\gamma)=  \mu\{\lambda \in \fa:   {(|\lambda|^2+\zeta^2)}^{-\frac{\sigma}{2}} >\gamma\}  = \mu\{\lambda \in \fa:   |\lambda|^2 < \left( 1/\gamma\right)^{\frac{2}{\sigma}}-\zeta^2\}=\mu\{\emptyset \} \equiv 0.
\end{align*} 
Consequently, by performing a similar computation, we deduce that inequality \eqref{inq_u,L1} holds for $\zeta > 0$ whenever $\sigma \geq n$. Therefore, applying Theorem \ref{thm_paley_uni}, we obtain the following inequality for $1 < p \leq 2$:
\begin{align*} 
\left( \int_{\mathbb{\fa}} \left(\int_B|\widetilde{f}(\lambda, b)|^{2} \, db\right)^{\frac{p}{2}} (|\lambda|^2+\zeta^2)^{-\frac{\sigma p}{2}} |\bfc(\lambda)|^{-2} \, d\lambda \right)^{\frac{1}{p}}
\leq C \left( \int_{\X}|f(x)|^p \,dx \right)^{\frac{1}{p}}.
\end{align*}
This inequality holds in the case $\zeta = 0$ provided that $\sigma$ satisfies $n(2/p - 1) \leq \sigma \leq \nu(2/p - 1)$, and in the case $\zeta \neq 0$ whenever $\sigma \geq n(2/p - 1)$.
\end{proof}
\begin{remark}\label{rem_prob_lam_0}
\begin{enumerate}
    \item We note that unless $\nu \geq n$, it may not be possible to find any $\sigma \geq 0$ for which the Hardy–Littlewood–Paley inequality \eqref{eqn_inq_paley_lem} holds with $\zeta = 0$. As justification, we observed in the rank one case (see Remark \ref{rem_after_J_pit}) that \eqref{eqn_inq_paley_lem} may not hold for all $1 < p \leq 2$ when $\zeta = 0$, unless the condition $n \leq \nu$ is satisfied.

Furthermore, in higher rank cases, we recall from Corollary \ref{cor_pitt_KRZ,0} that in order for unshifted Pitt's inequality \eqref{inq_pit_p,2} to hold with $q = 2$ and $1 \leq p < 2$, it is necessary that $0 \leq \kappa \leq \sigma < \nu/2$, $\kappa p' < n$, and $n(1/p - 1/2) \leq \sigma - \kappa$. These conditions together imply $n(1/p - 1/2) < \nu/2$. From this, it follows that if $n > \nu$, then for some $p$ sufficiently close to 1, inequality \eqref{inq_pit_p,2} may fail. Therefore, to ensure that the unshifted Pitt's inequality holds for all $1 < p \leq q < \infty$, the condition $n \leq \nu$ is essential.  This justifies our assumption $\min\{n, \nu\} = n$ in Theorem \ref{thm_pitt_X_n} for the case $\zeta = 0$, and ensures the existence of some $\sigma \geq 0$ for which \eqref{eqn_inq_paley_lem} holds.
\item We recall that one of the most distinctive features of symmetric spaces of noncompact type is the holomorphic extension of the Fourier transform. This motivates the study of Fourier inequalities not merely on $\fa$, but over the entire domain of holomorphic extension, something without a direct analogue in $\R^N$.  Particularly, in the rank one case, it is natural to investigate the Hardy–Littlewood–Paley inequality (see Lemma \ref{lem_plaey_inq}) for $L^p$-functions ($1\leq p<2$)  whose Fourier transform is defined on the tube domain $S_p = \fa + i\mathrm{co}(W \cdot \rho_p)$, where $\rho_p := |2/p - 1||\rho|$.

In this direction, using the Paley inequality on non-unitary duals from \cite[Theorem 1.1, Remark 5.2 (2)]{RR24} and a similar argument as in Lemma \ref{lem_plaey_inq}, we obtain the following version of the Hardy–Littlewood–Paley inequality on non-unitary duals in rank one cases: Let $\X$ be a rank one symmetric space of dimension $n \geq 2$, and let $1 < p \leq 2$.  Then for any $q_0$ with $p \leq q_0 \leq p'$, we have the  following inequality for all $f \in C_c^\infty(\X)$ 
  \begin{align*}
\left( \int_{\mathbb{\fa}} \left( \int_B  \left|\widetilde {f}(\lambda +i\rho_{q_0} , b) \right|^{q_0}   \,  db \right)^{\frac{p}{q_0}} (|\lambda|^2+\zeta^2)^{-\frac{\sigma p}{2}} |\bfc(\lambda)|^{-2} \, d\lambda \right)^{\frac{1}{p}}
\leq C_p \left( \int_{\X}|f(x)|^p \,dx \right)^{\frac{1}{p}},
\end{align*}
which holds for  $\zeta > 0$ if $\sigma >n(2/p-1)$, and  for $\zeta = 0$ if $n (2/p-1) < \sigma < 3(2/p-1)$. We also note that $\nu = 3$ for all rank one symmetric spaces.
\end{enumerate}
\end{remark}
We require the following lemma, which follows from \cite[Theorem 1.2]{KRZ_23} by taking $\beta = 0$, $q = 2$, and $\zeta > 0$, and then applying the Plancherel theorem.
\begin{lemma}\label{cor_pitt_KRZ>0}
     Let $\X$ be symmetric space of dimension $n\geq 2$ and $ \zeta > 0$. Suppose that $\sigma >0$ and $f \in C_c^{\infty}(\X)$. Then, the inequality 
    \begin{align}\label{inq_pit_p,2_n}
     \left(   \int_{\fa} \int_{B}|\widetilde{f}(\lambda, b)|^2 (|\lambda|^2+\zeta^2)^{-{\sigma}} |\bfc(\lambda)|^{-2} \, db\, d\lambda \right)^{\frac{1}{2}} \leq C \left( \int_{\X}|f(x)|^p |x|^{\kappa p}\,dx \right)^{\frac{1}{p}}
    \end{align}
    holds for $1\leq  p\leq 2$ if and only if  $0\leq \kappa\leq \sigma$, $\kappa<{n}/{p'}$ (when $p=1$, $\kappa=0$), and $\sigma-\kappa \geq n\left( {1}/{p}-{1}/{2}\right)$. 
\end{lemma}
With the Hardy–Littlewood–Paley inequality and the preceding lemma in place, we are now prepared to prove one of our main results on shifted Pitt's inequality on noncompact type symmetric spaces. Our approach is inspired by the mixed-norm version (Lemma \ref{cor_sw_ana}) of Stein’s interpolation theory developed in \cite{Ste56}. However, a key novelty of our method lies in how we handle the presence of shifted weights on the Fourier transform side. The shifted structure in our setting necessitates a careful modification of the interpolation framework. To accommodate this, we adapt the interpolation argument by introducing a modified inequality that captures the interaction between the weight and the spectral shift. This refinement is crucial for establishing the desired Pitt-type inequalities in our setting.

\begin{proof}[\textbf{Proof of Theorem \ref{thm_pitt_X_n}}]
We will prove the theorem for the case $p > 1$; when $p = 1$, the parameter $\kappa$ equals zero, and the theorem can be established using a similar argument. So, let us assume $p > 1$ and recall the Hardy–Littlewood–Paley inequality from \eqref{eqn_inq_paley_lem}, which we now rewrite for $1 < r \leq 2$
\begin{align}\label{inq_paley_pol} 
\left( \int_{\mathbb{\fa}} \left(\int_B|\widetilde{f}(\lambda, b)|^{2} \, db\right)^{\frac{r}{2}} (|\lambda|^2+\zeta^2)^{-\frac{\delta r}{2}} |\bfc(\lambda)|^{-2} \, d\lambda \right)^{\frac{1}{r}}
\lesssim \left( \int_{\X}|f(x)|^r \,dx \right)^{\frac{1}{r}},
\end{align}
which holds for $\zeta = 0$ if $n (2/r-1) \leq \delta \leq \nu(2/r-1)$, and for $\zeta > 0$ if $\delta \geq n(2/r-1)$. \par
     First, we will handle the case when $\zeta \neq 0$. For this, we recall from \eqref{inq_pit_p,2_n} that
 \begin{align}\label{inq_pit_p,2_rec>0}
     \left(   \int_{\fa} \int_{B}|\widetilde{f}(\lambda, b)|^2 (|\lambda|^2 +\zeta^2)^{-{\beta} } |\bfc(\lambda)|^{-2} \, db\, d\lambda \right)^{\frac{1}{2}} \lesssim  \left( \int_{\X}|f(x)|^2 |x|^{2\alpha }\,dx \right)^{\frac{1}{2}}.
    \end{align} 
        holds for if   $0\leq \alpha\leq \beta$ and $\alpha<{n}/{2}$.
      \par
    \textbf{Case I.} $1<p=q\leq 2$.   
    We choose $\theta$ such that 
    $1/p= (1-\theta)/r +\theta/2$, where we will vary $r$ between $1$ to $p$. Applying Lemma \ref{cor_sw_ana} with the following
    \begin{align*}
        p_0=q_0=r,\,  p_1=q_1=2, \quad v_0 =1,v_1=|x|^{\alpha}, \quad w_0=(|\lambda|^2+\zeta^2)^{-\frac{\delta}{2}} , w_1=(|\lambda|^2+\zeta^2)^{-\frac{\beta}{2}},
    \end{align*}
    and interpolating \eqref{inq_paley_pol}, \eqref{inq_pit_p,2_rec>0}, we obtain
    \begin{align}\label{pitt_p=q<2,X}
          \left(   \int_{\fa} \left(\int_{B}|\widetilde{f}(\lambda, b)|^2 db\right)^{\frac{p}{2}} (|\lambda|^2 +\zeta^2)^{-\frac{\sigma p}{2}} |\bfc(\lambda)|^{-2} \, d\lambda \right)^{\frac{1}{p}} \lesssim \left( \int_{\X}|f(x)|^p |x|^{\kappa p }\,dx \right)^{\frac{1}{p}}
    \end{align}
    where 
    \begin{align*}
        \kappa= \theta \alpha, \qquad \sigma =(1-\theta) \delta +\theta \beta.
    \end{align*}
    By putting the value $\theta = (1/r-1/p)/(1/r-1/2)$ and $(1-\theta) = (1/p-1/2)/(1/r-1/2)$, the equation above transforms into 
  \begin{align*}
        \kappa=   \frac{\left( \frac{1}{r}-\frac{1}{p} \right)}{ \left( \frac{1}{r}-\frac{1}{2} \right)} \alpha, \qquad \sigma = \frac{\left( \frac{1}{p}-\frac{1}{2} \right)}{ \left( \frac{1}{r}-\frac{1}{2} \right)} \delta+  \frac{\left( \frac{1}{r}-\frac{1}{p} \right)}{ \left( \frac{1}{r}-\frac{1}{2} \right)} \beta .
    \end{align*}
    Since $\alpha \in [0,n/2)$, and we can choose $r \in (1,p)$ arbitrarily, so \eqref{pitt_p=q<2,X} holds for  all  $0 \leq \kappa <n/{p'} $. Additionally, we observe that since $\beta \geq \alpha$, and \eqref{inq_paley_pol}  holds for all $\delta \geq n(2/r-1) $, therefore, we have shown that  \eqref{pitt_p=q<2,X} holds for all $\kappa$ and $\sigma$ such that 
    \begin{align*}
        0\leq \kappa < \frac{n}{p'}, \quad \text{and} \quad \sigma \geq n\left(\frac{2}{p}-1 \right)+\kappa.
    \end{align*}
\par 
    \textbf{Case II.} $1<p\leq 2 <q=p'$. We choose $\theta$ such that $1/p= (1-\theta)+\theta/2$.\\
   \noindent We have the following restriction inequality from \eqref{f_L2K<L1_pre}
    \begin{align}\label{f_L2K<L1}
        \left( \int_B  \left| \widetilde{f}(\lambda,b)\right|^2 \,db \right)^{\frac{1}{2}}  & \leq  \|f\|_{L^1(\X)}.
    \end{align}
 We now apply Lemma \ref{cor_sw_ana} with 
    \begin{align*}
        p_0=1, q_0=\infty,\,  p_1=q_1=2, \quad v_0 =1,v_1=|x|^{\alpha}, \quad w_0=1, w_1=(|\lambda|^2+\zeta^2)^{-\frac{\beta}{2}},
    \end{align*}
    and interpolating \eqref{inq_pit_p,2_rec>0} with \eqref{f_L2K<L1}, we obtain
     \begin{align}\label{pitt_p<2<p',X}
          \left(   \left(\int_{\fa} \int_{B}|\widetilde{f}(\lambda, b)|^2  \,db \right)^{\frac{p'}{2}}(|\lambda|^2 +\zeta^2)^{-\frac{\gamma p'}{2}} |\bfc(\lambda)|^{-2} \, d\lambda \right)^{\frac{1}{p'}} \lesssim \left( \int_{\X}|f(x)|^p |x|^{\kappa p }\,dx \right)^{\frac{1}{p}},
    \end{align}
    where 
    \begin{align*}
        \kappa=   \frac{2\alpha}{p'}, \qquad \gamma =\frac{2\beta}{p'}.
    \end{align*}
    Since \eqref{inq_pit_p,2_rec>0} holds for all  $0\leq \alpha\leq \beta$ and $\alpha<{n}/{2}$, thus we have shown that \eqref{pitt_p<2<p',X} holds for all  $\kappa$ and $\gamma$ such that 
    \begin{align*}
        0\leq \kappa < \frac{n}{p'}, \quad \text{and} \quad \gamma \geq \kappa,
    \end{align*}
    where $1\leq p\leq 2$.
         \par
    
    \textbf{Case III.} $1<p\leq 2, p\leq q \leq p'$. We can rewrite \eqref{pitt_p=q<2,X} as 
       \begin{align}\label{pitt_p=q<2,X_nw}
          \left(   \int_{\fa} \left( \int_{B}|\widetilde{f}(\lambda, b)|^2 db\right)^{\frac{p}{2}} (|\lambda|^2 +\zeta^2)^{-\frac{\vartheta p}{2}} |\bfc(\lambda)|^{-2} \, d\lambda \right)^{\frac{1}{p}} \lesssim  \left( \int_{\X}|f(x)|^p |x|^{\kappa p }\,dx \right)^{\frac{1}{p}},
    \end{align}
    which holds for $1<p\leq 2$, $0\leq \kappa<n/{p'}$, and $\vartheta \geq n( {2}/{p}-1)+\kappa.$ Let us choose $\theta$ such that $1/q= (1-\theta)/{p}+ \theta/{p'}$. Applying Lemma \ref{cor_sw_ana} with the following
    \begin{align*}
        p_0= q_0= p,\,  p_1=p, q_1=p', \quad v_0 =|x|^{\kappa},v_1=|x|^{\kappa}, \quad w_0=(|\lambda|^2+\zeta^2)^{-\frac{\vartheta}{2}} , w_1=(|\lambda|^2+\zeta^2)^{-\frac{\gamma}{2}},
    \end{align*}
    and interpolating \eqref{pitt_p=q<2,X_nw}, \eqref{pitt_p<2<p',X}, we get
    \begin{align}\label{pitt_p<2,p<q,X}
          \left(   \int_{\fa} \left( \int_{B}|\widetilde{f}(\lambda, b)|^2 \,db\right)^{\frac{q}{2}} (|\lambda|^2 +\zeta^2)^{-\frac{\sigma q}{2}} |\bfc(\lambda)|^{-2} \, d\lambda \right)^{\frac{1}{q}} \lesssim \left( \int_{\X}|f(x)|^p |x|^{\kappa p }\,dx \right)^{\frac{1}{p}},
    \end{align}
    where 
    \begin{align*}
        0\leq \kappa< \frac{n}{p'}, \qquad \sigma =(1-\theta) \vartheta +\theta  \gamma. 
    \end{align*}
  Now using the range of $ \vartheta$ and $\gamma $ (as in \eqref{pitt_p=q<2,X_nw}, \eqref{pitt_p<2<p',X}) and putting the value of $\theta = (1/p-1/q)/(2/p-1) $, we can say that the inequality \eqref{pitt_p<2,p<q,X} holds for all $0\leq \kappa <n/p'$ and for all $\sigma$ such that
  \begin{align}
      \sigma\geq (1-\theta)  \left( n\left(\frac{2}{p}-1 \right)+\kappa\right)+\theta  \kappa = n\left(  \frac{1}{p}+\frac{1}{q}-1\right) +\kappa. 
  \end{align}
  This concludes the sufficient part of Theorem~\ref{thm_pitt_X_n} for $\zeta>0$.

    \vspace{3mm}
    Now we will handle the case $\zeta=0$.  In this case, we assume that $\nu\geq n$.  Then we  recall the Hardy-Littlewood-Paley's inequality from \eqref{eqn_inq_paley_lem}  for $1<r\leq 2$
\begin{align}\label{inq_paley_pol_0} 
\left( \int_{\mathbb{\fa}} \left(\int_B|\widetilde{f}(\lambda, b)|^{2} \, db\right)^{\frac{r}{2}} |\lambda|^{-{\delta r}} |\bfc(\lambda)|^{-2} \, d\lambda \right)^{\frac{1}{r}}
\lesssim \left( \int_{\X}|f(x)|^r \,dx \right)^{\frac{1}{r}},
\end{align}
which holds for $\zeta = 0$ if $n (2/r-1) \leq \delta \leq \nu(2/r-1)$.  We have from  \eqref{inq_pit_p,2} that
 \begin{align}\label{inq_pit_p,2_0}
     \left(   \int_{\fa} \int_{B}|\widetilde{f}(\lambda, k)|^2 |\lambda|^{-2\alpha} |\bfc(\lambda)|^{-2} \, db\, d\lambda \right)^{\frac{1}{2}} \lesssim \left( \int_{\X}|f(x)|^2 |x|^{2\alpha }\,dx \right)^{\frac{1}{2}}.
    \end{align} 
        holds for if and only if $\alpha<n/2$. Next, we will follow a similar strategy as in the case $\zeta \neq 0$.

        \textbf{Case I.} $1<p=q\leq 2$.  
   By choosing $\theta$ such that 
    $1/p= (1-\theta)/r +\theta/2$, and following a similar calculation as in \eqref{pitt_p=q<2,X}, except we make use of \eqref{inq_paley_pol_0} and \eqref{inq_pit_p,2_0} to obtain 
    \begin{align}\label{pitt_p=q<2,X_0}
          \left(   \int_{\fa} \left(\int_{B}|\widetilde{f}(\lambda, b)|^2 db\right)^{\frac{p}{2}} |\lambda| ^{-\sigma p} |\bfc(\lambda)|^{-2} \, d\lambda \right)^{\frac{1}{p}} \lesssim \left( \int_{\X}|f(x)|^p |x|^{\kappa p }\,dx \right)^{\frac{1}{p}},
    \end{align}
    where 
    \begin{align*}
        \kappa= \theta \alpha, \qquad \sigma =(1-\theta) \delta +\theta \alpha.
    \end{align*}
    By putting the value $\theta = (1/r-1/p)/(1/r-1/2)$, and observing that we can choose any $r\in (1,\infty) $, we can derive that \eqref{pitt_p=q<2,X_0} holds for all $\kappa$ and $\sigma$ such that 
    \begin{align}
        0\leq \kappa < \frac{n}{p'} \quad \text{and} \quad n\left(\frac{2}{p}-1 \right)\leq  \sigma-\kappa \leq \nu\left(\frac{2}{p}-1 \right).
    \end{align}
  \par
   \textbf{Case II.} $1<p\leq 2 <q=p'$. Following the exact same calculation as in \eqref{pitt_p<2<p',X}, except that we use \eqref{inq_pit_p,2_0} instead of \eqref{inq_pit_p,2_rec>0} to obtain the following 
    \begin{align}\label{pitt_p<2<p',X_0}
          \left(  \int_{\fa} \left( \int_{B}|\widetilde{f}(\lambda, b)|^2  db \right)^{\frac{p'}{2}}|\lambda|^{-\kappa p'} |\bfc(\lambda)|^{-2} \, d\lambda \right)^{\frac{1}{p'}} \lesssim \left( \int_{\X}|f(x)|^p |x|^{\kappa p }\,dx \right)^{\frac{1}{p}}
    \end{align}
    holds for all $ 0\leq \kappa < {n}/{p'}$.

    \textbf{Case III.} $1<p\leq 2, p\leq q\leq p'$. We can rewrite \eqref{pitt_p=q<2,X_0} as 
       \begin{align}\label{pitt_p=q<2,X_nw_0}
          \left(   \int_{\fa} \left( \int_{B}|\widetilde{f}(\lambda, b)|^2 \, db\right)^{\frac{p}{2}} |\lambda|^{-\vartheta p} |\bfc(\lambda)|^{-2} \, d\lambda \right)^{\frac{1}{p}} \lesssim \left( \int_{\X}|f(x)|^p |x|^{\kappa p }\,dx \right)^{\frac{1}{p}},
    \end{align}
    which holds for $1<p\leq 2$, $0\leq \kappa<n/{p'}$, and $n( {2}/{p}-1) \leq \vartheta-\kappa \leq \nu( {2}/{p}-1).$ Now by choosing $\theta$ such that $1/q= (1-\theta)/{p}+ \theta/{p'}$ and interpolating between \eqref{pitt_p=q<2,X_nw_0} and \eqref{pitt_p<2<p',X_0}, we obtain
    \begin{align}\label{pitt_p<2,p<q,X_0}
          \left(   \int_{\fa} \left( \int_{B}|\widetilde{f}(\lambda, k)|^2 \,db\right)^{\frac{q}{2}} |\lambda|^{-{\sigma q}} |\bfc(\lambda)|^{-2} \, d\lambda \right)^{\frac{1}{q}} \lesssim \left( \int_{\X}|f(x)|^p |x|^{\kappa p }\,dx \right)^{\frac{1}{p}},
    \end{align}
    where 
    \begin{align*}
        0\leq \kappa< \frac{n}{p'}, \qquad \sigma =(1-\theta) \vartheta +\theta  \gamma. 
    \end{align*}
  Now using the range of $ \vartheta$ and $\gamma $  and putting the value of $\theta = (1/p-1/q)/(2/p-1) $, we can say that the inequality \eqref{pitt_p<2,p<q,X_0} holds for all $0\leq \kappa <n/p'$ and for all $\sigma$ such that
  \begin{align}
      n\left(  \frac{1}{p}+\frac{1}{q}-1\right) \leq \sigma -\kappa \leq  \nu\left(  \frac{1}{p}+\frac{1}{q}-1\right). 
  \end{align}
This concludes the proof of the sufficiency part of Theorem~\ref{thm_pitt_X_n} for all $\zeta \geq 0$.

The necessity part of Theorem~\ref{thm_pitt_X_n} in the rank one case follows from Theorem~\ref{thm_nec_pitt_n}, which will be established in Section~\ref{sec_jac_transform}
\end{proof}
\begin{remark}
   To address the case $q > p'$, Stein \cite{Ste56} used the Hardy–Littlewood–Paley inequality for $2 < p = q < \infty$ with $\sigma = 0$ in \eqref{Clss_Pitt_Rn_pw}. In the setting of rank one symmetric spaces of noncompact type, we have shown that establishing a similar inequality (for $2 < p = q < \infty$ and $\sigma = 0$ in \eqref{inq_pitt_pol}) requires the introduction of additional exponential weights on the function side. More generally, Theorem~\ref{thm_nec_pitt_n} shows that for the shifted Pitt's inequality \eqref{inq_pitt_pol} to hold, a necessary condition is $p \leq 2$, due to the exponential volume growth of these spaces. 
   Since the present work focuses on Pitt-type inequalities with polynomial weights, consistent with the classical framework, we leave the investigation of exponential-weighted versions and the case $q > p'$ on symmetric spaces to future research.
\end{remark}

\section{Shifted Pitt's inequality for Jacobi transforms}\label{sec_jac_transform}
We now turn our focus to shifted Pitt's inequality in the setting of the Jacobi transform, which generalizes spherical analysis on rank one symmetric spaces \cite{Koo84}. Unlike the higher rank case, the Jacobi functions that define the Fourier transform in this context exhibit a rich set of special properties. We leverage these features to establish not only sufficient conditions but also necessary ones for the validity of the shifted Pitt's inequality. In contrast to Theorem \ref{thm_pitt_X_n}, where our method relied on mixed-norm interpolation techniques, the proof here follows a different strategy tailored specifically to the Jacobi setting and adapted to handle the range $q >p'$, which was not accessible in the symmetric space case.

 We begin by reviewing some general facts and necessary preliminaries regarding Jacobi transforms. Most of this information is already known and can be found in \cite{FJT73,Koo75,Koo84, GLT18}.
 %; see also \cite{FJT79, BP97, BP08}. 
 We will use the standard notation from \cite{Koo84, GLT18}. To keep the article self-contained, we will include only the results necessary for this article, following the preliminaries outlined in \cite{Koo84,GLT18}.
\subsection{Basic properties of Jacobi transforms}\label{sec_jac_base} The Jacobi functions are defined by
\begin{align*}
    \varphi_{\lambda}^{(\alpha, \beta)}(t) =\, _2F_1\left( \frac{\rho+i\lambda}{2},\frac{\rho-i\lambda}{2};\alpha
    +1; -(\sinh t)^2\right), \quad t\geq 0,\, \alpha, \beta, \lambda\in \C,
\end{align*}
where $\rho= \alpha+\beta+1$ and $_2F_1(a,b;c;z) $ is the hypergeometric Gauss function defined as in \cite[(2.1)]{Koo84}. We consider the case
\begin{align}\label{eqn_rel_a,b}
    \alpha\geq \beta \geq -\frac{1}{2}, \quad \alpha>-\frac{1}{2}.
\end{align}
Equivalently, the Jacobi function $\varphi_{\lambda}^{(\alpha,\beta)}$ can be defined as the unique even $C^{\infty}$ function on $\R$ such that $\varphi(0)=1$ and 
\begin{align*}
   \left( \mathcal{L}_{\alpha, \beta} +\lambda^2+\rho^2 \right)\varphi=0,
\end{align*}
where 
\begin{align}\label{defn_of_L_Jac}
   \mathcal{L}_{\alpha, \beta}:= \frac{d^2}{dt^2} +(2\alpha+1)\coth t+(2\beta+1)\tanh t)\frac{d}{dt}.
\end{align}
We have the following estimate of $\varphi_{0}^{(\alpha,\beta)}$ (see
 \cite[(3.5)]{ADY96} and \cite[Prop. 8.1]{GLT18} for instance):
\begin{align}\label{eqn_est_phi_0}
\varphi_0^{(\alpha,\beta)}(t) \asymp \left(1+t\right)e^{-\rho t}, \quad t\geq 0.
\end{align}
From now on, we fix $\alpha$ and $\beta$ satisfying \eqref{eqn_rel_a,b}, and for simplicity of notation, we will write $\varphi_{\lambda}$ instead of $\varphi_{\lambda}^{(\alpha,\beta)}$ and $\mathcal{L}$ instead of $\mathcal{L}_{\alpha,\beta}$.
% From now onward, we fix $\alpha, \beta$, which satisfy \eqref{eqn_rel_a,b} and to simplify notation we will write $\varphi_{\lambda}$ instead $\varphi_{\lambda}^{(\alpha,\beta)}$.
\subsubsection{Standard Jacobi transforms}   For suitable functions $f,F$ on $\R_+$, the direct Jacobi transform is defined by 
\begin{align*} 
\mathcal{J}f(\lambda) =  \int_0^{\infty} f(t) \varphi_{\lambda}^{(\alpha,\beta)}(t) m(t) \,dt, \quad \lambda \geq 0,
\end{align*}
and the inverse Jacobi transform defined by
\begin{align*}
    \mathcal{I}F(t) =  \int_0^{\infty} F(\lambda) \varphi_{\lambda}^{(\alpha,\beta)}(t) n(\lambda) d\lambda, \quad t\geq 0,
\end{align*}
where 
\begin{equation}\label{def_m(t)}
\begin{aligned}
     m(t)= (2\pi)^{-\frac{1}{2}} \Delta(t), \qquad & \Delta(t)= 2^{2\rho}(\sinh t)^{2\alpha+1}(\cosh t)^{2\beta+1},\\
     n(\lambda) =(2\pi)^{-\frac{1}{2}}|\bfc(\lambda)|^{-2}, \qquad& \bfc(\lambda)=\frac{2^{\rho-i\lambda} \Gamma(\alpha+1)\Gamma(i\lambda)}{\Gamma\left( \frac{\rho+i\lambda}{2}\right)\Gamma\left( \frac{\rho+i\lambda}{2}-\beta\right)}.
\end{aligned}
\end{equation}
The Jacobi transform $\mathcal{J}$ establishes a bijection between the space of even, infinitely differentiable functions with compact support and the space of even entire functions of exponential type exhibiting rapid decay; see \cite[Theorem 3.4]{Koo75}. It also extends to an isometric isomorphism between the spaces $L^2(dm(t))$ and $L^2(dn(\lambda))$, and satisfies the corresponding Plancherel theorems
\begin{align}\label{eqn_planc_J}
  \left(\int_0^{\infty}|f(t)|^2 m(t)\, dt \right)^{\frac{1}{2}} &=\left( \int_0^{\infty} |\mathcal{J}f(\lambda)|^2 n(\lambda) \, d\lambda \right)^{\frac{1}{2}},\\
  \label{eqn_planc_I}
     \left( \int_0^{\infty}|F(\lambda)|^2 n(\lambda) \, d\lambda \right)^{\frac{1}{2}}&= \left(\int_0^{\infty} |\mathcal{I}F(t)|^2 m(t)\, dt \right)^{\frac{1}{2}}.
\end{align}
When $\alpha = \beta$, the Jacobi transform is commonly referred to as the Mehler–Fock transform. In particular, when $\alpha = \beta = -{1}/{2}$, the Jacobi transform reduces to the classical cosine transform. For a detailed study of Pitt's inequality in the context of the cosine transform, we refer the reader to \cite[Sec. 7]{GLT18}.
\begin{remark}
\begin{enumerate}
    \item    We remark that the definition of $\mathcal{J}f(\lambda)$ coincides precisely with the spherical Fourier transform \eqref{eqn_sph_trans} when the parameters $\alpha$ and $\beta$ correspond to geometric cases. More precisely, as shown in \cite[(4.6)]{Koo84}, if $\X$ is a rank one symmetric space of noncompact type, then $\alpha = (n-1)/2$ and $\beta =(m_2-1)/2$,    where $m_2$ is a constant depending on  $\X$.
    \item In the Euclidean setting, the authors in \cite{BH03, He84} employed the method of non-increasing rearrangements to prove Pitt's inequality. However, in the Jacobi setting, this approach does not seem to be very effective for polynomial weights, as the exponential growth of the measure dominates and essentially ignores the polynomial decay at infinity (see Remark \ref{rem_after_J_pit} \eqref{rem_problem_exp_space}).  To address this limitation, we introduce a suitable weight into the Jacobi transform, thereby modifying the function space to reflect polynomial growth. More precisely, following Gorbachev et al. \cite{GLT18}, we define a modified Jacobi transform by adjusting the underlying measure. This leads to a setting with doubling volume growth, where the rearrangement method becomes applicable and more accurately captures the interplay between the geometry of the space and the behavior of polynomial weights.
    \end{enumerate}

\end{remark}
 \subsubsection{Modified Jacobi transforms}
The modified Jacobi transforms are defined by the relations
\begin{align*} 
\tilde{\mathcal{J}}f(\lambda) &=  \int_0^{\infty} f(t) \tilde{\varphi}_{\lambda}^{(\alpha,\beta)}(t) \tilde{m}(t) \,dt, \,\,\,\quad \lambda \geq 0,\\
    \tilde{\mathcal{I}}F(t) &=  \int_0^{\infty} F(\lambda) \tilde{\varphi}_{\lambda}^{(\alpha,\beta)}(t) n(\lambda) \, d\lambda, \quad t\geq 0,
\end{align*}
where
\begin{align*}
    \tilde{\varphi}_{\lambda}(t)=  \tilde{\varphi}_{\lambda}^{\alpha, \beta}(t):=\frac{\varphi_{\lambda}(t)}{\varphi_0(t)}, \qquad \tilde{m}(t)= \varphi_0^2(t) m(t).
\end{align*}
 It follows from \eqref{def_m(t)} and the estimate of $\varphi_0$ \eqref{eqn_est_phi_0}  that
\begin{align}\label{est_sharp_mtil}
    \tilde{m}(t)\asymp \begin{cases}
        t^{2\alpha+1} &\quad t\leq 1,\\
        t^{2} & \quad  t \geq 1.
    \end{cases}
\end{align}
On the other hand, we have from   \cite[(8.20)]{GLT18}
\begin{align}\label{est_sharp_n}
    n(\lambda)\asymp \begin{cases}
        \lambda^{2} &\quad \lambda\leq 1,\\
        \lambda^{2\alpha+1} & \quad  \lambda \geq 1. 
    \end{cases}
\end{align}
  We have the following relation  between the transforms
 \begin{align}\label{eqn_rel_J_mJ}
      \tilde{\mathcal{J}} f(\lambda) = \mathcal{J}(f\cdot \varphi_0)(\lambda), \qquad \tilde{\mathcal{I}}F(t) = \varphi_0(t)^{-1}\mathcal{I}F(t).
 \end{align}
 Also, we have from \eqref{eqn_planc_J} and \eqref{eqn_planc_I} that
 \begin{align}\label{eqn_plac_mJ,I}
     \|f\|_{L^2(\tilde{m})} = \|\mathcal{J}f\|_{L^2(n)}, \qquad  \|F\|_{L^2({n})} = \|\mathcal{I}F\|_{L^2(\tilde{m})}. 
 \end{align}
 The authors  \cite[Theorem 8.2]{GLT18} proved the following unshifted Pitt's inequality for piecewise power weights. However, we are going the state the result for the usual power weights. 
 %We refer the reader to Section \ref{sec_Jacobi} or \cite[Theorem 8.2]{GLT18}   for any unexplained notations.
\begin{theorem}[{{\cite[Theorem 8.2]{GLT18}}}]\label{thm_pitt_Jac}
 Suppose $1<p \leq q<\infty$.
 \begin{enumerate}
     \item[(A)] Let $\sigma$ and $\kappa$ be \textbf{dual} and $(\sigma, \kappa) \in \mathbf{D_{p, q}}(2\alpha+1, 2)$; then,
\begin{align}\label{Pitt_pow_J_int}
  \left( \int_0^{\infty} |\tilde{\mathcal{J}} f(\lambda)|^q  \lambda^{-\sigma q} {n}(\lambda)\, d\lambda \right)^{1/q} \lesssim   \left( \int_0^{\infty} | f(t)|^p  t ^{\kappa p} \Tilde{m}(t) \, dt \right)^{1/p}  
\end{align}
for the direct transform. If $({\sigma}, \kappa) \in \mathbf{D_{p, q}}(2, 2\alpha+1)$, then
\begin{align}\label{Pitt_pow_J-1_int}
\left( \int_0^{\infty} |\tilde{\mathcal{I}} F(t)|^q  t^{-\sigma q} \Tilde{m}(t)\, dt \right)^{1/q} \lesssim   \left( \int_0^{\infty} | F(\lambda)|^p  \lambda ^{\kappa p} {n}(\lambda) \, d\lambda \right)^{1/p}  
\end{align}
for the inverse one.
\item[(B)] Then, for the modified Jacobi transform $\tilde{\mathcal{J}}$, the following conditions
\begin{equation}
\begin{aligned}
2\left(\alpha+1\right)\left(\frac{1}{p}+\frac{1}{q}-1\right)\leq \sigma-\kappa  \leq 3\left(\frac{1}{p}+\frac{1}{q}-1\right),\\
\sigma \geq (2\alpha+1)\left(\frac{1}{q}-\frac{1}{2}\right) +\frac{1}{q}-\frac{1}{p^{\prime}}, \qquad \kappa \geq 2\left(\frac{1}{p^{\prime}}-\frac{1}{2}\right) +\frac{1}{p^{\prime}}-\frac{1}{q}
\end{aligned}
\end{equation}
are necessary for Pitt's inequality \eqref{Pitt_pow_J_int} to be valid. For the inverse modified Jacobi transform $\tilde{\mathcal{I}}$ the following conditions 
\begin{equation}
\begin{aligned}
3\left(\frac{1}{p}+\frac{1}{q}-1\right)\leq \sigma-\kappa  \leq 2\left(\alpha+1\right) \left(\frac{1}{p}+\frac{1}{q}-1\right), \\
\sigma  \geq 2 \left(\frac{1}{q}-\frac{1}{2}\right) +\frac{1}{q}-\frac{1}{p^{\prime}}, \qquad \kappa  \geq (2\alpha+1)\left(\frac{1}{p^{\prime}}-\frac{1}{2}\right) +\frac{1}{p^{\prime}}-\frac{1}{q},
\end{aligned}
\end{equation}
are necessary for Pitt's inequality \eqref{Pitt_pow_J-1_int}  to be valid.

\end{enumerate}
\end{theorem}
\begin{remark}
\begin{enumerate}
    \item We refer the reader to \cite[Def. 5.1]{GLT18} for the definition of the \textbf{dual}, and to \cite[p. 1974, Figure 5.2]{GLT18} for the explicit, but highly intricate, definition of the set $\mathbf{D_{p,q}}(v_1,v_2)$. Due to the numerous auxiliary conditions and technical subtleties involved, we do not reproduce the definition here. However, we emphasize that the sufficient condition derived from their approach, namely $(\sigma, \kappa) \in \mathbf{D_{p,q}}(2\alpha+1,2)$, does not coincide with the necessary condition; see \cite[p. 1981]{GLT18} for further discussion. One of the main objectives of this work is to refine both the sufficient and necessary conditions for Pitt-type inequalities so that they coincide.
    \item We also note that in \cite[Theorem 8.2]{GLT18}, the authors applied a general framework for spaces with polynomial growth to derive necessary conditions for modified Jacobi transforms. However, their analysis was limited to the unshifted Pitt’s inequality, involving Fourier weights of the form $\lambda^{-\sigma}$, and may not extend to the \textit{shifted} version with weights of the form $(\lambda^2 + \zeta^2)^{-\sigma/2}$ considered in our work. Furthermore, their necessary conditions for the modified Jacobi transform may not apply to the standard Jacobi transform. Motivated by these distinctions, we now proceed to establish sharper and more natural necessary conditions tailored to the setting of the standard Jacobi transform and the shifted Pitt’s inequality.

\end{enumerate}  
\end{remark}

\subsection{Necessary conditions for shifted Pitt's inequality for Jacobi transforms}
The exponential volume growth inherent in the Jacobi setting requires substantial modifications to the classical assumptions from the Euclidean case. This makes it crucial to carefully determine the precise limitations and range of validity for such inequalities. In Euclidean spaces, the existence of a dilation structure allows for simple scaling arguments to establish the necessity of the \textit{balance condition} \eqref{k-s=N_equiv}. However, no such dilation structure exists in the Jacobi setting. To address this, we will instead rely on a key property of the Jacobi functions. 
\begin{lemma}
 Suppose $0<\theta_0<\pi/2$ is a constant. Then, the following holds 
    \begin{align}\label{eqn_sph_low}
        \varphi_\lambda(t) \asymp_{\theta_0}  \varphi_0 (t) \quad \text{whenever } \lambda t \leq \theta_0.
    \end{align}
\end{lemma}
\begin{proof} By using Mehler's integral representation (see \cite[p. 47]{Koo84}, also \cite[(2.16)]{Koo75}), and using similar technique as in \cite[Prop. 4.2]{KPRS24}, the lemma follows.
\end{proof}
Let us consider the following class of operators for $\zeta \geq 0$ and $\kappa,\sigma \in \R$:
\begin{align*}
    \mathcal{J}^{\zeta}&= \left\{\mathcal{J}^{\zeta}_{\kappa,\sigma}f(\lambda) =(\lambda^2+\zeta^2)^{-\frac{\sigma}{2}}  \int_0^{\infty}t^{-\kappa}f(t) \varphi_{\lambda}^{(\alpha,\beta)}(t) m(t) \,dt, \quad f\in C_c^{\infty}(\R_+) \right\},\\
    \mathcal{I}^{\zeta}&= \left\{\mathcal{I}^{\zeta}_{\kappa,\sigma}F(t) =t^{-\sigma}  \int_0^{\infty}(\lambda^2+\zeta^2)^{-\frac{\kappa}{2}}F(\lambda) \varphi_{\lambda}^{(\alpha,\beta)}(t) n(\lambda) d\lambda, \quad F\in C_c^{\infty}(\R_+) \right\}.
\end{align*}
\begin{theorem}\label{thm_nec_pitt_n}
Let $1<p,q<\infty$ and $\kappa,\sigma \in \R$. Suppose the following Pitt-type inequality holds:
\begin{align}\label{eqn_pit_nec,n}
   \left( \int_0^{\infty} |\mathcal{J}^{\zeta}_{\kappa,\sigma}f(\lambda)|^q  n(\lambda) d\lambda\right)^{\frac{1}{q}}\leq C \left( \int_0^{\infty} |f(t)|^p m(t) dt \right)^{\frac{1}{p}}
\end{align}
for all $f \in C_c^{\infty}(\R^+)$. Then it is necessary that $p \leq 2$, and the following conditions hold
\begin{enumerate}
 \item \label{it_stanJ1} if $\zeta \neq 0$, then $\kappa p'<{2(\alpha+1)}$ and
     \begin{align}\label{nec_stan_s-k>}
            \sigma-\kappa \geq 2(\alpha+1)\left( \frac{1}{p}+\frac{1}{q}-1\right);
     \end{align}
     \item   if  $\zeta \neq 0$ and $p=2$, then $\kappa <{(\alpha+1)}$, \eqref{nec_stan_s-k>} and  $\kappa \geq 3(1/2-1/q)$. Additionally, if $p=q=2$, then $ 0\leq \kappa <{\alpha+1}$ and $\kappa \leq  \sigma$;
     \item if $\zeta=0$,   then $\kappa p'<{2(\alpha+1)}$,  $\sigma q<3$ and \eqref{nec_stan_s-k>};
     \item if  $\zeta=0$ and $p=2$, then $\kappa <{(\alpha+1)}$,   \eqref{nec_stan_s-k>}, $\sigma q<3$, and $\sigma-\kappa \leq 3(1/q-1/2) $. Additionally, if 
     $p=q=2$, then  $\kappa=\sigma<{\min\{2(\alpha+1),3\}}/{2}$.
\end{enumerate}
\end{theorem}
\begin{proof}
  First, we will show that the condition $\kappa<2(\alpha+1)/{p'} $ is necessary for all $\zeta\geq 0$. To that end, assume that the inequality \eqref{eqn_pit_nec,n} holds with $\kappa= 2(\alpha+1)/{p'}+\epsilon $ for some $\epsilon\geq 0$. By a standard density argument, it follows that the operator $\mathcal{J}^{\zeta}_{\kappa,\sigma}$ is bounded from $L^p(dm)$ to $L^q(dn)$. Now, consider the function 
\begin{align}
f_0(t) =  \Chi_{(0,1)}(t)  t^{-\frac{2(\alpha+1)}{p}} \quad t\in (0,\infty),
\end{align} 
where $\Chi_{(a,b)}$ is the characteristic function on the interval $(a,b)$. It is straightforward to verify that $f_0 \in L^p(dm)$. Then, for all $0 \leq \lambda \leq 1$, using \eqref{eqn_sph_low}, we can write: 
\begin{equation}\label{eqn_jf>c}
\begin{aligned}
(\lambda^2+\zeta^2)^{\frac{\sigma}{2}} \mathcal{J}^{\zeta}_{\kappa,\sigma} f_0(\lambda) &\geq \int_0^1 t^{-\kappa} f_0(t) \varphi_{\lambda}(t) m(t) \,dt \\ 
&\gtrsim \int_0^1 t^{-\frac{2(\alpha+1)}{p'}-\epsilon} t^{-\frac{2(\alpha+1)}{p}} \varphi_0(t)  t^{2\alpha+1} \,dt \\
& \gtrsim \int_0^1 t^{- 1-\epsilon} \,dt. 
\end{aligned}
\end{equation}
We observe that the integral in the last line above diverges for all $\epsilon\geq 0$, contradicting the boundedness of $\mathcal{J}_{\kappa,\sigma}$. Hence, we conclude that the condition $\kappa<2(\alpha+1)/{p'} $ is necessary. Assuming $\kappa<2(\alpha+1)/{p'} $, it follows from \eqref{eqn_jf>c} that for $\zeta=0$
\begin{align}\label{eqn_nec_sigm}
\int_0^{\infty} | \mathcal{J}^{0}_{\kappa,\sigma} f_0(\lambda) |^q n(\lambda) \, d\lambda \gtrsim  \int_0^1 \lambda^{-\sigma q} \lambda^2 \, d\lambda, 
\end{align} where, in the final line, we used the sharp estimate \eqref{est_sharp_n} of $n(\lambda)$ near $\lambda = 0$. We note that the integral on the right-hand side is finite only if $\sigma < 3/q$. It is worth mentioning that when $\zeta > 0$, the condition $\sigma q < 3$ may not be necessary.

Having established that $\kappa p' < 2(\alpha+1)$ is necessary for $\zeta \geq 0$, and that the additional condition $\sigma < 3/q$ is required when $\zeta = 0$, we now assume that inequality \eqref{eqn_pit_nec,n} holds for such values of $\kappa$ and $\sigma$. We define the following function, for $0 <s\leq  1$
\begin{align}
    f_1(t) =   t^{-\frac{\kappa}{p-1}} \varphi_{0}^{\frac{1}{p-1}}(t)\Chi_{(0,s)}(t).
\end{align}
 Since $\kappa p' < {2(\alpha+1)}$, the function $f_1$ is locally $L^p$-integrable, particularly near zero. We now bound the left-hand side of
\eqref{eqn_pit_nec,n} in the following way
 \begin{align*}
      \int_0^{\infty} | \mathcal{J}_{\kappa,\sigma} f_1(\lambda) |^q n(\lambda) \, d\lambda \geq   \int_1^{1/{s}}  \lambda^{-\sigma q}  \left| \int_0^{s}t^{-\kappa(1+\frac{1}{p-1})} \varphi_{0}^{\frac{1}{p-1}} \varphi_{\lambda}(t) m(t) \,dt \right|^q n(\lambda)d\lambda,
    \end{align*}
    where we used the fact $f_1$ is supported on the interval $(0,s)$. Using \eqref{eqn_sph_low}, the inequality above
    reduced to
 \begin{align}\label{eqn_jf1>}
      \int_0^{\infty} | \mathcal{J}_{\kappa,\sigma} f_1(\lambda) |^q n(\lambda) \, d\lambda \geq C  \left( \int_0^{s}t^{-\kappa p'} \varphi_{0}(t)^{p'} m(t) \,dt \right)^q  \left(\int_1^{{1}/{s}}  \lambda^{-\sigma q}  n(\lambda)d\lambda \right).
    \end{align}
By substituting $f = f_1$ into \eqref{eqn_pit_nec,n} and combining it with the inequality preceding \eqref{eqn_jf1>}, we derive
\begin{align*}
  \left( \int_0^{s}t^{-\kappa p'} \varphi_{0}(t)^{p'} m(t) \,dt \right)  \left( \int_1^{{1}/{s}}  \lambda^{-\sigma q}   n(\lambda)\,d\lambda\right)^{\frac{1}{q}}\leq C \left( \int_0^{s}t^{-\kappa p'} \varphi_{0}(t)^{p'} m(t) \,dt \right)^{\frac{1}{p}}.
\end{align*}
Since the constant $C$ is independent of $s$ and was chosen arbitrarily, we can conclude that
\begin{align}\label{eqn_pitt_nec1}
    \sup_{0<s\leq 1} \left( \int_0^{s}t^{-\kappa p'} \varphi_{0}(t)^{p'} m(t) \,dt \right)^{\frac{1}{p'}}  \left( \int_1^{{1}/{s}}  \lambda^{-\sigma q}   n(\lambda)\,d\lambda\right)^{\frac{1}{q}}<\infty.
\end{align}
Similarly for $s>1$, by considering 
\begin{align}\label{eqn_defn_f2}
     f_2(t) =   t^{-\frac{\kappa}{p-1}} \varphi_{0}^{\frac{1}{p-1}}(t)\Chi_{(1,s)}(t),
\end{align}
  we can prove that
  \begin{align}\label{eqn_pitt_nec>1}
    \sup_{s>1} \left( \int_1^{s}t^{-\kappa p'} \varphi_{0}(t)^{p'} m(t) \,dt \right)^{\frac{1}{p'}}  \left( \int_0^{{1}/{s}}  \lambda^{-\sigma q}   n(\lambda)\,d\lambda\right)^{\frac{1}{q}}<\infty.
\end{align}
We will now use \eqref{eqn_pitt_nec1} and \eqref{eqn_pitt_nec>1} to derive more necessary conditions. In fact, we have for $s>1$ from \eqref{est_sharp_n} that
\begin{align}\label{eqn_n(l)>0}
    \left( \int_0^{{1}/{s}}  \lambda^{-\sigma q}   n(\lambda)\,d\lambda\right)^{\frac{1}{q}} \geq C s^{\sigma -\frac{3}{q}},
\end{align}
On the other hand utilizing the estimate \eqref{eqn_est_phi_0} of  $\varphi_0(t)$  and $m(t)$ near infinity, we can write  for $s>1$ 
\begin{align}\label{eqn_tk>1}
   \left( \int_1^{s}t^{-\kappa p'} \varphi_{0}(t)^{p'} m(t) \,dt \right)^{\frac{1}{p'}}   \geq  C \left( \int_1^{s}t^{p'-\kappa p'}e^{\rho t(2-p')}  \,dt \right)^{\frac{1}{p'}} \geq C \min\{s^{1-\kappa},1\}   e^{{\rho s(\frac{2}{p'}-1)}}.
\end{align}
We claim that for Pitt-type inequality in this exponential volume growth setting, it is necessary that $p\leq 2$. More precisely, if $p>2$, then $p'<2$ and plugging the inequality above in \eqref{eqn_n(l)>0}, we get for all $\sigma>0$
\begin{align}\label{eqn_p<2_nec}
    \sup_{s>1} \left( \int_1^{s}t^{-\kappa p'} \varphi_{0}(t)^{p'} m(t) \,dt \right)^{\frac{1}{p'}}  \left( \int_0^{{1}/{s}}  \lambda^{-\sigma q}   n(\lambda)\,d\lambda\right)^{\frac{1}{q}} \geq C \sup_{s>1} e^{{\rho s(\frac{2}{p'}-1)}} s^{\sigma -\frac{3}{q}}=\infty.
\end{align}
This contradicts \eqref{eqn_pitt_nec>1}, therefore $p\leq 2$ is necessary.

We will now proceed to establish that the condition in \eqref{nec_stan_s-k>} is necessary for all $\zeta \geq 0$. However, before doing that, we note that if $\sigma q \geq 2(\alpha + 1)$, then together with the necessary condition $\kappa p' < 2(\alpha + 1)$, we already have \eqref{nec_stan_s-k>}. Thus, it is sufficient to consider the case where $\sigma q < 2(\alpha + 1)$. For $0 < s < 1$, using the estimates of $\varphi_0(t)$ and $m(t)$ near zero, as well as the estimates of $n(\lambda)$ near infinity, we can deduce that 
\begin{multline*}
    \sup_{0<s\leq 1} \left( \int_0^{s}t^{-\kappa p'} \varphi_{0}(t)^{p'} m(t) \,dt \right)^{\frac{1}{p'}}  \left( \int_1^{{1}/{s}}  \lambda^{-\sigma q}   n(\lambda)\,d\lambda\right)^{\frac{1}{q}}\\ 
        \gtrsim \sup_{0<s\leq 1} \left( \int_0^{s}t^{-\kappa p'} t^{2\alpha+1} \,dt \right)^{\frac{1}{p'}}  \left( \int_1^{{1}/{s}}  \lambda^{-\sigma q}   \lambda^{2\alpha+1}\,d\lambda\right)^{\frac{1}{q}}.
\end{multline*}
 Then we bound below the right-hand side of the inequality above  by
\begin{align}\label{eqn_nec_0}
\sup_{0<s\leq 1} s^{\sigma-\kappa+2(\alpha+1)\left(\frac{1}{p'}-\frac{1}{q}\right)},
\end{align}
which will be finite if and only if \eqref{nec_stan_s-k>} holds. Thus \eqref{nec_stan_s-k>} is necessary for all $\zeta \geq 0$.  Moreover,
when $p\leq 2$ and $q=2$, we derive from \eqref{nec_stan_s-k>}  that $\kappa\leq \sigma$ is also necessary for all $\zeta \geq 0$. 
Additionally for $\zeta=0$, utilizing \eqref{eqn_n(l)>0} we can write 
\begin{align}\label{eqn_nec>1_sJ}
 \sup_{s>1}\left( \int_1^{s}t^{- 2 \kappa } \varphi_{0}(t)^{2} m(t) \,dt \right)^{\frac{1}{2}}   \left( \int_0^{{1}/{s}}  \lambda^{-\sigma q }   n(\lambda)\,d\lambda\right)^{\frac{1}{q}}  \geq C \sup_{s>1} s^{\frac{3}{2}-\kappa} s^{\sigma-\frac{3}{q}},
\end{align}
which will be finite if and only if 
\begin{align*}
    \sigma- \kappa \leq 3 \left(\frac{1}{q}-\frac{1}{2}\right).
\end{align*}
Thus for $p=q=2$ and $\zeta =0$ case, it is necessary that $\kappa=\sigma<{\min\{2(\alpha+1),3\}}/{2}$.  
Similarly, when $\zeta \neq 0$ and $p = 2$, we can show that $\kappa \geq 3\left(\frac{1}{2} - \frac{1}{q}\right)$ is a necessary condition. In particular, for $p = q = 2$, we also find that $\kappa \geq 0$ is necessary.
\end{proof}

\begin{theorem}\label{thm_nec_inv_pitt_n}
   Let $1<p,q<\infty$. Suppose the following Pitt-type inequality holds:
\begin{align}\label{eqn_pit_nec,n_I}
   \left( \int_0^{\infty} |\mathcal{I}_{\kappa,\sigma}F(t)|^q  m(t) \,dt \right)^{\frac{1}{q}}\leq C \left( \int_0^{\infty} |F(\lambda)|^{p} n(\lambda) d\lambda \right)^{\frac{1}{p}}
\end{align}
for all $f \in C_c^{\infty}(\R_+)$.  Then it is necessary that $q\geq 2$, and the following conditions hold
\begin{enumerate}
\item if $\zeta \neq 0$, then $\sigma  q<{2(\alpha+1)}$ and
\begin{align}\label{eqn_nec_bal_I,ks}
    \sigma- \kappa\leq  2\left(\alpha+1 \right)\left(\frac{1}{p}+\frac{1}{q}-1\right).
\end{align}
\item if  $\zeta \neq 0$ and $q=2$, then $\sigma <{\alpha+1}$,   \eqref{eqn_nec_bal_I,ks}, and $\sigma \geq 3(1/p-1/2)$. Additionally, if $p=q=2$, then $ 0\leq \sigma <{\alpha+1}$ and $\sigma \leq  \kappa$;
\item if $\zeta=0$,   then $\sigma q<{2(\alpha+1)}$,  $\kappa p'<3$, and \eqref{eqn_nec_bal_I,ks};
\item\label{it_sJ4} if $\zeta=0$ and $q=2$,   then $\sigma q<{2(\alpha+1)}$,  $\kappa p'<3$,  \eqref{eqn_nec_bal_I,ks} and $\sigma -\kappa \geq 3(1/p-1/2) $. Additionally, if $p=q=2$, then  $\sigma=\kappa<{\min\{2(\alpha+1),3\}}/{2}$.
\end{enumerate}
\end{theorem}
\begin{proof}
   By considering the function
    \begin{align}
        F_0(\lambda)= \Chi_{(0,1)}(\lambda)  \lambda^{-\frac{3}{p}} ,\qquad \lambda\in (0,\infty),
    \end{align}
    and following similar calculations as outlined in equations $ \eqref{eqn_jf>c} $ and $ \eqref{eqn_nec_sigm} $, we can demonstrate that the condition $ \sigma q < 2(\alpha + 1) $ is necessary for all $ \zeta \geq 0 $. Additionally, when $ \zeta = 0 $, the condition $ \kappa p' < 3 $ is also  necessary.
    
Next, considering the function for $0<s<1$
\begin{align*}
    F_1(\lambda) = \Chi_{(0,s)}(\lambda)  \lambda^{\frac{\kappa}{1-p}} ,\qquad \lambda\in (0,\infty),
\end{align*}
we get 
    \begin{align*}
     \left( \int_0^{\infty} |\mathcal{I}_{\kappa,\sigma}F_1(t)|^q  m(t) \,dt \right)^{\frac{1}{q}} & \geq \left( \int_0^{\infty} t^{-\sigma q}  \left|\int_0^{\infty}\lambda^{-\kappa}F_1(\lambda) \varphi_{\lambda}(t) n(\lambda) \,d\lambda \right|^q m(t) \,dt\right)^{\frac{1}{q}}\\
     &\geq \left( \int_1^{1/s} t^{-\sigma q}  \left(\int_0^{s}\lambda^{-\kappa p'} \varphi_{0}(t) n(\lambda) \,d\lambda \right)^q m(t)\, dt\right)^{\frac{1}{q}}\\
     &\geq \left( \int_1^{1/s} t^{-\sigma q} \varphi_{0}(t)^{q}   m(t) \,dt\right)^{\frac{1}{q}} \left(\int_0^{s}\lambda^{-\kappa p'}  n(\lambda)\, d\lambda \right).
    \end{align*}
    Putting the estimate above in \eqref{eqn_pit_nec,n_I} implies the following
    \begin{align}\label{eqn_I<1_nec}
       \sup_{0<s\leq 1}  \left(\int_0^{s}\lambda^{-\kappa p'}  n(\lambda)\, d\lambda \right)^{\frac{1}{p'}} \left( \int_1^{1/s} t^{-\sigma q} \varphi_{0}(t)^{q}   m(t) \,dt\right)^{\frac{1}{q}} <\infty.
    \end{align}
    Similarly for $s>1$, by considering 
    \begin{align*}
    F_2(\lambda) = \Chi_{(1,s)}(\lambda)  \lambda^{\frac{\kappa}{1-p}} ,\qquad \lambda\in (0,\infty),
\end{align*}
we derive the following
\begin{align}\label{eqn_I>_nec}
       \sup_{s>1}  \left(\int_1^{s}\lambda^{-\kappa p'}  n(\lambda)\, d\lambda \right)^{\frac{1}{p'}} \left( \int_0^{1/s} t^{-\sigma q} \varphi_{0}(t)^{q}   m(t) \,dt\right)^{\frac{1}{q}} <\infty.
    \end{align}
    Following a similar calculation as in \eqref{eqn_p<2_nec}, and using \eqref{eqn_I<1_nec}, we can derive that in this case $q\geq 2$ is necessary. Moreover, using \eqref{eqn_I>_nec} and following the calculation as in \eqref{eqn_nec_0}, we can derive that \eqref{eqn_nec_bal_I,ks} is necessary, that is 
    \begin{align}\label{eqn_nec_InvJ}
        \sigma-\kappa \leq 2(\alpha+1)\left(\frac{1}{p}+\frac{1}{q}-1 \right).
    \end{align}
    Moreover, it follows from above that for $q\geq p'$, $\sigma\leq \kappa <{3}/{p'}$ is also necessary. Additionally, for the $q=2$ case, we can write for $\zeta =0$ case
    \begin{align} 
       \sup_{0<s\leq 1}  \left(\int_0^{s}\lambda^{-\kappa p' }  n(\lambda)\, d\lambda \right)^{\frac{1}{p'}} \left( \int_1^{1/s} t^{-2\sigma } \varphi_{0}(t)^{2}   m(t) \,dt\right)^{\frac{1}{2}} \geq C  \sup_{0<s\leq 1} s^{-\kappa+\frac{3}{p'}}s^{\sigma-\frac{3}{2}}.
    \end{align}
   The right hand side of the inequality above will be finite if and only if $$ \sigma - \kappa \geq 3\left( \frac{1}{2}-\frac{1}{p'}\right)= 3\left(\frac{1}{p}-\frac{1}{2} \right).$$ In particular, for the $p=q=2$ case, we can say that its necessary that $\sigma=\kappa<\min\{3, 2(\alpha+1)\}/2$. The calculation for $\zeta \neq 0$ is similar.
\end{proof}
\begin{remark}
When $p < 2$, the presence of the exponentially decaying term $e^{\rho(2-p')t}$ on the left-hand side of \eqref{eqn_pitt_nec>1} prevents us from directly applying the argument used in \eqref{eqn_nec>1_sJ} to formulate a necessary condition, since the supremum may remain finite irrespective of the polynomial weight. However, when $p = 2$ and $q > 1$, using \eqref{eqn_n(l)>0}, we  establish that the condition $\sigma - \kappa \leq 3\left( {1}/{q} - {1}/{2} \right)$ is necessary for the inequality \eqref{eqn_pit_nec,n} to hold. In the next section, we will see that the absence of the exponential factor in the modified Jacobi transform allows us to derive additional necessary conditions, which coincide with those of the standard Jacobi transform when $p = 2$ and $q > 1$. 
\end{remark}

 \subsection{Necessary conditions for shifted Pitt's inequality for modified Jacobi transforms}

Here we will consider the following class of operators
\begin{align*}
    \tilde{\mathcal{J}^{\zeta}}&= \left\{\tilde{\mathcal{J}}^{\zeta}_{\kappa,\sigma}f(\lambda) =(\lambda^2+\zeta^2)^{-\frac{\sigma}{2}}  \int_0^{\infty}t^{-\kappa}f(t) \tilde{\varphi}_{\lambda}^{(\alpha,\beta)}(t) m(t) \,dt, \quad f\in C_c^{\infty}(\R_+) \right\},\\
    \tilde{\mathcal{I}}^{\zeta}&= \left\{\tilde{\mathcal{I}}^{\zeta}_{\kappa,\sigma}F(t) =t^{-\sigma}  \int_0^{\infty}(\lambda^2+\zeta^2)^{-\frac{\kappa}{2}}F(\lambda) \tilde{\varphi}_{\lambda}^{(\alpha,\beta)}(t) n(\lambda) \,d\lambda, \quad F\in C_c^{\infty}(\R_+) \right\}.
\end{align*}
\begin{theorem}\label{thm_nec_modJ}
    Suppose that the following inequality holds for $1<p,q<\infty$
\begin{align}\label{eqn_nec_modJ}
   \left( \int_0^{\infty} |\tilde{\mathcal{J}}^{\zeta}_{\kappa,\sigma}f(\lambda)|^q  n(\lambda) d\lambda\right)^{\frac{1}{q}}\leq C \left( \int_0^{\infty} |f(t)|^p m(t) dt \right)^{\frac{1}{p}}
\end{align}
for all $f \in C_c^{\infty}(\R^+)$.
Then, the following conditions are necessary:
\begin{enumerate}
     \item\label{it_modJ2}  if  $\zeta \neq 0$, then $\kappa p' < 2(\alpha+1)$,  \eqref{nec_stan_s-k>}, and
     \begin{align}\label{eqn_nec_mJ,k>}
         \kappa \geq 3\left( 1-\frac{1}{p}-\frac{1}{q}\right);
     \end{align}
    
       \item\label{it_modJ4} if  $\zeta=0$, then $\kappa p' <  2(\alpha+1)$, $\sigma q<3$, \eqref{nec_stan_s-k>}, and 
       \begin{align}\label{eqn_nec_s-k<}
            \sigma-\kappa \leq 3\left( \frac{1}{p}+\frac{1}{q} -1\right).
       \end{align}
\end{enumerate}
\end{theorem}
\begin{proof}
We follow the approach outlined in Theorem \ref{thm_nec_pitt_n}. In fact, by repeating a similar argument, we find that the condition $\kappa p' < 2(\alpha + 1)$ and inequality \eqref{nec_stan_s-k>} are necessary for all $\zeta \geq 0$. Moreover, when $\zeta = 0$, the additional condition $\sigma q < 3$ is also required. To demonstrate that \eqref{eqn_nec_s-k<} is necessary in the case $\zeta = 0$, we consider the function $f_2$ defined in \eqref{eqn_defn_f2}. Using \eqref{eqn_nec_modJ}, we derive the following analogue of \eqref{eqn_pitt_nec>1} 
\begin{align}\label{eqn_pitt_nec-mj>1}
    \sup_{s>1} \left( \int_1^{s}t^{-\kappa p'} \tilde{m}(t) \,dt \right)^{\frac{1}{p'}}  \left( \int_0^{{1}/{s}}  \lambda^{-\sigma q}   n(\lambda)\,d\lambda\right)^{\frac{1}{q}}<\infty.
\end{align}
We observe that if $\kappa p'\geq 3$, then \eqref{eqn_nec_s-k<} automatically satisfied. So let us assume that $\kappa p' < \min\{3, 2(\alpha+1)\}$. Then, using the estimate \eqref{est_Delta} for $\tilde{m}(t)$, we obtain the lower bound 
\begin{align}\label{est_lower_vk} \int_1^s t^{-\kappa p'} \tilde{m}(t) \, dt \gtrsim s^{-\kappa p'+3}. 
\end{align} Combining the lower bound \eqref{est_lower_vk} with \eqref{eqn_n(l)>0}, we find that 
\begin{align*} \sup_{s>1} \left( \int_1^{s}t^{-\kappa p'} \tilde{m}(t) \,dt \right)^{\frac{1}{p'}} \left( \int_0^{{1}/{s}} \lambda^{-\sigma q} n(\lambda)\,d\lambda\right)^{\frac{1}{q}} \gtrsim \sup_{s>1} s^{\sigma -\frac{3}{q}+ \frac{3}{p'}-\kappa}. 
\end{align*} The right-hand side is finite if and only if condition \eqref{eqn_nec_s-k<} holds, thus establishing \eqref{it_modJ4}.

The case \eqref{it_modJ2} follows by noting that when $\zeta\neq  0$, the term $\lambda^2 + \zeta^2$ is bounded below by a positive constant for small $\lambda$, which ensures that the relevant integral remains finite.
\end{proof}

\begin{theorem}\label{thm_nec_modI}
    Suppose that the following inequality holds for $1<p,q<\infty$
\begin{align}\label{eqn_nec_modI}
   \left( \int_0^{\infty} |\tilde{\mathcal{I}}^{\zeta}_{\kappa,\sigma}F(t)|^q  \tilde{m}(t)\, dt \right)^{\frac{1}{q}}\leq C \left( \int_0^{\infty} |F(\lambda)|^p n(\lambda) \,  d\lambda \right)^{\frac{1}{p}}
\end{align}
for all $f \in C_c^{\infty}(\R^+)$.
Then, the following conditions are necessary:
\begin{enumerate}
     \item  if  $\zeta \neq 0$, then $\sigma q <  2(\alpha+1)$,  \eqref{eqn_nec_bal_I,ks}, and
     \begin{align}\label{eqn_nec_con_mI,s>}
         \sigma \geq 3\left( \frac{1}{p}+\frac{1}{q}-1\right);
     \end{align}
     %\item if $\zeta=0$ and $\sigma q \geq 3$, then \eqref{eqn_nec_bal_I,ks} and $\kappa p'<3$;
       \item if  $\zeta=0$, then $\sigma q < 2(\alpha+1)$,  $\kappa p'<3$, \eqref{eqn_nec_bal_I,ks}, and 
       \begin{align}\label{eqn_nec_s-k>_mI}
            \sigma-\kappa \geq 3\left( \frac{1}{p}+\frac{1}{q} -1\right).
       \end{align}
\end{enumerate}
\end{theorem}
\begin{proof}
    The proof follows by applying arguments similar to those used in Theorems \ref{thm_nec_inv_pitt_n} and \ref{thm_nec_modJ}.
\end{proof}

\subsection{Sufficient conditions for shifted Pitt's inequality for Jacobi transforms} In this section, we will establish the shifted Pitt's inequality for the Jacobi transforms. To that end, we first develop a general framework for weighted boundedness of sublinear operators.

Throughout this section, $u$ and $v$ will denote non-negative measurable functions defined on $\mathbb{R}_+$. Since we will be working with the non-increasing rearrangement of functions on $\mathbb{R}_+$ with respect to two different measures, namely $n(\lambda)\,d\lambda$ and $\tilde{m}(t)\,dt$, we adopt the following notation: the rearrangement of $u$ with respect to $n(\lambda)\,d\lambda$ is denoted by $u^{\star}$, while the rearrangement of $v$ with respect to $\tilde{m}(t)\,dt$ is written as $v^*$.
We also let $v_* = [(1/v)^*]^{-1}$. For $g\geq 0$, let us define for non-negative locally integrable functions 
\begin{align*}
    P_x g=\int_0^x g(y) dy, \quad Q_x g=\int_x^{\infty} g(y) dy,
\end{align*}
which are known as the Hardy and Bellman operators, respectively. Then, for non-negative locally integrable functions $u$ and $v$, one can consider the following Hardy inequalities for $1\leq p\leq q<\infty$,
\begin{align}\label{H_ineq_P}
    \|P_x g\|_{q,u}\lesssim \|g\|_{p,v}
\end{align}
and
\begin{align}\label{H_ineq_Q}
    \|Q_x g\|_{q,u}\lesssim \|g\|_{p,v}.
\end{align}
It is known that \cite{Bra78} inequality \eqref{H_ineq_P} holds if and only if, for each $r>0$,
\begin{align*}
    (Q_ru)^{1/q}(P_r v^{1-p'})^{1/p'} \lesssim 1.
\end{align*}
Similarly \eqref{H_ineq_Q} holds if and only if for each $r>0$,
\begin{align*}
    (P_r u)^{1/q} (Q_r v^{1-p'})^{1/p'} \lesssim 1.
\end{align*}
We consider the following measure spaces $X= (\R_+, d\tilde{m}(t))$ and $Y= (\R_+, dn(\lambda) )$.
\begin{theorem}\label{thm_pitt_subl}
Let $1 < p \leq q < \infty$, and let $T$ be a sublinear operator from $X$ to $Y$ that maps $C_c^{\infty}(\R_+)$ into measurable functions on $\R_+$. Suppose that $T$ is of type $(1,\infty)$ and type $(2,2)$, and let $u, v$ be weight functions satisfying the condition
 \begin{align}\label{hyp_on_gen_wts}
        \sup_{r>0} \left(P_{1/r}{u^{\star}}^q \right)^{1/q}\left(P_r v_{*}^{-p'} \right)^{1/{p'}}<\infty.
    \end{align}
    Then the following weighted inequality holds:
 \begin{align}\label{eqn_pitt_gent}
        \|Tf\cdot u\|_{L^q(Y)} \leq C \|f\cdot v\|_{L^p(X)}
    \end{align}
    for all $f \in C_c^{\infty}(\R_+)$.    
    Moreover, an analogous result holds if the roles of $X$ and $Y$ are interchanged, with the corresponding adjustments to the hypothesis \eqref{hyp_on_gen_wts} on the weights.
\end{theorem}
\begin{proof}
   The proof follows similar techniques to those employed in \cite[Theorem 2.4]{He84} and \cite[Proposition 2.6]{He84} for the case when both $p, q$ are not equal to 2, where we use \eqref{H_ineq_P} and \eqref{H_ineq_Q}. In fact, in \cite[Theorem 1.3]{KPRS24}, the first author established an analogue of \cite[Theorem 2.4]{He84} in a comparable setting, from which the proof of the theorem directly follows when both $p, q$ are not equal to  $2$.
   
For the case $p = q = 2$, a similar approach to that in \cite[Theorem (3.1) B]{GLT18} can be applied; (see also \cite[Theorem 1(i)]{BH03}). While the underlying methods can easily be extended to a broader context, we include a sketch of the proof for the case $p = q = 2$ for the sake of completeness. Suppose $T$ is a sublinear operator of   type $(1, \infty)$ and  type $(2, 2)$. Then, following the calculation as in \cite[Theorem 4.6, p. 260]{JT70}, we obtain 
\begin{align*}
\int_{0}^{x} (Tf)^{\star}(t)^2 \,dt \lesssim \int_0^{x} \left( \int_0^{1/t}f^*(s)\, ds \right)^2 dt.
\end{align*} Using this inequality and applying a similar argument as in \cite[Theorem B]{GLT18}, we can show that inequality \eqref{eqn_pitt_gent} holds.
\end{proof}
Next, we apply Theorem \ref{thm_pitt_subl} to the modified Jacobi transform. To do so, we first need to estimate the non-increasing rearrangement of the associated polynomial weights.
\subsubsection{Weights and their non-increasing rearrangement}
\begin{enumerate}
   \item\label{exm-triv} \textit{(Trivial weight)}  We denote by $w_0$ the weight that is identically $1$. It is easy to verify that the non-increasing rearrangement with respect to an infinite measure is the constant function $1$.
    \item\textit{(Rearrangement w.r.t. the measure $\tilde{m}(t) dt$)} For $\kappa\geq 0$, let us consider the weight $v_{\kappa}(s):= s^{\kappa}, s\geq 0$. We first find the non-increasing rearrangement of the function $1/v_{\kappa}$ for $ \kappa>0$ with respect to the measure $\tilde{m}(t)dt$. The $\kappa=0$ case follows from the previous case \eqref{exm-triv}. For $\kappa>0$, the distribution function of $1/v_\kappa$ is 
\begin{align*}
    d_{1/v_{\kappa}}(\gamma)&=\mu \left(\left\{s \geq 0: v_{\kappa}^{-1}(s) >\gamma \right\}\right)\\
    &=  \mu \left(\left\{s \geq 0: s^{-\kappa} >\gamma \right\}\right)\\&= \mu \left( \left\{s:s <(1/\gamma)^{\frac{1}{\kappa}} \right\}\right).
\end{align*}
For $1 < \gamma < \infty$, using 
\begin{align}\label{est_Delta}
 \tilde{m}(s) \asymp \begin{cases}
                s^{2\alpha+1}  \quad &\text{if } 0\leq s\leq 1,\\
                   s^2 \quad  & \text{if } 1< s <\infty,
                   \end{cases}
\end{align}
 we get
\begin{align*}
     d_{1/v_{\kappa}}(\gamma)\asymp \int_0^{(1/\gamma)^{\frac{1}{\kappa}}} s^{2\alpha+1}\, dt \asymp \gamma^{-\frac{2(\alpha+1)}{\kappa}}.
\end{align*}
When $0<\gamma\leq 1$, 
\begin{align*}
     d_{1/v_{\kappa}}(\gamma)\asymp \int_{0}^{1} s^{2\alpha+1}ds + \int_1^{\gamma^{-\frac{1}{\kappa}}}s^2\,ds
    \asymp \gamma^{-\frac{3}{\kappa}}.
\end{align*}
Thus the non-increasing rearrangement $(1/v_{\kappa})^{*}$ of $1/v_{\kappa}$ is given by
\begin{align*}
     {(1/v_{\kappa})}^{*}(t) & = \inf\{\gamma>0: d_{1/v_{\kappa}}(\gamma)\leq t\} \\
                   & =  \min\{ \inf\{0<\gamma\leq 1: d_{{1/v_{\kappa}}}(\gamma) \leq t\}, \, \inf\{1<\gamma < \infty: d_{{1/v_{\kappa}}}(\gamma) \leq t\}\}.
\end{align*}
Now we make use of the estimate of $d_{1/v_{\kappa}}(\gamma)$ for $\gamma$ near zero and away from zero to derive the following
\begin{align*}
    (1/v_{\kappa})^*(t) & \asymp \min\{ \inf\{0<\gamma\leq 1: \gamma^{-\frac{3}{\kappa}} \lesssim t\}, \, \inf\{1<\alpha < \infty: \gamma^{-\frac{2(\alpha+1)}{\kappa}} \lesssim t\}\}.
\end{align*}
This gives
\begin{align*}
     (1/v_{\kappa})^*(t) & \asymp  \min\left\{ \inf\left\{0<\gamma\leq 1:  \gamma \geq   t ^{-\frac{\kappa}{3}} \right\}, \, \inf\left\{1<\gamma < \infty: \gamma \geq  t^{-\frac{\kappa}{2(\alpha+1)}} \right\}\right\}.
\end{align*}
Hence, by considering $t$ is near and away from zero, we obtain from the above
\begin{align}\label{est_1/vk*}
     (1/v_{\kappa})^*(t) \asymp &\begin{cases}
         t^{-\frac{\kappa}{(2(\alpha+1)}} &\text{if  } 0\leq t\leq 1,\\
        t^{-\frac{\kappa}{3}} &\text{if }1 <t<\infty.
     \end{cases}
\end{align}
    
    \item  \textit{((Rearrangement w.r.t. the measure $n(\lambda)d\lambda$)}
 For $\sigma \geq 0$ and $\zeta\geq 0$, we consider the weight $u_{\zeta, \sigma}$ given by 
 \begin{equation*}
u_{\zeta,\sigma}(\lambda)=(\lambda^2+\zeta^2)^{-\frac{\sigma}{2}}, \qquad \text{for all $\lambda>0$.}
\end{equation*} 
Using a similar computation as in Lemma \ref{lem_plaey_inq}, together with the estimate for $n(\lambda)$ provided in \eqref{est_sharp_n}
 \begin{equation}\label{est_c-2_d}
 \begin{aligned}
        n(\lambda)\asymp \begin{cases}
            \lambda^2 \qquad &\text{if } 0\leq \lambda\leq 1,\\
            \lambda^{2\alpha+1} \qquad &\text{if } \lambda > 1,
        \end{cases}
    \end{aligned}
    \end{equation}
    we arrive at the following inequality for $\zeta=0$
  \begin{align}
      d_{u_{0,\sigma}}(\gamma)\leq C \begin{cases}
          \gamma^{-\frac{2(\alpha+1)}{\sigma}} \quad &\text{for } 0<\gamma\leq 1,\\
          \gamma^{-\frac{3}{\sigma}} \quad &\text{for } \gamma>1.
      \end{cases}
  \end{align}  
  When $\zeta \neq 0$,
   \begin{align}
      d_{u_{\zeta,\sigma}}(\gamma)\leq C \begin{cases}
          \gamma^{-\frac{2(\alpha+1)}{\sigma}}  \quad &\text{for } 0<\gamma<\left(\frac{1}{\zeta^2}\right)^{\frac{\sigma}{2}},\\
          0 \quad &\text{for } \gamma \geq \left(\frac{1}{\zeta^2}\right)^{\frac{\sigma}{2}}.
      \end{cases}
  \end{align} 
Thus, from the definition of $ u_{0,\sigma}^{\star}$, we have
\begin{align*}
    u_{0,\sigma}^{\star}(t) & = \inf\{\gamma>0: d_{u_{0,\sigma}}(\gamma)\leq t\} \\
                   &\lesssim \min\{ \inf\{0<\gamma\leq 1: {\gamma^{-\frac{2(\alpha+1)}{\sigma}}}\lesssim t\}, \, \inf\{1< \gamma< \infty: \gamma^{-\frac{3}{\sigma}} \lesssim t\}\}\\
    & \lesssim \min\{ \inf\{0<\gamma\leq 1: {\gamma \gtrsim  t^{-\frac{\sigma}{2(\alpha+1)}}}\}, \, \inf\{1< \gamma< \infty: \gamma \gtrsim t^{-\frac{\sigma}{3}}\}\}.
\end{align*}
Therefore,  we have obtained the following
\begin{align}\label{est_u_s^*}
     u_{0,\sigma}^{\star}(t)\lesssim
 \begin{cases}
        t^{-\frac{\sigma}{3}} \qquad &\text{if } 0 \leq t\leq 1,\\
        t^{-\frac{\sigma}{2(\alpha+1)}} \qquad & \text{if } 1<t< \infty.\\
        \end{cases}
\end{align}
Similarly, for $\zeta \neq 0$, we can show that
\begin{align}\label{est_u_s^*z>0}
     u_{\zeta,\sigma}^{\star}(t)\lesssim
 \begin{cases}
       1 \qquad &\text{if } 0 \leq t\leq 1,\\
        t^{-\frac{\sigma}{2(\alpha+1)}} \qquad & \text{if } 1<t< \infty.\\
        \end{cases}
\end{align}
\end{enumerate}
Now that we have obtained estimates for the non-increasing rearrangements of the weights, we apply Theorem \ref{thm_pitt_subl} to derive sufficient conditions on the weights $u_{\zeta, \sigma}$ and $v_{\kappa}$ for the  Pitt's inequality \eqref{eqn_pitts_modi} to hold for modified Jacobi transforms. Our first result in this direction is as follows.

\begin{theorem}\label{thm_pitt's_mod_J}
  Let $\zeta \geq 0$ and $1 < p \leq q < \infty$.   Assume that $\sigma, \kappa \geq 0$ such that $  \kappa p' < 2(\alpha + 1)$ and  \eqref{nec_stan_s-k>} is satisfied.    Then the following shifted Pitt's inequality holds   
    \begin{align}\label{eqn_pitts_modi}
         \left( \int_0^{\infty} |\tilde{\mathcal{J}} f(\lambda)|^q  (\lambda^2 +\zeta^2)^{-\frac{\sigma q}{2}} n(\lambda) \, d\lambda \right)^{1/q} \leq C   \left( \int_0^{\infty} | f(t)|^p  t^{\kappa p} \tilde{m}(t) \, dt \right)^{1/p},
    \end{align}
  provided one of the following additional conditions is satisfied
   \begin{enumerate}
       \item if $\zeta \neq 0$ and $\kappa p' \geq 3$;
       \item if  $\zeta \neq 0$ and $\kappa p' < \min\{3, 2(\alpha+1)\}$, then \eqref{eqn_nec_mJ,k>};
       \item if $\zeta=0$ and $\kappa p' \geq 3$, then $\sigma q<3$;
       \item if  $\zeta=0$ and $\kappa p' < \min\{3, 2(\alpha+1)\}$, then $\sigma q<3$ and  \eqref{eqn_nec_s-k<}; 
   \end{enumerate}
\end{theorem}
\begin{proof}
We recall from \eqref{eqn_plac_mJ,I} that the modified Jacobi transform $\tilde{\mathcal{J}}$ is a bounded operator from $L^2(d\tilde{m})$ to $L^2(dn)$. Moreover, its boundedness from $L^1(d\tilde{m})$ to $L^\infty(dn)$ follows directly from its definition. Therefore, we can apply Theorem \ref{thm_pitt_subl} with $T = \tilde{\mathcal{J}}$. It is thus sufficient to show that if the parameters $\sigma$ and $\kappa$ satisfy the conditions stated in the hypothesis of Theorem \ref{thm_pitt's_mod_J}, then the following inequality holds for all $1 < p \leq q < \infty$:
\begin{align*}
\sup_{r > 0} \left(P_{1/r} u_{\zeta, \sigma}^{\star q} \right)^{1/q} \left(P_r v_\kappa^{-p'} \right)^{1/p'} < \infty.
\end{align*}
Since the estimates of ${u_{\zeta, \sigma}^{\star}}(s) $ and $ {v_\kappa}_{*}(s)$ differ significantly when $s$ is near zero, and when it is away from zero, it is natural to consider these two cases separately. Let us first examine the case when $s$ is near zero.

\par\textbf{Case I: $0<r \leq 1 $.} We need to consider the cases $\alpha>1/2$ and $-1/2<\alpha\leq 1/2$ separately. First, let us consider $\alpha>1/2$ and $\zeta=0$.  For  $\sigma q<3$, we can write from  \eqref{est_u_s^*}
\begin{equation}\label{est_u,a>1/2}
\begin{aligned}
\left(\int_0^{1/r}{u_{0, \sigma}^{\star}}^qdt\right)^{\frac{1}{q}} & \asymp \left(\int_0^1t^{-\frac{\sigma q}{3}}dt+\int_1^{1/r}t^{-\frac{\sigma q}{2(\alpha+1)}}dt\right)^{\frac{1}{q}}\\
    & \lesssim (1+r^{\frac{\sigma}{2(\alpha+1)}-\frac{1}{q}}).
\end{aligned}
\end{equation}
 For the same range of $s$ as above  and $\kappa p'<2(\alpha+1)$, we get from  \eqref{est_1/vk*}
\begin{align*}
     \left(\int_0^r {v_{\kappa}}_{*}^{-p'}\,dt\right)^{\frac{1}{p'}} \leq C \left(\int_0^r t^{-\frac{\kappa p'}{2(\alpha+1)}}\,dt\right)^{\frac{1}{p'}} = C r^{-\frac{\kappa}{2(\alpha+1)}+\frac{1}{p'}}.
\end{align*}
Therefore, we have
\begin{align}\label{eqn_sufpitt<1}
   \sup_{0<r\leq 1}  \left(P_{1/r}{u_{0,\sigma}^{\star}}^q \right)^{1/q}\left(P_r {v_\kappa}_{*}^{-p'} \right)^{1/{p'}} \leq C \left(1+\sup_{0<r\leq 1} r^{-\frac{\kappa}{2(\alpha+1)}+\frac{1}{p'}+\frac{\sigma}{2(\alpha+1)}-\frac{1}{q}} \right).
\end{align}
The expression on the right-hand side of the inequality above is finite if and only if \begin{align}\label{cond_pitt<1}
\frac{\sigma}{2(\alpha+1)}  +\frac{1}{p'} \geq \frac{\kappa}{2(\alpha+1)}+ \frac{1}{q} \quad \Leftrightarrow \quad \sigma-\kappa \geq 2(\alpha+1)\left( \frac{1}{p}+\frac{1}{q}-1\right).
\end{align}
When $\alpha \leq 1/2$, that is, when $2(\alpha+1) \leq 3$, we observe that for $\sigma q < 2(\alpha+1)$, the computation remains same, and the left-hand side of \eqref{eqn_sufpitt<1} is finite whenever condition \eqref{cond_pitt<1} is satisfied.
In the case where $2(\alpha+1) \leq \sigma q < 3$, the assumption $\kappa p' < 2(\alpha+1)$ alone guarantees that the left-hand side of \eqref{eqn_sufpitt<1} remains finite.

Additionally, we note that for the case $\zeta > 0$, to derive the inequality \eqref{est_u,a>1/2},  we can use \eqref{est_u_s^*z>0} instead of  \eqref{est_u_s^*}. In this case, the condition $\sigma q < 3$ is not required, which ensures that the right-hand side of \eqref{est_u,a>1/2} is finite when $\zeta = 0$.

\par\textbf{Case II: $1<r<\infty$.}   We can write from \eqref{est_u_s^*} and  \eqref{est_u_s^*z>0} 
\begin{equation}\label{est_u*>1}
    \begin{aligned}
    \left(\int_0^{1/r}u_{\zeta, \sigma}(t)^qdt\right)^{\frac{1}{q}} & \lesssim \begin{cases}
        r^{\frac{\sigma}{3}-\frac{1}{q}}&\quad \text{if } \zeta=0, \text{ and } \sigma q<3,\\
        r^{-\frac{1}{q}} &\quad \text{if } \zeta \neq 0.
    \end{cases} 
\end{aligned}
\end{equation}
To estimate $P_r  {v_{\kappa}}_{*}^{-p'}$, we need to consider the ranges of $\kappa$ separately.  First, let us handle the case when $\kappa p'<3$, for which we can write from \eqref{est_1/vk*} 
\begin{equation}\label{est_v_*>1,<3}
\begin{aligned}
    \left(\int_0^r {v_{\kappa}}_{*}^{-p'}dt\right)^{\frac{1}{p'}} & \leq  C \left(\int_0^{1} t^{-\frac{ \kappa p'}{2(\alpha+1)}}\,dt+\int_{1}^s t^{-\frac{\kappa {p'}}{3}}\,dt\right)^{\frac{1}{p'}}\\
    %& \leq C  \left(1+\int_{1}^s  t^{-\frac{\kappa {p'}}{3}}dt\right)^{\frac{1}{p'}}\\
    & \leq C r^{\frac{1}{p'}-\frac{\kappa}{3}}.
\end{aligned}
\end{equation}
 Thus, for the above range of $\kappa$, we have obtained the following 
\begin{align}\label{eqn_est_suff_r>1}
      \sup_{r>1} \left(P_{1/r}{u_{\zeta, \sigma}^{\star}}^q \right)^{1/q}\left(P_r {v_\kappa}_{*}^{-p'} \right)^{1/{p'}} \leq C  \begin{cases}\sup_{r>1}
        r^{\frac{\sigma}{3}-\frac{\kappa}{3}+\frac{1}{p'}-\frac{1}{q}}&\quad \text{if } \zeta=0, \text{ and } \sigma q<3,\\
        \sup_{r>1} r^{\frac{1}{p'}-\frac{\kappa}{3}{-\frac{1}{q}} }&\quad \text{if } \zeta \neq 0.
    \end{cases}
\end{align}
Therefore, for the case $\kappa p'<3$, the left-hand side of the inequality above will be finite, provided the following conditions are met 
\begin{enumerate}
    \item if $\zeta=0$, then $\sigma q<3$, and $\sigma -\kappa \leq 3(1/q-1/p')$;
    \item if $\zeta \neq 0$, then $ \kappa \geq 3(1/p'-1/q)$.
\end{enumerate}

We now consider the case $\kappa p' \geq 3$. We first observe that if $\alpha \leq 1/2$, then the result follows exactly as in the case $\kappa p' < 3$, since in this situation, we have $\kappa p' < 2(\alpha+1) \leq  3$ by assumption. Therefore, it remains to address the case when $\alpha > 1/2$ to complete the proof.
Suppose $\alpha>1/2$ and   $3< \kappa p' < 2(\alpha+1)$, then from \eqref{est_1/vk*} we obtain
\begin{equation}\label{est_v_*>1,1}
\begin{aligned}
    \left(\int_0^r {v_{\kappa}}_{*}^{-p'}dt\right)^{\frac{1}{p'}} & \leq  C \left(\int_0^{1} t^{-\frac{ \kappa p'}{2(\alpha+1)}}\,dt+\int_{1}^s t^{-\frac{\kappa {p'}}{3}}\,dt\right)^{\frac{1}{p'}}\\
    %& \leq C  \left(1+\int_{1}^s  t^{-\frac{\kappa {p'}}{3}}dt\right)^{\frac{1}{p'}}\\
    & \leq C \left(1+ \frac{1}{\frac{\kappa p'}{3}-1}(1-s^{1-\frac{\kappa p'}{3}}) \right)^{\frac{1}{p'}}<C,
\end{aligned}
\end{equation}
where the constant $C$ is independent of $r$. Similarly, when $\kappa p'=3 $, we can show that 
\begin{align}\label{est_v_*>1,2}
     \left(\int_0^r {v_{\kappa}}_{*}^{-p'}dt\right)^{\frac{1}{p'}} & \leq  C (1+\log r).
\end{align}
Thus combining the inequalities \eqref{est_u*>1}, \eqref{est_v_*>1,1}, \eqref{est_v_*>1,2},  it follows that
\begin{align}\label{eqn_pr_p1/r>1}
 \sup_{r>1} \left(P_{1/r}{u_{\zeta, \sigma}^{\star}}^q \right)^{1/q}\left(P_r {v_\kappa}_{*}^{-p'} \right)^{1/{p'}} \leq C  \begin{cases}\sup_{r>1}
        r^{\frac{\sigma}{3}-\frac{1}{q}}&\quad \text{if } \zeta=0, \text{ and } \sigma q<3,\\
        \sup_{r>1} r^{-\frac{1}{q}} (1+\log r) &\quad \text{if } \zeta \neq 0.
        \end{cases}
\end{align}
We conclude the proof of the theorem by noting that the right-hand side of \eqref{eqn_pr_p1/r>1} is always finite when $\zeta \neq 0$, and also when $\zeta = 0$, provided that $\sigma q < 3$.
\end{proof}
 
\begin{theorem}\label{thm_pitt's_mod_J-1}
  Let $\zeta \geq 0$ and $1 < p \leq q < \infty$.   Assume that $\sigma, \kappa \geq 0$ such that $\sigma q<2(\alpha+1)$ and  \eqref{eqn_nec_bal_I,ks} is satisfied.
    Then the following shifted Pitt's inequality  holds
    \begin{align}\label{eqn_pitts_modi-1}
         \left( \int_0^{\infty} |\tilde{\mathcal{I}} F(t)|^q t^{-{\sigma q}} \tilde{m}(t) \, dt \right)^{1/q} \leq C   \left( \int_0^{\infty} |F(\lambda)|^p  (\lambda^2+\zeta^2)^{\frac{\kappa p}{2}} n(\lambda)\, d\lambda \right)^{1/p},
    \end{align}
 provided one of the following additional conditions is satisfied
   \begin{enumerate}
       \item if $\zeta\neq 0$ and $\sigma q \geq 3$;
       \item if  $\zeta \neq 0$ and $\sigma q <\min\{3, 2(\alpha+1)\}$,  then \eqref{eqn_nec_con_mI,s>};
       % \begin{align}\label{cond_pit_z0_mI}
       %     \sigma \geq 3\left(\frac{1}{p}+\frac{1}{q}-1 \right);
       % \end{align} 
       \item if $\zeta=0$ and $\sigma q \geq 3$, then $\kappa p'<3$;
       \item if  $\zeta=0$ and $\sigma q < \min\{3, 2(\alpha+1)\}$, then $\kappa p'<3$ and \eqref{eqn_nec_s-k>_mI}.
   \end{enumerate}
\end{theorem}
\begin{proof}We begin by recalling that the operator $\tilde{\mathcal{I}}$ is bounded from $L^1(dn)$ to $L^{\infty}(dm)$, as well as from $L^2(dn)$ to $L^2(dm)$. Therefore, to apply Theorem \ref{thm_pitt_subl}, it is enough to show that if the parameters $\sigma$ and $\kappa$ satisfy the hypothesis,, then the following inequality holds:
\begin{align*}
\sup_{r>0} \left(P_{1/r}(u_{\sigma}^{* q}) \right)^{1/q}\left(P_r \left({v_{\zeta,\kappa}}_{{\star}}^{-p'}\right) \right)^{1/{p'}}<\infty,
\end{align*}
where $u_{\sigma}(t) = t^{-\sigma}$ and $v_{\zeta, \kappa}(\lambda) = (\lambda^2 + \zeta^2)^{\kappa}$.

\par\textbf{Case I: $0<r \leq 1 $.} Using the hypothesis $\sigma q < 2(\alpha+1)$ and the estimate \eqref{est_1/vk*}, and by considering different cases, we obtain
\begin{equation}\label{est_u,a<1_mI_b}
\begin{aligned}
\left(\int_0^{1/r}{u_{ \sigma}^{*}}^qdt\right)^{\frac{1}{q}} & \lesssim 
\begin{cases}
  1+  r^{\frac{\sigma}{3}-\frac{1}{q}} \quad & \text{if } \sigma q <\min\{3, 2(\alpha+1)\},\\
    1+\log({\frac{1}{r}}) \quad & \text{if } \sigma q =3,\\
   1 \quad & \text{if } \sigma q >3.
\end{cases}
\end{aligned}
\end{equation}
Similarly, for $\kappa p'<3$ (when $\zeta=0$), we can derive the following from \eqref{est_u_s^*} and \eqref{est_u_s^*z>0}
\begin{equation}\label{est_v_z,k*<1_mI}
\begin{aligned}
     \left(\int_0^r {v_{\zeta,\kappa}}_{\star}^{-p'}\,dt\right)^{\frac{1}{p'}} \lesssim \begin{cases}
         r^{\frac{1}{p'}-\frac{\kappa}{3}} \quad & \text{if } \zeta=0,\\
         r^{\frac{1}{p'}} \quad & \text{if } \zeta\not=0.\\
     \end{cases}
\end{aligned}
\end{equation}
We observe the right-hand side of \eqref{est_v_z,k*<1_mI} is bounded for $0<r\leq 1$. Now let us first consider the case $\alpha\leq 1/2$ and $\zeta=0$. In this case, we obtain the following estimate
\begin{align}\label{eqn_sufpitt<1_mI}
   \sup_{0<r\leq 1}  \left(P_{1/r}{u_{\sigma}^{*}}^q \right)^{1/q}\left(P_r {v_{\zeta,\kappa}}_{\star}^{-p'} \right)^{1/{p'}} \lesssim \begin{cases}
       \left(1+\sup_{0<r\leq 1} r^{\frac{\sigma}{3}-\frac{1}{q}-\frac{\kappa}{3}+\frac{1}{p'}} \right) \quad & \text{if } \zeta=0,\\
        \left(1+\sup_{0<r\leq 1} r^{\frac{\sigma}{3}-\frac{1}{q}+\frac{1}{p'}} \right) \quad &\text{if } \zeta\not=0.\\
   \end{cases} 
\end{align} 
For $\zeta = 0$, the expression on the right-hand side of the inequality above is finite if condition \eqref{eqn_nec_s-k>_mI} holds. For $\zeta \neq 0$, it is finite if condition \eqref{eqn_nec_con_mI,s>} is satisfied.

Next, we consider the case $\alpha>1/2 $. If $\sigma q <\min\{3, 2(\alpha+1)\}$, then the calculation is same as in \eqref{eqn_sufpitt<1_mI}. However, if $\sigma q\geq 3$, then the condition $\kappa p'<3$ alone guarantees that the expression in \eqref{eqn_sufpitt<1_mI} is finite.

\par\textbf{Case II: $1<r<\infty$.}   We obtain the following estimate from \eqref{est_1/vk*} for $\sigma q<2(\alpha+1)$
\begin{equation}\label{est_u*>1_mI}
    \begin{aligned}
    \left(\int_0^{1/r}u_{ \sigma}^{*}(t)^qdt\right)^{\frac{1}{q}} & \lesssim  r^{\frac{\sigma}{2(\alpha+1)}-\frac{1}{q}}.
\end{aligned}
\end{equation}
For $\zeta > 0$, using the estimate \eqref{est_u_s^*z>0}, and for $\zeta = 0$, using \eqref{est_u_s^*} along with the assumption $\kappa p' < 3$, we obtain
\begin{equation}\label{est_u,a<1_mI}
\begin{aligned}
\left(\int_0^{r}{{v_{\zeta, \kappa}}_{\star}}^{-p'}dt\right)^{\frac{1}{q}} & \lesssim 
\begin{cases}
  1+  r^{-\frac{\kappa}{2(\alpha+1)}+\frac{1}{p'}} \quad & \text{if } \kappa p' <2(\alpha+1),\\
    1+\log({{r}}) \quad & \text{if } \kappa p' =2(\alpha+1),\\
   1 \quad & \text{if } \kappa p' \geq 2(\alpha+1).
\end{cases}
\end{aligned}
\end{equation}
Under the conditions $\sigma q < 2(\alpha + 1)$ and $\kappa p' < 2(\alpha + 1)$ when $\zeta \neq 0$ (with the additional assumption $\kappa p' < 3$ required when $\zeta = 0$), it follows from \eqref{est_u*>1_mI} and \eqref{est_u,a<1_mI} that
\begin{align}\label{eqn_P1/r,Pr>1_mJ}
      \sup_{r>1}  \left(P_{1/r}{u_{\sigma}^{*}}^q \right)^{1/q}\left(P_r {v_{\zeta,\kappa}}_{\star}^{-p'} \right)^{1/{p'}}  \lesssim \left(1+\sup_{r>1}
        r^{\frac{\sigma}{2(\alpha+1)}-\frac{\kappa}{2(\alpha+1)}+\frac{1}{p'}-\frac{1}{q}} \right).
\end{align}
The right-hand side of the inequality above is finite if \eqref{eqn_nec_bal_I,ks} holds. When $\kappa \geq 2(\alpha+1)$, we observe that the condition $\sigma q <2(\alpha+1)$ alone guarantees that  right hand of \eqref{eqn_P1/r,Pr>1_mJ} is finite.
\end{proof}
\begin{remark}\label{rem_nec_suff_mJ}
     If we assume that $\sigma, \kappa \geq 0$, a condition required for obtaining the non-increasing rearrangements of the weights, then the necessary conditions for the shifted Pitt's inequality for the modified Jacobi transforms (for $\tilde{\mathcal{J}}$ in Theorem~\ref{thm_nec_modJ} and $\tilde{\mathcal{I}}$ in Theorem~\ref{thm_nec_modI}) coincide with the sufficient conditions in Theorem~\ref{thm_pitt's_mod_J} and Theorem~\ref{thm_pitt's_mod_J-1}, respectively. This provides a complete characterization of the class of polynomial weights with non-negative exponents for which shifted Pitt's inequality holds in the setting of modified Jacobi transforms, thereby complementing the result of \cite[Theorem 8.2]{GLT18}. More precisely, we have the following corollary for the transform $\tilde{\mathcal{J}}$ (see also Figure~\ref{pitt_mJ_Fig}). A similar result also holds for $\tilde{\mathcal{I}}$.
\end{remark}
\begin{corollary}\label{cor_char_pitt_mJ}
      Let $\zeta \geq 0$ and $1 < p \leq q < \infty$.   Assume that $\sigma, \kappa \geq 0$.  Then the shifted Pitt's inequality  \eqref{eqn_pitts_modi} holds  if and only if 
      \begin{align*}
            \sigma-\kappa \geq 2(\alpha+1)\left( \frac{1}{p}+\frac{1}{q}-1\right),
      \end{align*}
 and one of the following conditions is met
 \begin{enumerate}
          \item  if  $\zeta \neq 0$, then $\kappa p' < 2(\alpha+1)$, and
     \begin{align*}
         \kappa \geq 3\left( 1-\frac{1}{p}-\frac{1}{q}\right);
     \end{align*}
       \item if  $\zeta=0$, then $\kappa p' <  2(\alpha+1)$, $\sigma q<3$, and 
       \begin{align*} 
            \sigma-\kappa \leq 3\left( \frac{1}{p}+\frac{1}{q} -1\right).
       \end{align*}
 \end{enumerate}
\end{corollary}

\begin{figure}[ht]
    \centering
        \begin{subfigure}[t]{0.33\textwidth}
            \begin{tikzpicture}[line cap=round,line join=round,>=triangle 45,x=1.0cm,y=1.0cm, scale=0.33]
\clip(-3,-2) rectangle (15,11);
\draw [line width=0.5pt] (0,0)-- (0,10);
\draw[-latex, color=black](10,0)--(10.5,0);
%\draw [line width=1pt,color=cyan] (10,1)-- (10,10);
\draw[-latex, color=cyan](0,10.5);
%\draw [line width=0.5pt] (10,0)-- (10,1);
%\draw [line width=0.5pt] (0.,10)-- (10,10);
%\draw [line width=0.5pt] (10,10)--(8,8);
%\draw [line width=0.5pt] (10,0)--(8,0);

%\draw [line width=1pt,color=cyan] (1,1)--(10,10);
%\draw [line width=0.5pt] (1,1)--(0,0);
%\draw [line width=1pt,color=cyan,dashed] (1,1)--(10,1);
%\draw [line width=1pt,color=cyan,dashed] (8,8)--(8,0);
\draw [line width=0.5pt] (0,0)--(10,0);

%\fill[line width=1pt,color=cyan,fill=cyan, opacity=0.1] (1,1) -- (10,1) -- (10,10) --cycle;

% \draw (10,1) node[color=cyan]{$\circ$};
% \draw (1,1) node[color=cyan]{$\circ$};
%\draw (10,10) node[color=cyan]{$\circ$};
%%%
\draw (-0.75,-0.75) node{\small $0$};
\draw (10,-0.75) node{\small $\kappa$};
\draw (-0.75,10) node{\small $\sigma$};

\draw (4,3.5) node[rotate=45]{\tiny $\!\frac{\sigma\! -\!\kappa}{\mathcal{N}} = \frac{1}{q}-\!\!\frac{1}{p'}\!$};
%%%%
%%%
\draw [line width=1pt,color=cyan] (0,1)--(6,7);
\draw [line width=1pt,color=cyan] (0,1)--(0,10);

\draw [line width=1pt,color=cyan,dashed] (6,7)--(6,10);
%\draw[line width=1pt,color=magenta] (0,1)--(6,7);

\fill[line width=1pt,color=cyan,fill=cyan, opacity=0.2] (0,1) -- (6,7) -- (6,10) --(0,10)--cycle;

\draw (0,1) node[color=cyan]{$\circ$};
\draw (6,7) node[color=cyan]{$\circ$};
%\draw (6,0) node[color=cyan]{$\circ$};
%\draw (0,7) node[color=cyan]{$\circ$};

 \draw (8,7) node{\tiny $(\frac{\mathcal{N}}{p'},\frac{\mathcal{N}}{q})$};
%\draw (-1,7) node{\small $\frac{n}{q}$};
%\draw (6,-1) node{\small $\frac{n}{p'}$};
%%%
%%%%
% \draw (10.7,1) node{\small $\frac{\beta}{n}$};
% \draw (1,-1) node{\small $\frac{\beta}{n}$};
\end{tikzpicture}
            \noindent\subcaption{$\zeta \neq 0$ and $q< p'$}
            \label{Fig_Pitt_mjN1}
        \end{subfigure}
        ~
        \begin{subfigure}[t]{0.33\textwidth}
            \begin{tikzpicture}[line cap=round,line join=round,>=triangle 45,x=1.0cm,y=1.0cm, scale=0.33]
\clip(-3,-2) rectangle (15,11);
\draw [line width=0.5pt] (0,0)-- (0,10);
\draw[-latex, color=black](10,0)--(10.5,0);
%\draw [line width=1pt,color=cyan] (10,1)-- (10,10);
\draw[-latex](0,10)--(0,10.5);
%\draw [line width=0.5pt] (10,0)-- (10,1);
%\draw [line width=0.5pt] (0.,10)-- (10,10);
%\draw [line width=0.5pt] (10,10)--(8,8);
%\draw [line width=0.5pt] (10,0)--(8,0);

%\draw [line width=1pt,color=cyan] (1,1)--(10,10);
%\draw [line width=0.5pt] (1,1)--(0,0);
%\draw [line width=1pt,color=cyan,dashed] (1,1)--(10,1);
%\draw [line width=1pt,color=cyan,dashed] (8,8)--(8,0);
\draw [line width=0.5pt] (0,0)--(10,0);

%\fill[line width=1pt,color=cyan,fill=cyan, opacity=0.1] (1,1) -- (10,1) -- (10,10) --cycle;

% \draw (10,1) node[color=cyan]{$\circ$};
% \draw (1,1) node[color=cyan]{$\circ$};
%\draw (10,10) node[color=cyan]{$\circ$};
%%%
\draw (-0.75,-0.75) node{\small $0$};
\draw (10,-0.75) node{\small $\kappa$};
\draw (-0.75,10) node{\small $\sigma$};

\draw (6,2.5) node[rotate=45]{\tiny $\!\frac{\sigma\! -\!\kappa}{\mathcal{N}} = \frac{1}{q}-\!\!\frac{1}{p'}\!$};
\draw (2,5) node[rotate=90]{\tiny $\!\frac{\kappa}{3} = \!\frac{1}{p'}-\!\!\frac{1}{q}\!$};
%%%%
%%%
\draw [line width=1pt,color=cyan] (2,0)--(8,6);
\draw [line width=1pt,color=cyan] (1,0)--(1,10);
\draw [line width=1pt,color=cyan] (1,0)--(2,0);

\draw [line width=1pt,color=cyan,dashed] (8,6)--(8,10);
%\draw[line width=1pt,color=magenta] (0,1)--(6,7);

\fill[line width=1pt,color=cyan,fill=cyan, opacity=0.2] (1,0) --(2,0)-- (8,6) -- (8,10) --(1,10)--cycle;

\draw (1,0) node[color=cyan]{$\circ$};
\draw (2,0) node[color=cyan]{$\circ$};
\draw (8,6) node[color=cyan]{$\circ$};
%\draw (6,0) node[color=cyan]{$\circ$};
%\draw (0,7) node[color=cyan]{$\circ$};

 \draw (10,6) node{\tiny $(\frac{\mathcal{N}}{p'},\frac{\mathcal{N}}{q})$};

%\draw (-1,7) node{\small $\frac{n}{q}$};
%\draw (6,-1) node{\small $\frac{n}{p'}$};
%%%
%%%%
 
% \draw (10.7,1) node{\small $\frac{\beta}{n}$};
% \draw (1,-1) node{\small $\frac{\beta}{n}$};
\end{tikzpicture}
            \noindent\subcaption{$\zeta \neq 0$, $\mathcal{N}>  3$, and $q> p'$}
            \label{Fig_Pitt_mjN2}
        \end{subfigure}
        ~
        \begin{subfigure}[t]{0.33\textwidth}
            \begin{tikzpicture}[line cap=round,line join=round,>=triangle 45,x=1.0cm,y=1.0cm, scale=0.33]
\clip(-3,-2) rectangle (15,11);
\draw [-latex,line width=0.5pt] (0,0)-- (0,10.5);
\draw[-latex, color=black](10,0)--(10.5,0);
%\draw [line width=1pt,color=cyan] (10,1)-- (10,10);
%\draw[line width=1pt, color=cyan](0,3);
%\draw [line width=0.5pt] (10,0)-- (10,1);
%\draw [line width=0.5pt] (0.,10)-- (10,10);
%\draw [line width=0.5pt] (10,10)--(8,8);
%\draw [line width=0.5pt] (10,0)--(8,0);

%\draw [line width=1pt,color=cyan] (1,1)--(10,10);
%\draw [line width=0.5pt] (1,1)--(0,0);
%\draw [line width=1pt,color=cyan,dashed] (1,1)--(10,1);
%\draw [line width=1pt,color=cyan,dashed] (8,8)--(8,0);
\draw [line width=0.5pt] (0,0)--(10,0);

\fill[line width=1pt,color=gray,fill=gray, opacity=0.4] (0,3) -- (6,9) -- (6,10) --(0,10)--cycle;
\fill[line width=1pt,color=red,fill=red, opacity=0.4] (0,1) -- (6,7) -- (6,0) --(0,0)--cycle;

% \draw (10,1) node[color=cyan]{$\circ$};
% \draw (1,1) node[color=cyan]{$\circ$};
%\draw (10,10) node[color=cyan]{$\circ$};
%%%
\draw (-0.75,-0.75) node{\small $0$};
\draw (10,-0.75) node{\small $\kappa$};
\draw (-0.75,10) node{\small $\sigma$};

\draw (2.7,6.9) node[rotate=45]{\tiny $\!\frac{\sigma\! -\!\kappa}{\mathcal{N}} \geq  \frac{1}{q}-\!\!\frac{1}{p'}\!$};

\draw (3,2.8) node[rotate=45]{\tiny $\!\frac{\sigma\! -\!\kappa}{3} \leq \frac{1}{q}-\!\!\frac{1}{p'}\!$};

% \draw [-latex,line width=0.5pt, color=red,] (1.5,4.5)-- (.5,5.5);
% \draw [-latex,line width=0.5pt, color=red,] (3,6)-- (2,7);
% \draw [-latex,line width=0.5pt, color=red,] (4.5,7.5)-- (3.5,8.5);
% \draw [-latex,line width=0.5pt, color=red,] (1.7,2.5)-- (3.5,.4);
%%%%
%%%
\draw [line width=1pt,color=red] (0,1)--(6,7);
%\draw [line width=1pt,color=cyan] (0,1)--(0,10);
\draw [line width=1pt,color=black] (0,3)--(6,9);
%\draw [line width=1pt,color=black] (0,1)--(0,3);

\draw [line width=.5pt,color=black,dashed] (6,0)--(6,10);
%\draw[line width=1pt,color=magenta] (0,1)--(6,7);

%\fill[line width=1pt,color=cyan,fill=cyan, opacity=0.2] (0,1) -- (6,7) -- (6,9) --(0,3)--cycle;

\draw (0,1) node[color=black]{$\circ$};
\draw (6,7) node[color=black]{$\circ$};
\draw (6,9) node[color=black]{$\circ$};
\draw (0,3) node[color=black]{$\circ$};
%\draw (6,0) node[color=cyan]{$\circ$};
%\draw (0,7) node[color=cyan]{$\circ$};

 \draw (9.2,7) node{\tiny $(\frac{\mathcal{N}}{p'}, \frac{\mathcal{N}-3}{p'}\!+\!\frac{3}{q})$};
 \draw (8,9) node{\tiny $(\frac{\mathcal{N}}{p'},\frac{\mathcal{N}}{q})$};
%\draw (-1,7) node{\tiny $\frac{n}{q}$};
%\draw (6,-1) node{\tiny $\frac{n}{p'}$};
%%%
%%%%
 
% \draw (10.7,1) node{\small $\frac{\beta}{n}$};
% \draw (1,-1) node{\small $\frac{\beta}{n}$};
\end{tikzpicture}

% \begin{tikzpicture}[line cap=round,line join=round,>=triangle 45,x=1.0cm,y=1.0cm, scale=0.33]
% \clip(-3,-2) rectangle (15,11);
% \draw [line width=0.5pt] (0,0)-- (0,10);
% \draw [line width=0.5pt] (0,0.)-- (10,0);
% \draw [line width=0.5pt] (10,10.)-- (10,0);
% \draw [line width=0.5pt] (0.,10)-- (10,10);
% \draw [line width=0.5pt] (0.,0)-- (1,1);
% \draw [line width=0.5pt] (8.,8)-- (10,10);
% \draw [line width=0.5pt] (5,0)-- (6,1);
% \draw [line width=0.5pt] (10,5)-- (8,3);

% %%%
% \draw [line width=1pt,color=cyan] (0,1)--(7,8);

% %%%

% \draw [line width=1pt,color=cyan] (1,1)--(8,8);
% \draw [line width=1pt,color=cyan] (5,0)--(8,3);
% \draw [line width=1pt,color=cyan,dashed] (8,8)--(8,3);
% \draw [line width=1pt,color=cyan,dashed] (1,1)--(6,1);

% \fill[line width=1pt,color=cyan,fill=cyan, opacity=0.1] (1,1) -- (6,1) -- (8,3) -- (8,8) --cycle;

% \draw (5,0) node[color=cyan]{$\circ$};
% \draw (8,3) node[color=cyan]{$\circ$};
% \draw (1,1) node[color=cyan]{$\circ$};
% \draw (8,8) node[color=cyan]{$\circ$};

% \draw (-0.75,-0.75) node{\small $0$};
% \draw (10,-0.75) node{\small $1$};
% \draw (-0.75,10) node{\small $1$};

% \draw (5,-1) node{\small $\frac{\sigma-\alpha-\beta}{n}$};
% \draw (12.5,5) node{\small $1\!-\!\frac{\sigma-\alpha-\beta}{n}$};
% %\draw (8,-1) node{\small $1\!-\!\frac{\alpha}{n}$};
% \draw (10.7,1) node{\small $\frac{\beta}{n}$};
% \draw (11.2,8) node{\small $1\!-\!\frac{\alpha}{n}$};
% %\draw (1,-1) node{\small $\frac{\beta}{n}$};

% \draw (8,2) node[rotate=45]{\tiny $\frac{1}{q}\!\!=\!\!\frac{1}{p}\!-\!\frac{\sigma\!-\!\alpha\!-\!\beta}{n}$};
% \end{tikzpicture}
            \noindent\subcaption{$\zeta = 0$, $\mathcal{N}> 3$ and $q< p'$- \\ no admissible region exists}
            \label{Fig_Pitt_mjN3}
        \end{subfigure}
        \caption{Admissible regions of $\sigma$ and $\kappa$ for shifted and unshifted Pitt's inequality for the modified  Jacobi transform $\tilde{\mathcal{J}}$.  Here 
        $\mathcal{N}:= 2(\alpha+1)$.}
         \label{pitt_mJ_Fig}
    \end{figure}
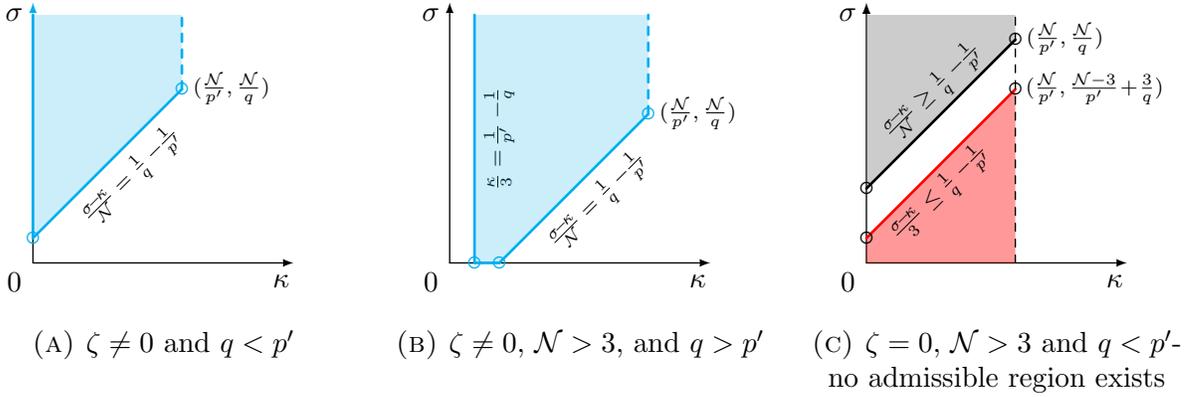 
    \begin{remark}
The result of the Corollary \ref{cor_char_pitt_mJ} is illustrated in Figure~\ref{pitt_mJ_Fig}. More specifically, when $\zeta \neq 0$, the blue region in Figures~\ref{pitt_mJ_Fig} ((A) and (B)) represents the admissible pairs $(\sigma, \kappa)$ for which the shifted Pitt's inequality holds, across different values of $p$ and $q$. Figure~\ref{pitt_mJ_Fig}(C) depicts the case $\zeta = 0$, where two necessary conditions, namely, $\sigma - \kappa \geq 2(\alpha + 1)(1/q - 1/p')$ (shown in gray) and $\sigma - \kappa \leq 3(1/q - 1/p')$ (shown in red), imply that no admissible values of $\sigma$ and $\kappa$ can satisfy both constraints, as the corresponding regions do not intersect. A similar phenomenon was observed in the case of symmetric spaces (see Remark~\ref{rem_prob_lam_0}). We note that this type of obstruction for $\zeta = 0$ in these non-Euclidean settings stems from the distinctive behavior of Harish-Chandra's $\mathbf{c}$-function near zero and infinity.
\end{remark}
We now use this result of shifted Pitt's inequality for the modified Jacobi transforms, along with sharp estimates of the Jacobi function $\varphi_0$, to establish shifted Pitt's inequality for the standard Jacobi transform.
 
\begin{theorem}\label{thm_pitt's_stand_J}
       Let $\zeta \geq 0$ and $1 < p \leq 2$ and $p\leq q < \infty$.   Assume that $\sigma, \kappa \geq 0$ are such that they satisfy the hypothesis of Theorem \ref{thm_pitt's_mod_J}. Then the following shifted Pitt's inequality holds 
    \begin{align}\label{eqn_pitts_stanJ}
         \left( \int_0^{\infty} |{\mathcal{J}} f(\lambda)|^q  (\lambda^2 +\zeta^2)^{-\frac{\sigma q}{2}} n(\lambda) \, d\lambda \right)^{1/q} \leq C   \left( \int_0^{\infty} | f(t)|^p  t^{\kappa p} {m}(t) \, dt \right)^{1/p}
    \end{align}
    for all $f \in C_c^{\infty}(\R_+)$.
    Moreover, in the special case when $p = q = 2$, we have the following result
    \begin{enumerate} 
    \item\label{en_pitt_J_2,2} Suppose $\zeta \neq 0$, then \eqref{eqn_pitts_stanJ} holds if and only if $0 \leq \kappa < \alpha + 1$ and $\sigma \geq \kappa$. 
    \item\label{en_pitt_J_2,2,0} Suppose $\zeta = 0$ and $\kappa \geq 0$, then \eqref{eqn_pitts_stanJ} holds if and only if $\kappa = \sigma < \min\{2(\alpha + 1), 3\}/2$. 
    \end{enumerate}
\end{theorem}
\begin{proof} We will show that for $1 < p \leq 2$, the inequality \eqref{eqn_pitts_modi} implies \eqref{eqn_pitts_stanJ}. As an application of Theorem \ref{thm_pitt's_mod_J}, the theorem will then follow. So, let us suppose that \eqref{eqn_pitts_modi} holds for all $f \in C_c^{\infty}(\mathbb{R}_+)$, that is,
\begin{align*}
   \left( \int_0^{\infty} |\tilde{\mathcal{J}} f(\lambda)|^q  (\lambda^2+\zeta^2)^{-\frac{\sigma q}{2}} n(\lambda) \, d\lambda \right)^{1/q} \lesssim   \left( \int_0^{\infty} | f(t)|^p  t ^{-\kappa p} \tilde{m}(t) \, dt \right)^{1/p} .  
\end{align*}
Now plugging relation $  \tilde{\mathcal{J}} f(\lambda) = \mathcal{J}\left(f\cdot \varphi_0\right)(\lambda)$  from \eqref{eqn_rel_J_mJ} in the equation above, it follows that
\begin{align*}
   \left( \int_0^{\infty} \left|\mathcal{J}\left(f\cdot \varphi_0 \right)(\lambda) \right|^q  (\lambda^2+\zeta^2)^{-\frac{\sigma q}{2}}  n(\lambda) \, d\lambda \right)^{1/q} \lesssim   \left( \int_0^{\infty} | f(t)|^p  t ^{-\kappa p} \tilde{m}(t) \, dt \right)^{1/p} , 
\end{align*}
which implies
\begin{align}\label{eqn_pit_mJ->J}
   \left( \int_0^{\infty} |\mathcal{J}f(\lambda) |^q  (\lambda^2+\zeta^2)^{-\frac{\sigma q}{2}}  n(\lambda) \, d\lambda  \right)^{1/q} \lesssim    \left( \int_0^{\infty} | f(t) |^p  t ^{-\kappa p} \varphi_0(t)^{-p} \tilde{m}(t) \, dt \right)^{1/p} . 
\end{align}
The asymptotic estimate of $\varphi_0 \asymp (1+t)e^{-\rho t}$, gives us the following for $1\leq p\leq 2$
\begin{align}\label{eqn_phi-p}
    \varphi_0(t)^{-p} \tilde{m}(t) \asymp \begin{cases}
        1\cdot t^{2\alpha+1} \asymp m(t), \quad &\text{if } t\leq 1,\\
         t^{2-p} e^{\rho p t} \lesssim e^{2\rho t }\asymp m(t), \quad &\text{if } t\geq 1.
    \end{cases} 
\end{align}
Plugging the estimate \eqref{eqn_phi-p} in \eqref{eqn_pit_mJ->J}, we obtain \eqref{eqn_pitts_stanJ}.

The sufficiency for the case $p = q = 2$ follows from Theorem \ref{thm_pitt's_mod_J} and the preceding calculations. Conversely, the necessity follows from Theorem \ref{thm_nec_pitt_n}.
\end{proof}
By following a similar line of calculation as in the previous theorem and using   $q \geq 2$, we can establish the following version of Pitt's inequality for the inverse Jacobi transform. 
\begin{theorem}\label{thm_pitt_inq_inv_J}
     Let $\zeta \geq 0$ and $1 < p\leq q<\infty$ and $ q \geq 2$.   Assume that $\sigma, \kappa \geq 0$ are such that they satisfy the hypothesis of Theorem \ref{thm_pitt's_mod_J-1}. Then the following shifted Pitt's inequality  
    \begin{align}\label{eqn_pitts_stanI}
         \left( \int_0^{\infty} |{\mathcal{I}} F(t)|^q   t^{-\kappa q} {m}(t) \, dt  \right)^{1/q} \leq C   \left( \int_0^{\infty} | F(\lambda)|^p (\lambda^2 +\zeta^2)^{-\frac{\sigma p}{2}} n(\lambda) \, d\lambda \right)^{1/p}
    \end{align}
   holds for all $F\in C_c^{\infty}(\R_+)$.
    Moreover, in the special case when $p = q = 2$, we have the following result
    \begin{enumerate} 
    \item\label{en_pitt_I_2,2} Suppose $\zeta \neq 0$, then \eqref{eqn_pitts_stanI} holds if and only if $0 \leq \sigma < \alpha + 1$ and $\sigma \leq \kappa$. 
    \item\label{en_pitt_I_2,2,0} Suppose $\zeta = 0$ and $\sigma \geq 0$, then \eqref{eqn_pitts_stanI} holds if and only if $\sigma = \kappa < \min\{2(\alpha + 1), 3\}/2$. 
    \end{enumerate}   
\end{theorem}

\begin{remark}\label{rem_after_J_pit}
\begin{enumerate}
\item We would like to mention that the shifted Pitt's inequality result for the case $p = q = 2$ in Theorem \ref{thm_pitt's_stand_J} can be extended to the range $1 < p \leq q \leq p'$  ($1 < p \leq 2$) using interpolation techniques similar to those employed in Theorem \ref{thm_pitt_X_n}. Specifically, by following the approach in Lemma \ref{lem_plaey_inq}, we can obtain a Hardy–Littlewood–Paley type inequality in the context of  the Jacobi transform for $1 < p \leq 2$: 
\begin{align}\label{eqn_paley_ineq_J}
\left( \int_0^{\infty} |\mathcal{J} f(\lambda)|^p (\lambda^2 + \zeta^2)^{-\frac{\sigma p}{2}} n(\lambda) \, d\lambda \right)^{1/p} \leq C \left( \int_0^{\infty} |f(t)|^p   m(t) \, dt \right)^{1/p}. \end{align} This inequality holds when $\zeta > 0$ provided that $\sigma \geq 2(\alpha + 1)(2/p - 1)$, and for $\zeta = 0$ if $2(\alpha + 1)(2/p - 1) \leq \sigma \leq 3(2/p - 1)$. 

By applying the interpolation method as in Theorem \ref{thm_pitt_X_n}, and interpolating the inequality above with the $p = q = 2$ case in Theorem \ref{thm_pitt's_stand_J} (\eqref{en_pitt_J_2,2} and \eqref{en_pitt_J_2,2,0}), one can derive the   Pitt-type inequalities for the range $1<p\leq q\leq p'$. However, this method does not extend to the case $q > p'$, as the analogue of the Hardy–Littlewood–Paley inequality cannot hold for $p > 2$.

\item From Theorem \ref{thm_nec_pitt_n}, we observe that for inequality \eqref{eqn_paley_ineq_J} to hold with $\zeta = 0$, it is necessary that $2(\alpha + 1)\left({2}/{p} - 1\right)\leq \sigma-\kappa < {3}/{p}$. This condition indicates that \eqref{eqn_paley_ineq_J} does not hold for all $1 < p \leq 2$ when $\zeta = 0$, unless $2(\alpha + 1) \leq 3$. In particular, for a rank one symmetric space of noncompact type with dimension $n$, this result implies that \eqref{eqn_inq_paley_lem} may fail to hold for all $1 < p \leq 2$ when $\zeta = 0$, unless $n \leq 3$.
\item\label{rem_problem_exp_space} 
We also point out that if the non-increasing rearrangement method were applied directly to the Jacobi transform without first modifying it to account for exponential volume growth,  we would be forced to impose an additional condition:
$ {\sigma} \leq 3( {1}/{p}  + {1}/{q}-1)$.
This arises because, in such a setting, for the weight function $v_\kappa(s) = s^{\kappa}$ $\kappa \geq 0$, the non-increasing rearrangement of $1/v_\kappa$ with respect to the measure $m(t)dt$ behaves like $(\log t)^{-\kappa}$ (instead of $t^{-\kappa/3}$ in the modified case). As a consequence, when attempting an analogous estimate to that in \eqref{eqn_est_suff_r>1}, the lack of polynomial decay in the rearrangement forces us to assume the additional hypothesis $ {\sigma} \leq 3( {1}/{p}  + {1}/{q}-1)$ in order to prove an analogue of Theorem \ref{thm_pitt's_stand_J}.
\item We conclude this section with a characterization of the weights for which the shifted Pitt's inequality holds for the standard Jacobi transform. This result follows as a corollary of Theorem~\ref{thm_nec_pitt_n} and Theorem~\ref{thm_pitt's_stand_J}. While a similar characterization also holds for the special case $p = 2 < q$, we state the result only for the range $1 < p \leq q \leq p'$ for simplicity and leave the remaining case to the interested reader. This corollary also recovers the result for rank one symmetric spaces in the $K$-biinvariant case, as presented in Corollary~\ref{cor_char_pitt_X<p'}.
\end{enumerate} 
\end{remark}
\begin{corollary}
     Let $\zeta >0$ and $1 < p \leq q\leq p'$.   Suppose that $\sigma \in \R, \kappa \geq 0$. Then the shifted   Pitt's inequality \eqref{eqn_pitts_stanJ} holds if and only  if $ p\leq 2$,
     \begin{align*}
            \sigma-\kappa \geq 2(\alpha+1)\left( \frac{1}{p}+\frac{1}{q}-1\right),
      \end{align*}
 and $\kappa p' < 2(\alpha+1)$.
\end{corollary}

\section{Heisenberg-Pauli-Weyl type uncertainty inequalities}\label{sec_uncn_principle} 
In this section, we develop a family of uncertainty principles on symmetric spaces of noncompact type, with particular emphasis on $L^p$ HPW-type inequalities. Motivated by earlier work on stratified Lie groups \cite{CRS07, CCR15}, we adapt and extend these ideas to the setting of symmetric spaces, where new phenomena arise due to the spectral gap of the Laplacian and the distinctive geometry at infinity. Our approach to $L^2$ HPW-type inequalities builds on shifted Pitt's inequality from  Section~\ref{sec_pitt_X} and Section \ref{sec_jac_transform}, while the $L^p$ analogues are derived using Hardy-type estimates involving the modified Laplacians $\mathcal{L}_{p_0}= \mathcal{L}+(|\rho|^2-|\rho_{p_0}|^2) I$, where we recall that $|\rho_{p_0}|=|2/{p_0}-1||\rho|$. Furthermore, we establish Landau–Kolmogorov-type inequalities for these operators, which allow us to formulate generalized HPW inequalities. We begin by proving the $L^2$ HPW inequality stated in Corollary~\ref{cor_unc_2,2_F} from the introduction.

\begin{proof}[\textbf{Proof of Corollary \ref{cor_unc_2,2_F}}] 
Suppose $p_0 \in [1,\infty)$ and $\gamma, \delta \in (0,\infty)$. Recalling from \eqref{inq_pit_p,2_n} and \eqref{inq_pit_p,2}, we note that the following inequality holds for all $f \in C_c^{\infty}(\X)$
\begin{align}\label{eqn_pitt's_p=q=2_n}
      \int_{\fa} \int_{B}|\widetilde{f}(\lambda,b)|^2  (|\lambda|^2+|\rho_{p_0}|^2)^{-{\sigma} } |\bfc(\lambda)|^{-2}\,db\, d\lambda \leq  \int_{\X} |x|^{2 \kappa} |f(x)|^2 dx
    \end{align}
   provided $\sigma \geq \kappa \geq 0$ and $\kappa < n/2$ when $p_0 \neq 2$, and $\sigma = \kappa < \min\{n/2,\nu/2\}$ when $p_0 = 2$. By applying Plancherel’s formula, we can write
     \begin{align*}
        \|f\|_{L^2(\X)}^2&=\frac{1}{|W|} \int_{\fa}   \int_{B} |\widetilde{f}(\lambda, b)|^2  |\bfc(\lambda)|^{-2}\, db \,d\lambda\\
        &= \frac{1}{|W|} \int_{\fa} \|\widetilde{f}(\lambda,\cdot)\|_{L^2(B)}^{\frac{2\gamma}{\gamma+\delta}} (|\lambda|^2+|\rho_{p_0}|^2)^{\frac{\gamma \delta}{(\gamma+\delta)}} \|\widetilde{f}(\lambda,\cdot)\|_{L^2(B)}^{\frac{2\delta}{2(\gamma+\delta)}} (|\lambda|^2+|\rho_{p_0}|^2)^{-\frac{\gamma \delta} {\gamma+\delta}}|\bfc(\lambda)|^{-2} d\lambda.
    \end{align*}
    Applying H\"older's inequality, we derive
    \begin{multline*}
          \|f\|_{L^2(\X)}^2\lesssim \left(\int_{\fa} \|\widetilde{f}(\lambda,\cdot)\|_{L^2(B)}^{2} (|\lambda|^2+|\rho_{p_0}|^2)^{{\delta} } |\bfc(\lambda)|^{-2}\, d\lambda\right)^{\frac{\gamma}{\gamma+\delta}} \\ \cdot \left(\int_{\fa} \|\widetilde{f}(\lambda,\cdot)\|_{L^2(B)}^{2} (|\lambda|^2+|\rho_{p_0}|^2)^{-{\gamma} } |\bfc(\lambda)|^{-2}\, d\lambda\right)^{\frac{\delta}{\gamma+\delta}} .
    \end{multline*}
    Now, taking $\sigma = \kappa = \gamma$ in \eqref{eqn_pitt's_p=q=2_n} and applying it to the second term on the right-hand side, we obtain the inequality  
    \begin{align}\label{eqn_unc_2,2_F_n} 
          \|f\|_{L^2(\X)} \leq C \| (|\lambda|^2+ |\rho_{p_0}|^2)^{\frac{\delta}{2}}\widetilde{f}\|_{L^2(\fa \times B, |\bfc(\lambda)|^{-2}d\lambda db)}^{\frac{\gamma} {\gamma+\delta}}  \| |x|^{\gamma} f\|_{L^2(\X)}^{\frac{\delta} {\gamma+\delta}} 
    \end{align}
    for $\gamma < n/2$ when $p_0 \neq 2$, and for $\gamma < \min\{n/2, \nu/2\}$ when $p_0 = 2$. 
To extend \eqref{eqn_unc_2,2_F_n} to the full range of $\gamma > 0$, we proceed as follows. Suppose $\gamma \geq \min\{n/2, \nu/2\}$, and choose any $\gamma_0 <\min\{n/2, \nu/2\}$.
%To extend \eqref{eqn_unc_2,2_F} to the remaining range of $\gamma$, we will proceed as follows. Suppose $\gamma \geq \min\{n/2, \nu/2\}$, then we choose any $\gamma_0 < \min\{n/2, \nu/2\}$. 
Then, for every $\epsilon > 0$, we have 
    \begin{align}
        \frac{|x|^{\gamma_0}}{\epsilon^{\gamma_0}}\leq 1+ \frac{|x|^{\gamma}}{\epsilon^{\gamma}},
    \end{align}
    from which it follows that 
    \begin{align*}
        \||x|^{\gamma_0}f\|_{L^2(\X)}\leq \epsilon^{\gamma_0} \|f\|_{L^2(\X)}+ \epsilon^{\gamma_0-\gamma} \||x|^{\gamma}f\|_{L^2(\X)}.
    \end{align*}
   Putting $\epsilon^{\gamma} = \||x|^{\gamma }f\|_{L^2(\X)}/\|f\|_{L^2(\X)}$  in the inequality above, yields that
     \begin{align*}
        \||x|^{\gamma_0}f\|_{L^2(\X)}\leq C   \|f\|_{L^2(\X)}^{1-\frac{\gamma_0}{\gamma}} \||x|^{\gamma}f\|_{L^2(\X)}^{\frac{\gamma_0}{\gamma}}.
    \end{align*}
    Substituting the estimate above into \eqref{eqn_unc_2,2_F_n} with $\gamma$ replaced by $\gamma_0$, we obtain
    \begin{align*}
        \|f\|_{L^2(\X)} \leq C \| (|\lambda|^2+ |\rho_p|^2)^{\frac{\delta}{2}}\widetilde{f}\|_{L^2(\fa \times B, |\bfc(\lambda)|^{-2}d\lambda dk)}^{\frac{\gamma_0} {\gamma_0+\delta}} \|f\|_{L^2(\X)}^{\left(1-\frac{\gamma_0}{\gamma} \right)\frac{\delta} {\gamma_0+\delta}} \||x|^{\gamma}f\|_{L^2(\X)}^{\frac{\gamma_0 \delta} {\gamma(\gamma_0+\delta)}} \, .
    \end{align*}
    Reorganizing the expression above, it follows that
    \begin{align*}
       \|f\|_{L^2(\X)}^{\frac{\gamma_0 (\gamma+\delta)} {\gamma(\gamma_0+\delta)}}  \leq C \| (|\lambda|^2+ |\rho_p|^2)^{\frac{\delta}{2}}\widetilde{f}\|_{L^2(\fa \times B, |\bfc(\lambda)|^{-2}d\lambda dk)}^{\frac{\gamma_0} {\gamma_0+\delta}} \||x|^{\gamma}f\|_{L^2(\X)}^{\frac{\gamma_0 \delta} {\gamma(\gamma_0+\delta)}},
    \end{align*}
    which gives the desired inequality \eqref{eqn_unc_2,2_F} for all $\gamma, \delta > 0$, concluding our result.
\end{proof}
Uncertainty inequalities of the form
 $\|f\|_{L^2}^2 \leq C \|v^{1/p}f\|_{L^p} \|w^{1/q} f\|_{q}$ are also of significant interest, where $v$ and $w$ are suitable weight functions and $1\leq p,q<\infty$; see \cite[p. 214]{FS97} and \cite{CP84} for further discussion. We now use the shifted Pitt's inequality to derive such an analogue of weighted uncertainty inequalities.

\begin{corollary}
Let $ \X$ be a symmetric space of noncompact type of dimension $n$. Let $1<p\leq 2$, $p\leq q'\leq p'$ and $\zeta \geq 0$.  Suppose $\zeta \neq 0$, $\sigma$, $\kappa\geq 0 $ be such that $ \kappa <n/p'$ and $\sigma -\kappa\geq n(1/p-1/q)$.  Then for every $f\in C_c^{\infty}(\X)$,
     \begin{align}\label{eqn_uncertain_2,p}
        \|f\|_{L^2(\X)}^2 &\leq C \left( \int_{\fa}\|\widetilde{f}(\lambda, \cdot)\|_{L^2(B)}^{q}   (|\lambda|^2+\zeta^2)^{\frac{\sigma q}{2}} |\bfc(\lambda)|^{-2}\,d\lambda \right)^{\frac{1}{q}}  \left(\int_{{\X}}{\left|  f(x)\right| }^p  |x|^{\kappa p} \,dx\right)^{\frac{1}{p}}.
    \end{align}
\end{corollary}
Furthermore the same statement will be true for $\zeta=0$, if in addition we assume $\sigma -\kappa \leq \nu(1/p-1/q) $.
\begin{proof}
    We can write from Plancherel's theorem and applying H{\"o}lder's inequality successively 
    \begin{align*}
        \|f\|_{L^2(\X)}^2&= \int_{\fa} \int_{B} \widetilde{f}(\lambda, k) \overline{\widetilde {f}(\lambda,k)} |\bfc(\lambda)|^{-2}\, dk \,d\lambda\\
        & = \int_{\fa} \|\widetilde{f}(\lambda, \cdot)\|_{L^2(B)}  (\lambda^2+ \zeta^2)^{\frac{\sigma}{2}} \|\widetilde {f}(\lambda,\cdot)\|_{L^2(B)} (\lambda^2+ \zeta^2)^{-\frac{\sigma}{2}} |\bfc(\lambda)|^{-2}\,d\lambda\\
        &\leq \left( \int_{\fa} \|\widetilde{f}(\lambda, \cdot)\|_{L^2(B)}^{q}  (\lambda^2+ \zeta^2)^{\frac{\sigma q}{2}}  |\bfc(\lambda)|^{-2}\,d\lambda \right)^{\frac{1}{q}} \\
        & \hspace{2cm} \cdot \left( \int_{\fa} \|\widetilde{f}(\lambda, \cdot)\|_{L^2(B)}^{q'}  (\lambda^2+ \zeta^2)^{-\frac{\sigma q'}{2}}  |\bfc(\lambda)|^{-2}\,d\lambda \right)^{\frac{1}{q'}}.
    \end{align*}
    By applying the shifted Pitt's inequality result from Theorem \ref{thm_pitt_X_n} to the second term on the right-hand side, we obtain \eqref{eqn_uncertain_2,p}
\end{proof}

\begin{corollary}
    Let $\zeta \neq 0$, $q\geq 2$, and $1<p\leq q'$. Suppose $\sigma, \kappa \geq 0$ be such that $ \kappa <n/p'$ and $\sigma \geq \kappa+n(1/p-1/q)$. Then we have for all $f \in C_c^{\infty}(\X)$
    \begin{align}\label{eqn_unc_p,q1_n}
        \|f\|_{L^q(\X)}^{q'} \leq C   \left( \int_{{\fa}} \left(\int_B|\widetilde{f}(\lambda, b)|^{2} \, db\right)^{\frac{q'}{2}} (|\lambda|^2+\zeta^2)^{\frac{\sigma q}{2}} |\bfc(\lambda)|^{-2} \,d \lambda \right)^{\frac{1}{q }}  \left( \int_{\X}|f(x)|^p  |x|^{\kappa p}\,dx \right)^{\frac{1}{p}}.
    \end{align}
\end{corollary}
\begin{proof}
    By using dual of Hausdorff-Young  inequality from \eqref{eqn_HYP_dual-int}, we get
    \begin{align*}
         \|f\|_{L^q(\X)} \leq   \left( \frac{1}{|W|}\int_{\mathbb{\fa}}  \|\widetilde{f}(\lambda, \cdot)\|_{L^2(B)}^{q'} |\bfc(\lambda)|^{-2} \,d \lambda \right)^{\frac{1}{q'}}.
    \end{align*}
    Writing $q'= \frac{1}{q-1}+1$ and applying H\"older's inequality, we get from the inequality above
    \begin{multline*}
         \|f\|_{L^q(\X)}^{q'} \leq   \left( \frac{1}{|W|}\int_{\mathbb{\fa}}  \|\widetilde{f}(\lambda, \cdot)\|_{L^2(B)}^{q'} (|\lambda|^2+\zeta^2)^{\frac{\sigma q}{2}}|\bfc(\lambda)|^{-2} \,d \lambda \right)^{\frac{1}{q}}\\
         \cdot \left( \frac{1}{|W|}\int_{\mathbb{\fa}} \ \|\widetilde{f}(\lambda, \cdot)\|_{L^2(B)}^{q'}  (|\lambda|^2+\zeta^2)^{-\frac{\sigma q'}{2}}|\bfc(\lambda)|^{-2} \,d \lambda \right)^{\frac{1}{q'}}.
    \end{multline*}
    Since $1<p\leq q'$,  using Theorem \ref{thm_pitt_X_n} we establish \eqref{eqn_unc_p,q1_n}.
\end{proof}
Using Theorem \ref{thm_pitt_X_n} with $q=p'$, $\zeta =|\rho_{p_0}|$, and taking $(-\mathcal{L}_{p_0})^{\sigma/2}f$ instead of $f$ in \eqref{inq_pitt_pol}, we obtain the following corollary.
\begin{corollary}
    Let $ \X$ be a symmetric space of noncompact type of dimension $n$, $p_0\in [1,2)$, and $p \in(1, 2],$. Suppose  that $0\leq \kappa <n/{p'}$. Then for any $\sigma \geq  \kappa$ , we have  the following inequality 
\begin{align}
         \left( \int_{{\fa}} \left(\int_B|\widetilde{f}(\lambda, b)|^{2} \, db\right)^{\frac{p'}{2}}   |\bfc(\lambda)|^{-2} \,d \lambda \right)^{\frac{1}{p'}}  \leq C \left( \int_{\X}|(-\mathcal{L}_{p_0})^{\sigma/2}f(x)|^p  |x|^{\kappa p}\,dx \right)^{\frac{1}{p}}
    \end{align}
    for all $f \in C_c^{\infty}(\X)$,
    which can be seen as a generalized version of the Hausdorff-Young inequality \eqref{eqn_HY_X,n}.
\end{corollary}

\subsection{$L^p$ Heisenberg-Pauli-Weyl uncertainty inequality}
We now proceed to establish the $L^p$-type HPW uncertainty inequality. We begin by recalling a class of Hardy-type inequalities for the modified Laplacians $\mathcal{L}_{p_0}$ in \cite{KRZ_23}.
For other forms of Hardy inequalities on symmetric spaces of noncompact type, see also \cite{BP22}. We have the following result from  \cite[Theorem 1.2]{KRZ_23} by putting $p=q$ and  $\alpha=0$. 
    \begin{lemma}\label{cor_hardy_inq_X}
    Let $\X$ be a general symmetric space of dimension $n\geq 2$ and pseudo-dimension $\nu \geq 3$. Suppose that $\zeta\geq 0$, $\sigma>0$ ($0<\sigma <\nu$ if $\zeta=0$) and $\beta \in \R$. Then the inequality 
    \begin{align}\label{eqn_hardy_inq_X}
    \| |\cdot |^{-\beta} (-\mathcal{L} - |\rho|^2 +\zeta^2))^{-\frac{\sigma}{2}} f \|_{L^{p}(\X)} \lesssim \|f\|_{L^p(\X)}    \quad \forall f\in L^p(\X),
    \end{align}
     holds if $1\leq p\leq \infty$, $\beta< {n}/{p}$ when $\beta>0$, $0\leq \beta \leq \sigma$, and one of the following conditions is met:
    \begin{enumerate}
        \item if $\zeta >|\rho|$;
        \item if $0 \leq \zeta \leq |\rho|$, then 
        \begin{align}\label{eqn_ran_of_p,c}
             \frac{1}{2}-\frac{\zeta}{2|\rho|}<\frac{1}{p}<\frac{1}{2}+\frac{\zeta}{2|\rho|}.
        \end{align}
        \item if $\zeta=0$ and $p=2$, then $\sigma<\nu/2$ and $\sigma\leq \beta$.
    \end{enumerate}
    \end{lemma}
\begin{lemma}\label{lem_Lp_unc_X_p_0}
 Let $\X$ be a general symmetric space of noncompact type and  $ 1\leq p_0<2$. Suppose that $\gamma_0, \delta_0 \in (0,\infty) $, $p>1$, $q\geq 1$, and $r\in (p_0,p_0')$, such that
\begin{align*}
    \delta_0 <\frac{n}{r_0} \quad  \text{and} \quad  \frac{\gamma_0+\delta_0}{p}= \frac{\delta_0}{q}+\frac{\gamma_0}{r}.
\end{align*}
Then for any $\sigma \geq \delta_0$, we have
\begin{align}\label{eqn_Lp_uncn_X_p_0}
    \|f\|_{L^{p}(\X)} \leq C \| |x|^{\gamma_0}f\|_{L^{q}(\X)}^{\frac{\delta_0}{\gamma_0+\delta_0}} \| (-\mathcal{L}_{p_0})^{\frac{\sigma}{2}}f\|_{L^{r}(\X)}^{\frac{\gamma_0}{\gamma_0+\delta_0}}
\end{align}
for all $f \in C_c^{\infty}(\X)$.
\end{lemma}
\begin{proof}
    Setting $\theta = \gamma_0/(\gamma_0 + \delta_0)$, we write
    \begin{align*}
        \int_{\X} |f(x)|^{p} \,dx=\int_{\X} \left(|x|^{\theta \delta_0} |f(x)|^{\theta} \right)^{p}  \left(|x|^{- {\theta \delta_0}} |f(x)|^{1-\theta} \right)^{p} \, dx.
    \end{align*}
   Let us  choose $s$ such that $q= \theta ps$, then $r =(1-\theta) p s'$. Applying H\"older’s inequality with exponent $s$ to the integrand and then taking the $p$th root, we obtain
   \begin{equation}\label{eqn_Lpf<xbf}
    \begin{aligned}
        \|f\|_{L^p(\X)} &\leq \left(\int_{\X} \left(|x|^{\theta \delta_0 } |f(x)|^{\theta} \right)^{ps }  \, dx\right)^{\frac{1}{ps}}  \left( \int_{\X}\left(|x|^{- {\theta \delta_0}} |f(x)|^{1-\theta} \right)^{ps'} \, dx \right)^{\frac{1}{ps'}}\\
        &= \||x|^{\gamma_0}f\|_{L^q(\X)}^{\frac{\delta_0}{\gamma_0+\delta_0}} \||x|^{-\delta_0}f\|_{L^r(\X)}^{\frac{\gamma_0}{\gamma_0+\delta_0}}.
    \end{aligned}
    \end{equation}
Now, applying \eqref{eqn_hardy_inq_X} with $\zeta =|\rho_{p_0}|$ and $\beta=\delta_0$, and using the assumption that $r \in (p_0, p_0')$, we deduce that the following inequality holds for any $\sigma \geq \delta_0$ and $r \in (p_0, p_0')$
 \begin{align*}
       \||x|^{-\delta_0}f\|_{L^r(\X)} \leq C   \| (-\mathcal{L}_{p_0})^{\frac{\sigma}{2}}f\|_{L^{r}(\X)},
 \end{align*}
 for all $f \in C_c^{\infty}(\X)$. Substituting the inequality above in \eqref{eqn_Lpf<xbf}, we establish \eqref{eqn_Lp_uncn_X_p_0}.
\end{proof}
\begin{remark}
   The conclusion of Lemma \ref{lem_Lp_unc_X_p_0} also holds without the additional condition $r \in (p_0, p_0')$ if we replace $-\mathcal{L}_{p_0}$ in \eqref{eqn_Lp_uncn_X_p_0} with $(-\mathcal{L} - |\rho|^2 + \zeta^2)$ for any $\zeta > |\rho|$. 
\end{remark}

We seek an analogue of the classical Landau–Kolmogorov inequality \cite{Kol39, Ste70} for the Laplacian on symmetric spaces of noncompact type. While similar results may be available in broader contexts, we present the full statement and proof here for the sake of completeness. Our approach builds on the method of \cite[Theorem 5.3]{CCR15}, along with sharp estimates for the imaginary powers of the Laplacian on symmetric spaces \cite{Ank90, CGM93, Ion02}. Moreover, to extend the Landau–Kolmogorov inequality beyond the standard Laplacian $\mathcal{L}$ to the more general $L^{p_0}$-Laplacians $\mathcal{L}_{p_0}$, we adapt and generalize the technique developed in  \cite[Theorem 5.3]{CCR15}.
\begin{lemma}\label{lem_LKP_p0_n}
Let $\X$ be a symmetric space of noncompact type, and let $1 \leq p_0 < 2$. Suppose $0 \leq \theta \leq 1$ and $p, q, r \in (p_0, p_0')$ are such that
\begin{align}\label{eqn_p,q,r_LKP}
\frac{1}{p} = \frac{\theta}{q} + \frac{1 - \theta}{r}.
\end{align}
Then, for all $\sigma > 0$, we have
    \begin{align}
        \|(-\mathcal{L}_{p_0})^{\frac{\sigma}{2}}f\|_{L^{p}(\X)} \leq C  \|(-\mathcal{L}_{p_0})^{\frac{\sigma}{2 \theta}}f\|_{L^{q}(\X)}^{\theta} \|f\|_{L^{r}(\X)}^{1-\theta} 
    \end{align}
    for all $f \in C_c^{\infty}(\X)$.
\end{lemma}
\begin{proof}
    Let us define the strip
    \begin{align*}
        S_{\sigma/\theta}= \left\{ z \in \C : 0\leq \Re z \leq {\sigma}/{\theta} \right\}.
    \end{align*}
Now, let us take arbitrary compactly supported function $g \in L^{p'}(\X)$. Then for $z \in S_{\sigma/\theta}$, we define
\begin{align*}
    g_z(x)= \|g\|_{L^{p'}(\X)}^{az+b}g(x) |g(x)|^{cz+d} \quad \text{for all }x\in \X,
\end{align*}
where 
\begin{align*}
    \sigma=\frac{\theta p'}{\sigma} \left(\frac{1}{q}-\frac{1}{p} \right), \quad b=-\frac{p'}{r'}, \quad c=\frac{\theta p'}{\sigma}\left(\frac{1}{p}-\frac{1}{q} \right), \quad \text{and} \quad d=\frac{p'}{r'}-1.
\end{align*}
When $\Re z= \eta$ and 
\begin{align}\label{rel_s_r,p,n}
    \frac{1}{s}= \frac{1}{r}-\frac{\theta \eta}{\sigma}\left(\frac{1}{p}-\frac{1}{q}\right),
\end{align}
then it follows that $\|g_z\|_{L^{s'}(\X)}^{s'}= \|g\|_{L^{p'}(\X)}^{s'(a\eta+b)+p'}=1$. Now let $ f \in C_c^{\infty}(\X) $ be fixed, and define the following analytic function $h: S_{\sigma/\theta} \rightarrow \C$ by
\begin{align}\label{defn_h_z}
    h(z) =e^{z^2}\int_{\X} (-\mathcal{L}_{p_0})^{\frac{z}{2}}f(x) g_z(x)dx.
\end{align}
It is easy to see that $h$ is well defined. For $s \in (p_0,p_0')$ and $y \in \R$,  the following estimate is known
\begin{align}\label{eqn_est_op_Lp_n}
        \|(-\mathcal{L}_{p_0})^{iy}\|_{L^s(\X) \rightarrow L^s(\X)} \leq C e^{\frac{2\pi}{3} |y|},
    \end{align}
 see for instance \cite[Cor. 4.2]{CGM93} \cite[Theorem 9 and (5.2)]{Ion02} (see also \cite{Ank90, Ion03}). We observe from \eqref{rel_s_r,p,n} that since $q,r \in (p_0,p_0')$, $s$ also lies in the interval $(p_0, p_0')$ whenever $0 \leq \eta \leq \sigma/\theta$. Therefore, applying the above estimate for $z = \eta + i y \in S_{\sigma/\theta}$, we get
\begin{align*}
    |h(z)| \leq e^{\eta^2-y^2}\|(-\mathcal{L}_{p_0})^{\frac{iy}{2}}(-\mathcal{L}_{p_0})^{\frac{\eta}{2}}f\|_{L^s(\X)} \|g_z\|_{L^{s'}(\X)}\leq C e^{-y^2+\frac{2\pi}{3} |y|} \| (-\mathcal{L}_{p_0})^{\frac{\eta}{2}}f\|_{L^s(\X)},
\end{align*}
where in the last step we used \eqref{eqn_est_op_Lp_n}. Thus we have obtained $|h(\eta+iy)| \leq C\| (-\mathcal{L}_{p_0})^{\frac{\eta}{2}}f\|_{L^s(\X)}$. Particular, we have from the inequality above and  \eqref{rel_s_r,p,n}
\begin{align*}
    |h(z)| \leq C\begin{cases}
        \|f\|_{L^r(\X)} \quad &\text{if }\Re z=0,\\
         \|(-\mathcal{L}_{p_0})^{\frac{\sigma}{2\theta}}f\|_{L^q(\X)} \quad &\text{if }\Re z=\frac{\sigma}{\theta}.
    \end{cases}
\end{align*}
By the Phragm{\'e}n-Lindel{\"o}f theorem, it follows that
\begin{align}
    |h(\sigma)|\leq C \|(-\mathcal{L}_{p_0})^{\frac{\sigma}{2\theta}}f\|_{L^q(\X)}^{\theta}  \|f\|_{L^r(\X)}^{1-\theta}.
\end{align}
Since
\begin{align*}
     h(\sigma) =e^{\sigma^2}\int_{\X} (-\mathcal{L}_{p_0})^{\frac{\sigma}{2}}f(x) g_{\sigma}(x)dx.
\end{align*}
and $g_{\alpha}$ is an arbitrary compactly supported  simple function on $\X$ with  $L^{p'}(\X)$-norm equal to one, this completes the proof. 
\end{proof}
\begin{remark}
We would like to point out that in the setting of rank one symmetric spaces, Ionescu \cite[Theorem 9]{Ion02} obtained the sharp estimate of $\|(-\mathcal{L}_{p_0})^{iy}\|_{L^s(\X) \rightarrow L^s(\X)}$ for all $s \in [p_0, p_0']$, improving upon an earlier result by Anker and Seeger, who had established similar estimates only for $s \in (p_0, p_0')$. Furthermore, using Ionescu’s result, we can extend the Landau–Kolmogorov inequality for $\mathcal{L}_{p_0}$ in Lemma~\ref{lem_LKP_p0_n} to the full range $p, q, r \in [p_0, p_0']$.
\end{remark}
As an application of the Landau–Kolmogorov inequality, we now extend the uncertainty inequalities from Lemma \ref{lem_Lp_unc_X_p_0} to prove Theorem \ref{thm_Lp_unc_x}. For the reader’s convenience, we restate the theorem below.
\begin{theorem}\label{thm_uncn_Lp_p0,n}
    Let $\X$ be a general symmetric space of noncompact type, and let $1 \leq p_0 < 2$. Suppose $\gamma, \delta  \in (0,\infty)$, and that the exponents $p,q$, and $r$ with $p, r \in (p_0, p_0')$ and $q\geq 1$, which satisfy 
    \begin{align}
          \frac{\gamma+\delta}{p}= \frac{\delta}{q}+\frac{\gamma}{r}.
    \end{align}
Then, for any $\sigma \geq \delta$, we have the following inequality
        \begin{align}\label{eqn_uncn_p0_thm}
            \|f\|_{L^p(\X)} \leq C \| |x|^{\gamma}f\|_{L^q(\X)}^{\frac{\delta}{\gamma+\delta}} \| (-\mathcal{L}_{p_0})^{\frac{\sigma}{2}}f\|_{L^r(\X)}^{\frac{\gamma}{\gamma+\delta}}
        \end{align}
        for all $f \in C_c^{\infty}(\X)$.
\end{theorem}
\begin{proof}
When $\delta < {n}/{r}$, the theorem follows directly from Lemma \ref{lem_Lp_unc_X_p_0}. So, suppose instead that $\delta \geq {n}/{r}$. In this case, choose $\theta \in (0, {n}/{r\delta})$ and define $\delta_0 = \theta \delta$. Then,  $\delta_0 < {n}/{r}$, and by a straightforward computation it follows that
\begin{align}\label{eqn_reln_d,g_n}
    \frac{\gamma+\delta_0}{p}=\frac{\delta_0}{q}+{\frac{\gamma}{r_0}}, \qquad \text{where } \quad \frac{1}{r_0}= \frac{\theta}{r}+\frac{1-\theta}{p}.
\end{align}
Now applying Lemma \ref{lem_Lp_unc_X_p_0}, we get that for any $\sigma \geq  \delta_0$  the following holds 
\begin{align}\label{eqn_unc_p,d<nr}
    \|f\|_{L^p(\X)} \leq C \| |x|^{\gamma}f\|_{L^q(\X)}^{\frac{\delta_0}{\gamma+\delta_0}} \| (-\mathcal{L}_{p_0})^{\frac{\sigma}{2}}f\|_{L^{r_0}(\X)}^{\frac{\gamma}{\gamma+\delta_0}}
\end{align}
for all $f \in C_c^{\infty}(\X)$. By using \eqref{eqn_reln_d,g_n} and applying the Landau–Kolmogorov inequality from Lemma \ref{lem_LKP_p0_n}, we obtain
\begin{align*}
     \|(-\mathcal{L}_{p_0})^{\frac{\sigma}{2}}f\|_{L^{r_0}(\X)} \leq C  \|(-\mathcal{L}_{p_0})^{\frac{\sigma}{2 \theta}}f\|_{L^{r}(\X)}^{\theta} \|f\|_{L^{p}(\X)}^{1-\theta}. 
\end{align*}
Substituting the inequality above into \eqref{eqn_unc_p,d<nr} and reorganizing the terms yields \eqref{eqn_uncn_p0_thm}, thereby completing the proof of the theorem.
\end{proof}
\subsection{Applications of shifted Pitt's inequality for Jacobi transforms}\label{sec_app_jac} We have similar applications of shifted Pitt's inequality for the Jacobi transform as in the previous section. We omit the proofs, as they follow the same arguments presented in earlier subsections,  relying on Theorem \ref{thm_pitt's_stand_J}.

\begin{corollary}
    Let $\gamma, \delta>0 $, and $\rho_p = |2/p-1| \rho$. Then we have the following 
    \begin{align}\label{eqn_unc_J2,2_F}
        \|f\|_{L^2(m)} \leq C \| t^{\gamma} f\|_{L^2(m)}^{\frac{\delta} {\gamma+\delta}}  \| (\lambda^2+ \rho_p^2)^{\frac{\delta}{2}}\mathcal{J}f\|_{L^2(\R_+, |\bfc(\lambda)|^{-2}d\lambda)}^{\frac{\gamma} {\gamma+\delta}}
    \end{align}
    for all $f \in C_c^{\infty}(\R_+)$. Equivalently, we have 
     \begin{align}\label{eqn_unc_J2,2_X}
        \|f\|_{L^2(m)} \leq C \| t^{\gamma} f\|_{L^2(m)}^{\frac{\delta} {\gamma+\delta}}  \|  (-\mathcal{L}_{p})^{\frac{\delta}{2}} f \|_{L^2(m)}^{\frac{\gamma} {\gamma+\delta}},
    \end{align}
    where $\mathcal{L}_{p}= \mathcal{L}+(|\rho|^2-|\rho_p|^2) I$ and $\mathcal{L}$ is the Jacobi Laplacian defined in \eqref{defn_of_L_Jac}. 
\end{corollary}

\begin{corollary}
Let $\zeta\neq 0$, $1<p\leq 2$, and $p\leq q'$.   Suppose  $\sigma,\kappa\geq 0$ be such that $3 (1/q-1/p) \leq \kappa <2(\alpha+1)/p'$ and $\sigma \geq \kappa+ 2(\alpha+1)(1/p-1/q)$.  Then for every $f\in C_c^{\infty}(\R_+)$,
     \begin{align}\label{eqn_uncertain_jac_2,p}
        \|f\|_{L^2(m)}^2 &\leq C \left( \int_{0}^{\infty} |\mathcal{J}f(\lambda)|^{q}   (\lambda^2+\zeta^2)^{\frac{\sigma q}{2}} |\bfc(\lambda)|^{-2}\,d\lambda \right)^{\frac{1}{q}}  \left(\int_{{0}}^{{\infty}}{\left|  f(t)\right| }^p  t^{\kappa p} m(t)\,dt\right)^{\frac{1}{p}}.
    \end{align}
\end{corollary}

\begin{corollary}
    Let $q\geq 2$,  $1<p\leq q'$ and $\zeta\neq 0$.   Suppose $\sigma$, $\kappa$ be such that $0\leq \kappa <2(\alpha+1)/p'$ and $\sigma \geq \kappa+2(\alpha+1)(1/p-1/q)$. Then we have for all $f \in C_c^{\infty}(\R_+)$
    \begin{align}\label{eqn_unc_p,q1}
        \|f\|_{L^q(m)}^{q'} \leq C   \left( \int_{0}^{\infty} |\mathcal{J}f(\lambda)|^{{q'}} (\lambda^2+\zeta^2)^{\frac{\sigma q}{2}} |\bfc(\lambda)|^{-2} \,d \lambda \right)^{\frac{1}{q}}  \left(\int_{{0}}^{{\infty}}{\left|  f(t)\right| }^p  t^{\kappa p} m(t)\,dt\right)^{\frac{1}{p}}.
    \end{align}
\end{corollary}
\begin{corollary}
Let $p \in (1,2]$ and  $p_0 \in [1,2)$.   Suppose that $0\leq \kappa <2(\alpha+1)/{p'}$, then for any $\sigma \geq  \kappa$, we have  the following inequality 
\begin{align}
         \left( \int_{0}^{\infty}  |\mathcal{J}f(\lambda)|^{p'}   |\bfc(\lambda)|^{-2} \,d \lambda \right)^{\frac{1}{p'}}  \leq C \left( \int_{0}^{\infty}|(-\mathcal{L}_{p_0})^{\sigma/2}f(t)|^p  t^{\kappa p} m(t) \,dt\right)^{\frac{1}{p}}
    \end{align}
    for all $f \in C_c^{\infty}(\R_+)$.
    \end{corollary}

\begin{remark}
Analogous applications for other types of Jacobi  transforms may be obtained using Theorem \ref{thm_pitt's_mod_J}, \ref{thm_pitt's_mod_J-1}, and \ref{thm_pitt_inq_inv_J}. Since the methods are similar, we omit the detailed statements and proofs.
\end{remark}
\section{Concluding remarks}
In conclusion, we would like to offer some observations and draw attention to a few open problems that, from our perspective, warrant further investigation.
\begin{enumerate}
    \item As in the Euclidean setting, it is also interesting to establish the Pitt-type inequality for the range $2 < p \leq q < \infty$. However, due to the exponential volume growth of symmetric spaces, it becomes necessary to incorporate exponential weights when $p > 2$. In fact, following a similar computation as in Theorem \ref{thm_nec_pitt_n}, we can show that for $p > 2$, the appropriate weight is $w_\delta = e^{2\rho \delta t}$ with $\delta$   greater than equal to  $({1}/{2} - {1}/{p})$. In this context, we also note that in the rank one case, Ray and Sarkar established a Hardy–Littlewood–Sobolev inequality for $p \geq 2$ in \cite[Theorem 4.11 (ii)]{RS09}. By interpolating their result with \eqref{inq_pit_p,2}, one can derive a version of the Pitt inequality for $p > 2$, albeit with exponential weights appearing on the function side.
    \item As noted earlier in Remark~\ref{rem_nec_suff_mJ}, we assumed $\sigma, \kappa \geq 0$  to compute the non-increasing rearrangement of polynomial weights. It is therefore natural to ask whether this condition is also necessary for the shifted Pitt's inequality to hold for $1 \leq  p \leq q < \infty$ in the setting of modified Jacobi transforms. Moreover, a complete characterization of admissible polynomial weights in shifted Pitt's inequality for Jacobi transforms, and, on Riemannian symmetric spaces, remains an important open problem.  
\end{enumerate}
\section*{Acknowledgments}
The authors thank Pritam Ganguly for some useful discussions on uncertainty principles. TR and MR are supported by the FWO Odysseus 1 grant G.0H94.18N: Analysis and Partial Differential Equations, the Methusalem program of the Ghent University Special Research Fund (BOF), (TR Project title: BOFMET2021000601). TR is also supported by a BOF postdoctoral fellowship at Ghent University BOF24/PDO/025.  MR is also supported by EPSRC grant EP/V005529/7.

\bibliographystyle{alpha}
\bibliography{Reference_PittsSt}

\end{document}